%% file: main.tex
\documentclass[english]{article}
\usepackage[T1]{fontenc}
\usepackage[latin9]{inputenc}
\usepackage{geometry}
\usepackage{babel}
\usepackage{verbatim}
\usepackage{float}
\usepackage{mathtools}
\usepackage{amsmath}
\usepackage{amssymb}
\usepackage{graphicx}
\usepackage[unicode=true,
 bookmarks=false,
 breaklinks=false,pdfborder={0 0 1},colorlinks=false]
 {hyperref}
\hypersetup{
 colorlinks,linkcolor=red,anchorcolor=blue,citecolor=blue}

\makeatletter

\providecommand{\tabularnewline}{\\}
\floatstyle{ruled}
\newfloat{algorithm}{tbp}{loa}
\providecommand{\algorithmname}{Algorithm}
\floatname{algorithm}{\protect\algorithmname}

\usepackage{babel}
\usepackage{babel}
\usepackage{babel}

\usepackage{mathtools}

\usepackage{cite}\usepackage{amsthm}\usepackage{dsfont}\usepackage{array}\usepackage{mathrsfs}\usepackage{comment}\onecolumn
\usepackage{natbib}
\usepackage{color}\usepackage{babel}

\newcommand{\R}{\mathbb{R}}

\allowdisplaybreaks

\usepackage{enumitem}
\setlist[itemize]{leftmargin=1.5em}
\setlist[enumerate]{leftmargin=1.5em}

\usepackage{babel}
\usepackage{algorithm}% http://ctan.org/pkg/algorithms
\usepackage{algorithmic}% http://ctan.org/pkg/algorithmicx
\usepackage{arydshln}

\usepackage{cleveref}
\usepackage{bm}

\newcommand{\cP}{\mathcal{P}}

\newcommand{\RR}{\mathbb{R}}

\newcommand{\argmin}{\mathop{\mathrm{argmin}}}

\DeclareMathOperator{\ind}{\mathds{1}}  % Indicator

\newcommand{\ud}{{\,\mathrm{d}}}

\newcommand{\rd}{{\rm d}}

\newcommand{\supp}{{\rm supp}}

\numberwithin{equation}{section}

\definecolor{yly}{RGB}{0,150,0}
\definecolor{pink}{RGB}{150,0,0}

\makeatother

\begin{document}
\theoremstyle{plain} \newtheorem{lemma}{\textbf{Lemma}} \newtheorem{prop}{\textbf{Proposition}}\newtheorem{theorem}{\textbf{Theorem}}\setcounter{theorem}{0}
\newtheorem{corollary}{\textbf{Corollary}} \newtheorem{assumption}{\textbf{Assumption}}
\newtheorem{example}{\textbf{Example}} \newtheorem{definition}{\textbf{Definition}}
\newtheorem{fact}{\textbf{Fact}} \newtheorem{condition}{\textbf{Condition}}\theoremstyle{definition}

\theoremstyle{remark}\newtheorem{remark}{\textbf{Remark}}\newtheorem{claim}{\textbf{Claim}}\newtheorem{conjecture}{\textbf{Conjecture}}
\title{Learning Gaussian Mixtures Using the\\ Wasserstein-Fisher-Rao Gradient Flow}
\author{Yuling Yan\thanks{The first two authors contributed equally.} \thanks{Department of Operations Research and Financial Engineering, Princeton
University, Princeton, NJ 08544, USA; Email: \texttt{yulingy@princeton.edu}.}\and Kaizheng Wang\footnotemark[1] \thanks{Department of Industrial Engineering and Operations Research, Columbia
University, New York, NY 10027, USA; Email: \texttt{kaizheng.wang@columbia.edu}.} \and Philippe Rigollet\thanks{Department of Mathematics, MIT, Cambridge, MA 02139, USA; Email: \texttt{rigollet@math.mit.edu}.}}

\maketitle
\input{abstract.tex}


% \tableofcontents{}

\input{intro.tex}

\input{algorithm.tex}

\input{numerical.tex}\input{discussion.tex}

\section*{Acknowledgements}
The authors thank Donghao Wang for a helpful discussion.
Y.~Yan is supported in part by Charlotte Elizabeth Procter Honorific Fellowship from Princeton University. Part of this work was done during Y.~Yan's visit to MIT in Fall 2022.  K.~Wang is supported by an NSF grant DMS-2210907 and a start-up grant at Columbia University. P.~Rigollet is supported by NSF grants IIS-1838071, DMS-2022448, and CCF-2106377.

\newpage

\appendix
\input{appendix_prelim.tex}\input{appendix_gf_derivation.tex}

\input{appendix_gd_theory.tex}

\input{appendix_ode.tex} \input{appendix_wass_theory.tex}
\bibliographystyle{apalike}
\bibliography{bibfileNonconvex}

\end{document}

%% file: abstract.tex
\begin{abstract}

Gaussian mixture models form a flexible and expressive parametric family of distributions that has found applications in a wide variety of applications. Unfortunately, fitting these models to data is a notoriously hard problem from a computational perspective. Currently, only moment-based methods enjoy theoretical guarantees while likelihood-based methods are dominated by heuristics such as Expectation-Maximization that are known to fail in simple examples. In this work, we propose a new algorithm to compute the nonparametric maximum likelihood estimator (NPMLE) in a Gaussian mixture model. Our method is based on gradient descent over the space of probability measures equipped with the Wasserstein-Fisher-Rao geometry for which we establish convergence guarantees. In practice, it can be approximated using an interacting particle system where the weight and location of particles are updated alternately.  We conduct
extensive numerical experiments to confirm the effectiveness of the
proposed algorithm compared not only to classical benchmarks but also to similar gradient descent algorithms with respect to simpler geometries. In particular, these simulations illustrate the benefit of updating both weight and location of the interacting particles.
\end{abstract}

%% file: intro.tex
\section{Introduction}

Owing to their flexibility and versatility, mixture models have emerged as central objects of statistical modeling since their introduction by Pearson in the nineteenth century. However, this flexibility often comes at computational cost: such models are often hard to fit and, until quite recently, the computational aspects of mixture models have been overlooked. Still today, theory and practice diverge, with the former largely focuses on the method of moments while the latter is dominated by variational approaches, chiefly maximum likelihood. The goal of this work is to reduce this gap by developing new algorithms for maximum likelihood estimation that are supported by theoretical guarantees. 

Consider i.i.d.~samples $\{X_{i}\}_{1\leq i\leq N}\in\mathbb{R}^{d}$
generated from an isotropic Gaussian\footnote{All of this work extends to general mixtures. See the Appendix for a general treatment.} mixture $\rho^{\star}*\mathcal{N}(0,I_{d})$
with density function
\[
(\rho^{\star}*\phi) \left(x\right)=\int_{\mathbb{R}^{d}}\phi\left(x-y\right)\rho^{\star}\left(\mathrm{d}y\right),
\]
where $\rho^{\star}$ is a mixing distribution over $\mathbb{R}^{d}$,
and $\phi(x)=(2\pi)^{-d/2}\exp(-\Vert x\Vert_{2}^{2}/2)$ is the density
function of the isotropic Gaussian distribution $\mathcal{N}(0,I_{d})$.
The goal is to learn the Gaussian mixture $\rho^{\star}*\mathcal{N}(0,I_{d})$
from $N$ samples. 

The negative log-likelihood for this problem is defined as
$$
\ell_{N}\left(\rho\right)=-\frac{1}{N}\sum_{i=1}^{N}\log\left[(\rho*\phi)\left(X_{i}\right)\right]\,.
$$
While $\ell_N$ itself is trivially a convex functional of $\rho$, the class of measures over which it is minimized is often not. This is for example the case of finite mixture models where $\rho$ is restricted to be a measure supported on at most $k$ atoms. This lack of convexity in the constraint is the main source of computational difficulty for this problem. To overcome this limitation,  \citet{kiefer1956consistency} proposed the nonparametric maximum likelihood estimator (NPMLE) which prescribes to minimize $\ell_N$ over the set $\mathcal{P}(\mathbb{R}^{d})$ of all probability distributions over $\RR^d$:
\begin{equation}
	\widehat \rho\in\argmin_{\rho\in\mathcal{P}(\mathbb{R}^{d})}\ell_{N}\left(\rho\right)\,.\label{eq:NPMLE}
\end{equation}
While $\mathcal{P}(\mathbb{R}^d)$ is convex, it is infinite-dimensional. In fact, NPMLE can be seen as an extreme instance of overparameterization, a phenomenon that is currently challenging conventional statistical wisdom in deep learning theory~\cite{chizat2018global,AllLiSon19}. In fact, the solution of~\eqref{eq:NPMLE} enjoys interesting structural properties when $d=1$. More specifically, in this case, the optimization problem \eqref{eq:NPMLE} admits a unique solution $\widehat \rho$ \citep{jewell1982mixtures, lindsay1993uniqueness}
which furthermore is supported on at most $N$ atoms \citep{lindsay1983ageometry} in general. This upper bound
was improved to $O_{\mathbb{P}}(\log N)$ by \citet{polyanskiy2020self} when $\rho^{\star}$
is sub-Gaussian. Quite strikingly, little is known about
the structure of NPMLE in dimension $d \ge 2$. In fact, to the best of our knowledge, Theorem~\ref{thm:NPMLE-basics} below is the first to establish existence of a solution to~\eqref{eq:NPMLE} in any dimension.

Classical statistical results provide Hellinger risk bounds for $\widehat \rho \ast \phi$ as an estimator of $\rho^\star \ast \phi$  \citep{zhang2009generalized, dicker2016high, saha2020nonparametric}. These rates are commensurate with minimax optimality even over the larger class of $C^\infty$ density functions up to logarithmic factors.

Despite notable contributions, the computational aspects of NPMLE are still vastly under-explored. Most of these contributions employ a  discretization scheme by setting a fine grid in advance and solve \eqref{eq:NPMLE} with the additional constraint that $\rho$ is supported on the grid~\citep{lindsay1983ageometry,jiang2009general,koenker2014convex,zhang2022efficient}. While well understood theoretically, those methods suffer from the curse of dimensionality, and their complexity scales exponentially with the dimension $d$. To overcome this limitation, \citet{zhang2022efficient}
proposed a heuristic that alternately updates the weights and the support using the Expectation-Maximization (EM) algorithm but it does not come with theoretical guarantees. Another notable contribution is the support reduction algorithm of \citet{groeneboom2008support} which is also computationally inefficient since it requires to compute the minimizer of a nonconvex function in $\mathbb{R}^{d}$ in each iteration. 

In this work, we propose to solve~\eqref{eq:NPMLE} using gradient descent in the space of probability measures endowed with the Wasserstein-Fisher-Rao (WFR) geometry~\citep{chizat2018interpolating,kondratyev2016new,liero2018optimal,gallouet2017jko}. More specifically, we introduce the WFR gradient flow of the negative log-likelihood $\ell_N$ and show that it converges to the NPMLE under mild conditions.  In turn, we implement this WFR gradient flow using Euler discretization in time and particle discretization in space, thus resulting in a system of weighted interacting particles. In essence, the resulting Algorithm~\ref{alg:WFR-GD} alternatively updates locations and weights, like the EM algorithm described above but the updates coming from the WFR gradient flow are inherently different.

As the name indicates, the WFR geometry is a composite of the Wasserstein geometry~\citep{ott01,ambrosio2008gradient} and the Fisher-Rao geometry~\citep{bauer2016uniqueness}. The former component governs updates of the support while the latter governs the weight updates. Either of these geometries leads to its own gradient descent algorithm but our numerical results indicate that their combination is the key to achieving fast convergence; see Section~\ref{sec:numerical}. In fact, we show that the fixed-location EM algorithm from prior literature~\citep[see, e.g.][]{jiang2009general} implements gradient descent with respect to the Fisher-Rao geometry. As a byproduct of our analysis, we also show that it converges to the NPMLE in the infinite-particle regime under certain conditions.

\section{Nonparametric Maximum Likelihood Estimator (NPMLE)}

In this section, we examine optimality conditions for the optimization problem~\eqref{eq:NPMLE}. To that end, denote by  $\delta\ell_{N}(\rho)$ the first variation of  $\ell_{N}$ at a measure $\rho$ and observe\footnote{Explicit calculations follow from standard arguments in calculus of variations and are deferred to Appendix~\ref{subsec:First-variation}.} that it is given by
\begin{equation}
	\delta\ell_{N}\left(\rho\right):x\mapsto-\frac{1}{N}\sum_{i=1}^{N}\frac{\phi\left(x-X_{i}\right)}{(\rho*\phi)\left(X_{i}\right)}.\label{eq:first-variation}
\end{equation}

The following theorem shows the existence as well as the optimality condition
of NPMLE in general dimension. The proof is deferred to Appendix \ref{sec:proof-thm-NPMLE-basics}.

\begin{theorem}\label{thm:NPMLE-basics} The following properties
	hold for NPMLE:
	\begin{enumerate}
		\item (Existence) The minimizer of the optimization problem \eqref{eq:NPMLE}
		exists.
		\item (Optimality condition) A distribution $\widehat{\rho}\in\mathcal{P}(\mathbb{R}^{d})$
		is an NPMLE if and only if (i) $\delta\ell_{N}(\widehat{\rho})(x)\geq-1$
		holds for all $x\in\mathbb{R}^{d}$, and (ii) $\delta\ell_{N}(\widehat{\rho})(x)=-1$
		for $\widehat{\rho}$-a.e.~$x$.
	\end{enumerate}
\end{theorem}

\begin{remark}We show in Appendix \ref{sec:proof-thm-NPMLE-basics}
	that for any $\rho\in\mathcal{P}(\mathbb{R}^{d})$, $\int\delta\ell_{N}(\rho)\mathrm{d}\rho=-1$
	always holds. As a result, the optimality condition (ii) in Theorem
	\ref{thm:NPMLE-basics} is implied by (i), which means that $\delta\ell_{N}(\widehat{\rho})(x)\geq-1$
	for all $x\in\mathbb{R}$ alone is already the necessary and sufficient
	condition for $\widehat{\rho}$ to be the NPMLE. However, we keep
	both conditions in the theorem as each of them reveal important structural
	information of NPMLE.
	
\end{remark}

Theorem \ref{thm:NPMLE-basics} asserts the existence of NPMLE, but
its uniqueness  when $d\ge 2$ is still an open problem. Although
the uniqueness is not settled, the convergence theory in this paper
is still valid: in this case NPMLE refers to any minimizer of  \eqref{eq:NPMLE}.

\paragraph{Notation.}
We use $\mathcal{P}(\mathbb{R}^{d})$ to denote the space of probability
measures over $\mathbb{R}^{d}$, $\mathcal{P}_2(\mathbb{R}^{d})$ to denote the space of probability
measures over $\mathbb{R}^{d}$ with finite second moments, and $\mathcal{P}_{\mathsf{ac}}(\mathbb{R}^{d})$
to denote the space of probability measures that are absolutely continuous
with respect to the Lebesgue measure on $\mathbb{R}^{d}$. Let $\Delta^{m-1}$ be the $m-1$ dimensional probability simplex.
For any $\rho \in \cP (\RR^d)$, $\mathsf{supp} (\rho)$ denotes its support set, i.e. the smallest closed set $C \subseteq \RR^d$ such that $\rho(C) = 1$. For any mapping $T:\mathbb{R}^{d}\to\mathbb{R}^{d}$ and any distribution
$\rho\in\mathcal{P}(\mathbb{R}^{d})$, let $T_{\#}\rho$ be the pushforward (or image measure) of $\rho$ by $T$, which is defined by  $T_{\#}\rho(A)=\rho(T^{-1}(A))$
for any Borel set $A$ in $\mathbb{R}^{d}$. For any $x\in\mathbb{R}^{d}$,
we use $\delta_{x}$ to denote the Dirac mass at point $x$. For two
probability measures $\mu,\nu\in\mathcal{P}(\mathbb{R}^{d})$, we
use $\mu\ll\nu$ to denote that $\mu$ is absolutely continuous with
respect to $\nu$. For a sequence $\{ \rho_n \}_{n=0}^{\infty}$ in $\cP (\RR^d)$ and $\rho \in \cP (\RR^d)$, we write $\rho_n \overset{\mathrm{w}}{\to} \rho$ if $\rho_n$ weakly converges to $\rho$, i.e.~$\int_{\RR^d} f(x) \rho_n (\rd x) \to \int_{\RR^d} f(x) \rho (\rd x)$ holds for every bounded continuous function $f:~\RR^d \to \RR$.
Let $C_{\mathrm{c}}^{\infty}(\mathbb{R}^{d})$
be the set of smooth functions with compact support in $\mathbb{R}^{d}$.
We say that $(\rho_{t})_{t\geq0}$ is a distributional solution to
the partial differential equation (PDE) $\partial_{t}\rho_{t}=-\mathsf{div}(\rho_{t}v_{t})+\rho_{t}\alpha_{t}$
where $v_{t}:\mathbb{R}^{d}\to\mathbb{R}^{d}$ and $\alpha_{t}:\mathbb{R}^{d}\to\mathbb{R}$,
if for any $\varphi\in C_{\mathrm{c}}^{\infty}(\mathbb{R}^{d})$ it
holds that
\[
\frac{\mathrm{d}}{\mathrm{d}t}\int_{\mathbb{R}^{d}}\varphi\left(x\right)\rho_{t}\left(\mathrm{d}x\right)=\int_{\mathbb{R}^{d}}
\left[
\left\langle \nabla\varphi\left(x\right),v_{t}\left(x\right)\right\rangle
+  \varphi(x) \alpha_t (x) \right]
\rho_{t}\left(\mathrm{d}x\right).
\]
Finally we use the shorthand ODE to refer to ordinary differential equations.

\begin{comment}
Maximum-likelihood estimation?

To estimate $\{\mu_{j}^{\star}\}_{j=1}^{K}$, a natural idea is to
fit a MLE
\[
\text{maximize}\quad\ell\left(\mu_{1},\ldots,\mu_{K}\right)\coloneqq\frac{1}{n}\sum_{i=1}^{n}\log\left\{ \frac{1}{K}\sum_{j=1}^{K}\frac{1}{\sqrt{2\pi}}\exp\left[-\frac{1}{2}\left(x_{i}-\mu_{j}\right)^{2}\right]\right\} .
\]
However, one might suffer from the following two issues.
\begin{itemize}
\item In practice we usually don't know $K$ a priori;
\item Even if we know $K$ in advance, it might not be an ideal choice to
fit a standard MLE using either EM algorithm or gradient descent.
\cite{jin2016local} constructs a pessimistic example showing that
when $K=3$, we can appropriately construct $\{\mu_{j}^{\star}\}_{j=1}^{K}$
such that the population log-likelihood function has bad local minima,
and any first-order method with random initialization (sampling the
parameters $\{\mu_{j}\}_{j=1}^{K}$ uniformly from $\{x_{i}\}_{i=1}^{n}$)
might converge to one of these local minima with non-vanishing probability.
\end{itemize}
Towards solving these issues, we find the overparameterization scheme
has good empirical performance.
\end{comment}

%% file: algorithm.tex
\section{Wasserstein-Fisher-Rao gradient descent}

Gradient flows over probability measures are a useful tool in the development of sampling algorithms where the goal is to produce samples from a target measure. A classical example arises when $\pi$ is, for example, a Bayesian posterior known only up to normalizing constant. In this context, Wasserstein-Fisher-Rao (WFR) gradient flows have recently emerged as a strong alternative to vanilla Wasserstein gradient flows. Indeed, they provide the backbone of the birth-death sampling algorithm of~\citet{LuLuNol19} as well as the particle-based method proposed in~\citet{LuSleWan22}. 
% In both cases however, the functional to be minimized is the Kullback-Leibler divergence $\mathsf{KL}(\cdot \, \|\, \pi)$.
% This question is somewhat different from the likelihood maximization problem at stake here and 
We refer the reader to the recent manuscript of~\cite{Che22} for more details on sampling and the role of Wasserstein gradient flows in this context. 

The likelihood maximization problem of interest in the present paper differs from sampling questions because its aims at optimizing a different objective. Nevertheless, it remains an optimization problem and the machinery of gradient flows over the space of probability measures may be deployed in this context. To the best of our knowledge, this paper present the first attempt at such a deployment. More specifically,
our main algorithm to solve~\eqref{eq:NPMLE}  relies on a specific discretization of the Wasserstein-Fisher-Rao gradient flow. We begin with a short introduction to gradient flows over metric spaces of probability measures that can be safely skipped by experts.

\subsection{Gradient flows over metric spaces of probability measures}

Gradient flows over metric spaces of probability measures is a central topic of the calculus of variations that has found applications in variety of fields ranging from analysis and geometry to probability and statistics. We briefly discuss the main idea behind this powerful tool and refer the reader to the formidable book of~\cite{ambrosio2008gradient} for more details about this deep question.

Recall that our goal is to derive a gradient flow for the functional $\ell_N$ over the space of probability measures. The nature of this gradient flow is simply a curve $(\rho_t)_{t\ge 0}$ such that $\partial_t \rho_t = - \nabla \ell_N(\rho_t)$ for a notion of gradient $\nabla$ to be defined. In their seminal work,~\cite{JorKinOtt98} were able to define a gradient flow over the Wasserstein space by analogy to the Euclidean gradient flow without the need to actually define a gradient. We follow their approach and define the gradient flow of $\ell_N$ with respect to a suitable geometry with geodesic distance  $d(\cdot, \cdot)$ over the space of probability measures as
$$
\partial_t \rho_t = \lim_{\eta \to 0} \frac{\rho_t^\eta-\rho_t}{\eta}\,,
$$
where 
$$
\rho_{t}^{\eta}\coloneqq\underset{\rho\in\mathcal{P}(\mathbb{R}^{d})}{\arg\min}\left\{ \int_{\mathbb{R}^{d}}\delta\ell_{N}\left(\rho_{t}\right)\mathrm{d}\left(\rho-\rho_{t}\right)+\frac{1}{2\eta}{\sf d}^{2}\left(\rho,\rho_{t}\right)\right\}
$$
Given a distance, the existence of a limiting absolutely continuous curve $(\rho_t)_{t\geq 0}$ is an important and central question that we omit in this overview.

Our main focus in this work is the Wasserstein-Fisher-Rao distance which is a composite of the Fisher-Rao distance and the (quadratic) Wasserstein distance. We now introduce these three distances. Recall that a geodesic distance between two points measures the shortest curve that links these two points. The difference between these three distances is governed by the differential structure put on probability distributions, which roughly corresponds to type of curves that are allowed. The length of these curves is then measured using a Riemannian metric which in all cases is rather straightforward so we focus our discussion on curves. In turn, these curves and their lengths define a geometry on the space of probability measures.

\paragraph{Fisher-Rao distance.}
The Fisher-Rao distance is linked to reaction equations of the form
\begin{equation}
	\label{eq:react}
	\partial_{t}\rho_{t}=\rho_{t}\Big(\alpha_{t}-\int\alpha_{t}\mathrm{d}\rho_{t}\Big)\,,
\end{equation}
where $\alpha_t(x) \in \R$ is a scalar  that governs how much mass is created at $x \in \RR^d$ and time $t$. It is easy to see that these dynamics preserve the total mass 1 of probability distributions. Among all such curves  that link $\rho_0$ to $\rho_1$, the Fisher-Rao geodesic is the one that minimizes the total length. More specifically, the Fisher-Rao distance $d_{\sf FR}$ is defined as~\citep{bauer2016uniqueness},
\begin{align}
	d_{\mathsf{FR}}^{2}\left(\rho_{0},\rho_{1}\right) & =\inf\bigg\{\int_{0}^{1}\int\Big[\Big(\alpha_{t}-\int\alpha_{t}\mathrm{d}\rho_{t}\Big)^{2}\Big]\mathrm{d}\rho_{t}\mathrm{d}t:~
	\left(\rho_{t},\alpha_{t}\right)_{t\in[0,1]}\,\text{solves}\nonumber \\
	& \qquad\qquad\qquad\partial_{t}\rho_{t}=\rho_{t}\Big(\alpha_{t}-\int\alpha_{t}\mathrm{d}\rho_{t}\Big)\bigg\}.\label{eq:FR-metric}
\end{align}
While this will not be useful for our problem, it is worth noting that
\[
d_{\mathsf{FR}}^{2}\left(\rho_{0},\rho_{1}\right)=4\int\Big|\sqrt{\frac{\ud\rho_{0}}{\ud \lambda}}-\sqrt{\frac{\ud \rho_{1}}{\ud \lambda}}\Big|^{2}\mathrm{d}\lambda,
\]
where $\lambda$ is any positive measure such that $\rho_0, \rho_1 \ll \lambda$ (e.g., $\lambda=\rho_0+\rho_1$.). Hence, up to a constant factor, the Fisher-Rao distance is simply the Hellinger distance that is well-known to statisticians.

\paragraph{Wasserstein distance.}

Given two probability measures $\rho_{0},\rho_{1}\in\mathcal{P}(\mathbb{R}^{d})$,
the (quadratic) Wasserstein distance between $\rho_{0}$ and $\rho_{1}$
is defined as \citep{villani2009optimal}
\[
d_{\mathsf{W}}^{2}\left(\rho_{0},\rho_{1}\right)=\inf_{\pi\in\Pi(\rho_{0},\rho_{1})}\int\left\Vert x-y\right\Vert _{2}^{2}\pi\left(\mathrm{d}x,\mathrm{d}y\right),
\]
where the infimum is taken over all couplings of $\rho_{0}$ and
$\rho_{1}$. 

It also admits a geodesic distance interpretation by means of the Benamou-Brenier formula:
\begin{align}
	d_{\mathsf{W}}^{2}\left(\rho_{0},\rho_{1}\right) & =\inf\bigg\{\int_{0}^{1}\int\left\Vert v_{t}\right\Vert ^{2}\mathrm{d}\rho_{t}\mathrm{d}t:\left(\rho_{t},v_{t}\right)_{t\in[0,1]}\;\text{solves}\;\partial_{t}\rho_{t}=-\mathsf{div}\left(\rho_{t}v_{t}\right)\bigg\}.\label{eq:Wass-metric}
\end{align}
Here admissible curves are given by the continuity equation 
\begin{equation}
	\label{eq:CE}
	\partial_{t}\rho_{t}=-\mathsf{div}\left(\rho_{t}v_{t}\right)\,,
\end{equation}
which describes the evolution of a density of particles in $\R^d$ evolving according to time-dependent vector field $v_t :\R^d \to \R^d$.

\paragraph{Wasserstein-Fisher-Rao distance.}

The geometry underlying the Wasserstein-Fisher-Rao distance is built using  curves that satisfy the following evolution equation:
$$
\partial_{t}\rho_{t}=-\mathsf{div}\left(\rho_{t}v_{t}\right)+\rho_{t}\Big(\alpha_{t}-\int\alpha_{t}\mathrm{d}\rho_{t}\Big)\,.
$$
Note that the right-hand side is precisely the sum of the right-hand sides for the reaction equation~\eqref{eq:react} and the continuity equation~\eqref{eq:CE} that govern the Fisher-Rao and the Wasserstein geometry respectively. As such it is a composite of the two geometries.

In turn, the Wasserstein-Fisher-Rao distance $d_{\mathsf{WFR}}$ is defined as~\citep{chizat2018interpolating,kondratyev2016new,liero2018optimal,gallouet2017jko}
\begin{align}
	d_{\mathsf{WFR}}^{2}\left(\rho_{0},\rho_{1}\right) & =\inf\bigg\{\int_{0}^{1}\int\Big[\left\Vert v_{t}\right\Vert ^{2}+\Big(\alpha_{t}-\int\alpha_{t}\mathrm{d}\rho_{t}\Big)^{2}\Big]\mathrm{d}\rho_{t}\mathrm{d}t:\left(\rho_{t},v_{t},\alpha_{t}\right)_{t\in[0,1]}\,\text{solves}\nonumber \\
	& \qquad\qquad\qquad\partial_{t}\rho_{t}=-\mathsf{div}\left(\rho_{t}v_{t}\right)+\rho_{t}\Big(\alpha_{t}-\int\alpha_{t}\mathrm{d}\rho_{t}\Big)\bigg\}.\label{eq:WFR-metric}
\end{align}
Equipped with this distance, we are in a position to define the WFR gradient flow and its time discretization, the WFR gradient descent.

\subsection{Wasserstein-Fisher-Rao gradient descent}

In this section we introduce our main algorithm: Wasserstein-Fisher-Rao Gradient Descent.

The gradient flow $\{\rho_{t}\}_{t\geq0}$ of the negative log-likelihood $\ell_{N}(\rho)$
in $\mathcal{P}_2(\mathbb{R}^{d})$ with respect to the Wasserstein-Fisher-Rao
distance $d_{\mathsf{WFR}}(\cdot,\cdot)$ is given by
\begin{equation}
	\partial_{t}\rho_{t}=-\left[1+\delta\ell_{N}\left(\rho_{t}\right)\right]\rho_{t}+ \mathsf{div}\left(\rho_{t}\nabla\delta\ell_{N}\left(\rho_{t}\right)\right).\label{eq:WFR-GF}
\end{equation}
The formal derivation of \eqref{eq:WFR-GF} is  based on the Riemannian
structure underlying the Wasserstein-Fisher-Rao metric can be found
in Appendix \ref{subsec:WFR-derivation}. 

The WFR gradient flow is not readily implementable for because of two obstacles. First, it is described in continuous time. Second, it requires the manipulation of full probability measures $\rho_t$ on $\R^d$ which are infinite dimensional objects.

To overcome the first obstacle, we employ a straightforward time-discretization scheme to obtain a WFR gradient descent. This algorithm produces a sequence of probability measures $\{\rho_{n}\}_{n\geq0}$.  It makes the following two steps alternately:
\begin{subequations}\label{eq:WFR-GD}
	\begin{align}
		& \frac{\mathrm{d}\widetilde\rho_{n}}{\mathrm{d}\rho_{n}}=1-\eta \left[1+\delta\ell_{N}\left(\rho_{n}\right)\right]\qquad\qquad\text{(Fisher-Rao gradient update)}\label{eq:WFR-GD-FR}\\
		& \rho_{n+1}=\left[\mathrm{id}-\eta\nabla\delta\ell_{N}\left(\widetilde\rho_{n}\right)\right]_{\#}\widetilde\rho_{n}\qquad\:\:\:\text{(Wasserstein gradient update)}\label{eq:WFR-GD-W}
	\end{align}
\end{subequations}for $n=0,1,\ldots$, where $\eta>0$ is
the step size. 
Note that each corresponds to a summand on the right-hand side of~\eqref{eq:WFR-GF}
In fact,  \eqref{eq:WFR-GD-FR} is one step of gradient descent update
with respect to the Fisher-Rao geometry, and  \eqref{eq:WFR-GD-W} corresponds to a gradient step with respect to the Wasserstein geometry. It
is also worth mentioning that  \eqref{eq:WFR-GD} is related to the
splitting scheme in \citet{gallouet2017jko}:  \eqref{eq:WFR-GD} can
be viewed as forward Euler scheme (explicit scheme) for numerical
approximation of  \eqref{eq:WFR-GF}, while \citet{gallouet2017jko}
uses backward Euler scheme (implicit scheme). The implementation of the latter requires optimization over probability measures in each iterate, which is a difficult problem.
% In a nutshell, the benefit of using Wasserstein-Fisher-Rao geometry (compared to Fisher-Rao)
% mainly lies in the Wasserstein gradient descent step \eqref{eq:WFR-GD-W}, which is very
% crucial especially when initialized from some discrete distribution
% $\rho_{0}$ with a finite number of particles; we will discuss the
% advantage of Wasserstein-Fisher-Rao gradient descent/flow in detail
% momentarily. 

The following theorem shows that Wasserstein-Fisher-Rao gradient descent converges when initialized from a distribution that puts weight on the entire space. The proof can be found in \Cref{sec-convergence-proofs}. In fact, a similar result  for the continuous-time Wasserstein-Fisher-Rao gradient flow can be derived at the cost of additional technical considerations.

\begin{theorem}[Convergence to NPMLE] \label{thm:WFR-convergence}
	Suppose that the initialization $\rho_{0}\in\mathcal{P} (\mathbb{R}^d)$ satisfies $\supp(\rho_0) = \RR^d$. Consider the Wasserstein-Fisher-Rao gradient descent $\{\rho_{n}\}_{n\geq0}$
	defined in  \eqref{eq:WFR-GD}. 
	There exists $\eta_{0} > 0$ determined by the samples $\{X_{i}\}_{1\leq i\leq N}$, such that if $0<\eta\leq\eta_{0}$ and $\rho_{n}\overset{\mathrm{w}}{\to}\widehat{\rho}$ when $n\to\infty$, then $\widehat{\rho}$ is the NPMLE.
\end{theorem}
The convergence result in Theorem \ref{thm:WFR-convergence} is conditional:
we only show that if WFR gradient descent converges weakly
to some limiting distribution $\widehat{\rho}$, then $\widehat{\rho}$
is an NPMLE. This is similar to the convergence theory established by \citet{chizat2018global} in their study of the training dynamics for shallow neural networks. This limitation is due to a lack of geodesic convexity which prevents us from establishing unconditional global convergence guarantees.

To overcome the second obstacle, we also need to discretize $\rho_n$ in space. Thanks to the Wasserstein update~\eqref{eq:WFR-GD-W}, we need not use a fixed grid as in previous algorithms for NPMLE. Instead, observe that if we initialize WFR gradient descent at a distribution supported on $m$ atoms, then $\rho_{n}$ remains supported on $m$ atoms for all $n$. In this case, $\rho_n$ describes the evolution of $m$ interacting particles with weights.

More concretely, consider the initialization
\[
\rho_{0}=\sum_{l=1}^{m}\omega_{0}^{(l)}\delta_{\mu_{0}^{(l)}},\quad\text{where }\mu_{0}^{(1)},\ldots,\mu_{0}^{(m)}\in\mathbb{R}^{d}\text{ and }\omega_{0}=(\omega_{0}^{(1)},\ldots,\omega_{0}^{(m)})\in\Delta^{m-1},
\]
where the location of particles are independently sampled from the
data points $\{X_{i}\}_{1\leq i\leq N}$.  The following theorem
gives a precise characterization of the Wasserstein-Fisher-Rao gradient
flow initialized from $\rho_{0}$. The proof is deferred to Appendix
\ref{sec:proof-WFR-ODE}.

\begin{theorem}[Particle Wasserstein-Fisher-Rao gradient flow]\label{thm:WFR-ODE}The system of coupled
	ODE \begin{subequations}\label{eq:WFR-ODE}
		\begin{align}
			\dot{\mu}_{t}^{(j)} & =\frac{1}{N}\sum_{i=1}^{N}\frac{\phi\left(X_{i}-\mu_{t}^{(j)}\right)}{\sum_{l=1}^{m}\omega_{t}^{(j)}\phi\left(X_{i}-\mu_{t}^{(l)}\right)}\left(X_{i}-\mu_{t}^{(j)}\right),\\
			\dot{\omega}_{t}^{(j)} & =\left[\frac{1}{N}\sum_{i=1}^{N}\frac{\phi\left(X_{i}-\mu_{t}^{(j)}\right)}{\sum_{l=1}^{m}\omega_{t}^{(j)}\phi\left(X_{i}-\mu_{t}^{(l)}\right)}-1\right]\omega_{t}^{(j)},
		\end{align}
	\end{subequations}with initialization $\mu_{0}^{(1)},\ldots\mu_{0}^{(m)}\overset{\text{i.i.d.}}{\sim}\mathsf{Uniform}(\{X_{i}\}_{1\leq i\leq N})$
	and $\omega_{0}=[\omega_{0}^{(j)}]_{1\leq j\leq m}\in\Delta^{m-1}$
	has unique solution on any time interval $[0,T]$. Moreover, the flow
	$(\rho_{t})_{t\geq0}$ defined as
	\begin{equation}
		\rho_{t}\coloneqq\sum_{l=1}^{m}\omega_{t}^{(l)}\delta_{\mu_{t}^{(l)}}\label{eq:WFT-discrete}
	\end{equation}
	is the Wasserstein-Fisher-Rao gradient flow, i.e.~%\eqref{eq:WFT-discrete} is 
	a distributional solution to the PDE  \eqref{eq:WFR-GF}.
	
\end{theorem}

In practice, we can obtain a time discretization of the Wasserstein-Fisher-Rao
gradient flow by discretizing the ODE system  \eqref{eq:WFR-ODE},
which gives the Wasserstein-Fisher-Rao gradient descent algorithm
as summarized in Algorithm \ref{alg:WFR-GD}.

\begin{algorithm}[h]
	\caption{Wasserstein-Fisher-Rao gradient descent.}
	
	\label{alg:WFR-GD}\begin{algorithmic}
		
		\STATE \textbf{{Input}}: data $\{X_{i}\}_{1\leq i\leq n}$, number
		of particles $m$, step size $\eta > 0$, maximum number of iterations
		$t_{0}$.
		
		\STATE \textbf{{Initialization}}: draw $\mu_{0}^{(1)},\ldots,\mu_{0}^{(m)}\overset{\text{i.i.d.}}{\sim}\mathsf{Unif}(\{X_{i}\}_{1\leq i\leq n})$
		and $\omega_{0}^{(1)}=\cdots=\omega_{0}^{(m)}=1/m$.
		
		\vspace{0.2em}
		
		\STATE \textbf{{Updates}}: \textbf{for }$t=0,1,\ldots,t_{0}$ \textbf{do}
		
		\STATE \vspace{-1em}
		\begin{subequations} 
			\begin{align*}
				\mu_{t+1}^{(j)} & =\mu_{t}^{(j)}+\eta\frac{1}{N}\sum_{i=1}^{N}\frac{\phi\left(X_{i}-\mu_{t}^{(j)}\right)}{\sum_{l=1}^{m}\omega_{t}^{(j)}\phi\left(X_{i}-\mu_{t}^{(l)}\right)}\left(X_{i}-\mu_{t}^{(j)}\right),\\
				\omega_{t+1}^{(j)} & =\omega_{t}^{(j)}+\eta \left[\frac{1}{N}\sum_{i=1}^{N}\frac{\phi\left(X_{i}-\mu_{t+1}^{(j)}\right)}{\sum_{l=1}^{m}\omega_{t}^{(j)}\phi\left(X_{i}-\mu_{t+1}^{(l)}\right)}-1\right]\omega_{t}^{(j)},
			\end{align*}
		\end{subequations} for all $j=1,\ldots,m$. Here $\phi(x)=(2\pi)^{-d/2}\exp(-\Vert x\Vert_{2}^{2}/2)$
		is the probability density function of $\mathcal{N}(0,I_{d})$.
		
		\STATE \textbf{{Output}} $\rho=\sum_{j=1}^{m}\omega_{t_{0}}^{(j)}\delta_{\mu_{t_{0}}^{(j)}}$
		as the (approximate) NMPLE. 
		
	\end{algorithmic}
\end{algorithm}

% We can see that Wasserstein-Fisher-Rao gradient descent/flow algorithm
% updates both the location and the weight of each particle in each
% iteration; in comparison, the Fisher-Rao gradient descent algorithm
% (cf.~Algorithm \ref{alg:FR-GD}) only updates the weight of each
% particle. The additional location updates can be viewed as gradient
% descent/flow updates under the Wasserstein geometry, which decreases
% the loss function $\ell_{N}$ when the step size is sufficiently small.

% Recall that it is known (in one dimension) and conjectured (in higher dimension)
% that NPMLE is supported on a small number of atoms \citep{polyanskiy2020self}. As a result, the capability of moving
% the location of the particles gives Wasserstein-Fisher-Rao gradient
% descent/flow algorithm the potential to converge to NPMLE even with
% a finite (even a few) number of particles. The numerical experiments
% in Section \ref{sec:numerical} confirm the advantage of Algorithm
% \ref{alg:WFR-GD}.

The proof techniques employed to establish Theorem~\ref{thm:WFR-convergence} do not cover Algorithm~\ref{alg:WFR-GD} unfortunately since the initial measure is not supported on the whole space. Nevertheless, since the NPMLE is known to be supported on a small number of atoms~\citep{polyanskiy2020self} in certain cases, it is likely that taking $m$ large enough will be sufficient to establish convergence results. This conjecture is supported by our numerical experiments in  Section~\ref{sec:numerical}.

\subsection{Surrogate geometries}

As discussed above the Wasserstein-Fisher-Rao (WFR) geometry is obtained as a composite of the Wasserstein geometry and the Fisher-Rao geometry. In turn, WFR gradient descent alternates between a Fisher-Rao gradient step and a Wasserstein gradient step. This observation raises a burning question: is the composite nature of WFR gradient descent necessary to obtain good performance of are either of these two building blocks alone sufficient? In this subsection, we explore properties and limitations of these two natural alternatives.

\paragraph{Fisher-Rao gradient descent.}
We show in Appendix~\ref{subsec:FR-derivation} that the Fisher-Rao
gradient flow $(\rho_{t})_{t\geq0}$ of the function $\ell_{N}(\rho)$
is defined by the following PDE:
\begin{equation}
	\partial_{t}\rho_{t}=-\left[1+\delta\ell_{N}\left(\rho_{t}\right)\right]\rho_{t}.\label{eq:Fisher-Rao-GF}
\end{equation}

By time discretization, one readily obtains the Fisher-Rao gradient descent updates $\{\rho_{n}\}_{n\geq0}$:
\begin{equation}
	\frac{\mathrm{d}\rho_{n+1}}{\mathrm{d}\rho_{n}}=1-\gamma\left[1+\delta\ell_{N}\left(\rho_{n}\right)\right],\qquad n=0,1,\ldots,\label{eq:Fisher-Rao-GD}
\end{equation}
where $\gamma>0$ is the step size. Although Fisher-Rao gradient flow/descent
is derived in an abstract way using Riemannian geometry, it has intimate
connection with some well-known algorithms.
\begin{enumerate}
	\item {\it Fisher-Rao gradient flow as proximal gradient flow}. The Fisher-Rao
	gradient flow \eqref{eq:Fisher-Rao-GF} can be viewed as the continuous-time
	limit of the proximal gradient algorithm under the Fisher-Rao metric:
	\begin{equation}
		\partial_{t}\rho_{t}=\lim_{\eta\to0+}\frac{\rho_{t}^{\eta}-\rho_{t}}{\eta},\quad\text{where}\quad\rho_{t}^{\eta}\coloneqq\underset{\rho\in\mathcal{P}(\mathbb{R}^{d})}{\arg\min}\left\{ \int_{\mathbb{R}^{d}}\delta\ell_{N}\left(\rho_{t}\right)\mathrm{d}\left(\rho-\rho_{t}\right)+\frac{1}{2\eta}d_{\mathsf{FR}}^{2}\left(\rho,\rho_{t}\right)\right\} .\label{eq:FR-proximal}
	\end{equation}
	\item {\it Fisher-Rao gradient flow as mirror flow}. The Fisher-Rao gradient flow \eqref{eq:Fisher-Rao-GF} can also be viewed as
	the continuous-time limit of mirror descent algorithm for $\ell_{N}(\rho)$
	with Kullback-Leibler (KL) divergence
	\begin{equation}
		\partial_{t}\rho_{t}=\lim_{\eta\to0+}\frac{\rho_{t}^{\eta}-\rho_{t}}{\eta},\quad\text{where}\quad\rho_{t}^{\eta}\coloneqq\underset{\rho\ll\rho_{t}}{\arg\min}\left\{ \int_{\mathbb{R}^{d}}\delta\ell_{N}\left(\rho_{t}\right)\mathrm{d}\left(\rho-\rho_{t}\right)+\frac{1}{\eta}\mathsf{KL}\left(\rho\,\Vert\,\rho_{t}\right)\right\} .\label{eq:FR-mirror}
	\end{equation}
	\item {\it Fisher-Rao gradient descent as fixed-location EM algorithm}.
	When $\rho_{0}$ is discrete, Fisher-Rao gradient descent \eqref{eq:Fisher-Rao-GD}
	with step size $\gamma=1$ coincides with the fixed-location EM algorithm
	for Gaussian mixture model in e.g.~\citet{jiang2009general}.
\end{enumerate}
Formal justifications of the above three connections can be found
in Appendix \ref{subsec:FR-equivalence}. The  theorem below shows
that when $\rho_{0}$ is diffused,
the Fisher-Rao gradient descent enjoys appealing convergence
property. The proof is deferred to \Cref{sec-convergence-proofs}.

\begin{theorem}[Convergence to NPMLE] \label{thm:FR-convergence}
	Suppose that the initialization $\rho_{0}\in\mathcal{P} (\mathbb{R}^{d})$ satisfies $\supp (\rho_0) = \RR^d$. Consider the Fisher-Rao gradient descent $\{\rho_{n}\}_{n\geq0}$ defined
	in \eqref{eq:Fisher-Rao-GD}. 
	There exists $\eta_{0}$ determined by the samples $\{X_{i}\}_{1\leq i\leq N}$, such that if $0<\eta\leq\eta_{0}$ and $\rho_{n}\overset{\mathrm{w}}{\to}\widehat{\rho}$ when $n\to\infty$, then $\widehat{\rho}$ is the NPMLE.
\end{theorem}

We can also show a similar result for the continuous-time Fisher-Rao gradient flow  $(\rho_{t})_{t\geq0}$ (i.e.~a distributional solution to the PDE \eqref{eq:Fisher-Rao-GF}): if $\rho_{t}\overset{\mathrm{w}}{\to}\widehat{\rho}$ as $t\to\infty$, then $\widehat{\rho}$ is an NPMLE.

Theorem \ref{thm:FR-convergence} provides convergence guarantees
for Fisher-Rao gradient flow/descent when initialized from a well-spread distribution. In practice, however, we
can only initialize from a discrete distribution with $m$ particles
\[
\rho_{0}=\sum_{l=1}^{m}\omega_{0}^{(l)}\delta_{\mu^{(l)}},\text{ where }\mu^{(1)},\ldots,\mu^{(m)}\in\mathbb{R}^{d}\text{ and }\omega_{0}=(\omega_{0}^{(1)},\ldots,\omega_{0}^{(m)})\in\Delta^{m-1}.
\]
The following theorem characterizes the Fisher-Rao gradient flow when
initialized from a discrete distribution $\rho_{0}$ with $m$ mass
points. 

\begin{theorem}[Particle Fisher-Rao gradient flow]\label{thm:FR-ODE}The
	ODE system 
	\begin{equation}
		\dot{\omega}_{t}^{(j)}=-\omega_{t}^{(j)}\left[1-\frac{1}{N}\sum_{i=1}^{N}\frac{\phi\left(X_{i}-\mu^{(j)}\right)}{\sum_{l=1}^{m}\omega_{t}^{(l)}\phi\left(X_{i}-\mu^{(l)}\right)}\right],\qquad1\leq j\leq m\label{eq:FR-ODE}
	\end{equation}
	with initialization $\omega_{0}=[\omega_{0}^{(j)}]_{1\leq j\leq m}\in\Delta^{m-1}$
	has unique solution on any time interval $[0,T]$. Moreover, the flow
	$(\rho_{t})_{t\geq0}$ defined as
	\begin{equation}
		\rho_{t}\coloneqq\sum_{l=1}^{m}\omega_{t}^{(l)}\delta_{\mu^{(l)}}\label{eq:FR-discrete-solution}
	\end{equation}
	is the Fisher-Rao gradient flow, i.e.~\eqref{eq:FR-discrete-solution}
	is a distributional solution to the PDE \eqref{eq:Fisher-Rao-GF}.
	
\end{theorem}

The proof is deferred to Appendix \ref{sec:proof-FR-ODE}. For the purpose of implementation, we can further discretize the gradient flow \eqref{eq:FR-discrete-solution} with respect to time 
and obtain the Fisher-Rao gradient descent; see Algorithm~\ref{alg:FR-GD}. 

\begin{algorithm}[h]
	\caption{Fisher-Rao gradient descent.}
	
	\label{alg:FR-GD}\begin{algorithmic}
		
		\STATE \textbf{{Input}}: data $\{X_{i}\}_{1\leq i\leq n}$, number
		of particles $m$, step sizes $\eta>0$, maximum number of iterations
		$t_{0}$.
		
		\STATE \textbf{{Initialization}}: draw $\mu^{(1)},\ldots,\mu^{(m)}\overset{\text{i.i.d.}}{\sim}\mathsf{Uniform}(\{X_{i}\}_{1\leq i\leq n})$
		and $\omega_{0}^{(1)}=\cdots=\omega_{0}^{(m)}=1/m$.
		
		\vspace{0.2em}
		
		\STATE \textbf{{Updates}}: \textbf{for }$t=0,1,\ldots,t_{0}$ \textbf{do}
		
		\STATE \vspace{-1em}
		\begin{subequations} 
			\begin{align*}
				\omega_{t+1}^{(j)} & =\omega_{t}^{(j)}+\eta\left[\frac{1}{N}\sum_{i=1}^{N}\frac{\phi\left(X_{i}-\mu^{(j)}\right)}{\sum_{l=1}^{m}\omega_{t}^{(j)}\phi\left(X_{i}-\mu^{(l)}\right)}-1\right]\omega_{t}^{(j)},
			\end{align*}
		\end{subequations} for all $j=1,\ldots,m$. Here $\phi(x)=(2\pi)^{-d/2}\exp(-\Vert x\Vert_{2}^{2}/2)$
		is the probability density function of $\mathcal{N}(0,I_{d})$.
		
		\STATE \textbf{{Output}} $\rho=\sum_{j=1}^{m}\omega_{t_{0}}^{(j)}\delta_{\mu^{(j)}}$
		as the (approximate) NMPLE.
		
	\end{algorithmic}
\end{algorithm}

A quick inspection of the pseudo-cpde presented in Algorithm~\ref{alg:FR-GD} reveals a fatal flaw: the locations
of the particles are fixed at their initialization and only the weights of the particles
are updated. This introduces a systematic approximation error that may scale exponentially with the dimension $d$. This diagnosic is supported by our numerical experiments in Section~\ref{sec:numerical}.

\paragraph{Wasserstein gradient descent.}
Wasserstein gradient flows have received significant attention recently both from a theoretical perspective~\citep{ambrosio2008gradient,santambrogio2017euclidean} and as a useful tool in a variety of applications ranging from sampling~\citep{chewi2020gradient}, to variation inference~\citep{lambert2022variational}, as well the theory of neural networks~\citep{chizat2018global,SanAblBlo22}. This practical success is largely enabled by the fact that Wasserstein gradient flows can be implemented using interacting particle systems.

In fact, akin to prior work on shallow neural networks~\citep{chizat2018global, mei2018mean}, the Wasserstein gradient flow of the negative log-likelihood precisely describes the dynamics of gradient descent on the location parameters of the fitted mixture. Interested readers are referred to Theorem \ref{thm:Wass-Euclidean-connection} in Appendix \ref{sec:Wass-theory} for a rigorous statement. %\ndpr{Yling: please add a reference to a theorem in the appendix}. 

% In this section we briefly introduce the Wasserstein gradient
% flow/descent for NPMLE, and defer the theoretical properties and discussions to Appendix
% \ref{sec:Wass-theory}. 

In Appendix \ref{subsec:Wass-derivation} we derive the gradient flow
of $\ell_{N}(\rho)$ under the Wasserstein geometry, which evolves
according to the PDE
\begin{equation}
	\partial_{t}\rho_{t}=\mathsf{div}\left(\rho_{t}\nabla\delta\ell_{N}\left(\rho_{t}\right)\right)\label{eq:Wass-GF}
\end{equation}
The corresponding discrete time algorithm is
\[
\rho_{t+1}=\left[\mathrm{id}-\eta\nabla\delta\ell_{N}\left(\rho_{t}\right)\right]_{\#}\rho_{t},\qquad t=0,1,\ldots
\]
for some step size $\eta>0$. When initialized from a discrete distribution
\[
\rho_{0}=\frac{1}{m}\sum_{l=1}^{m}\delta_{\mu_{0}^{(l)}},\text{ where }\mu_{0}^{(1)},\ldots,\mu_{0}^{(m)}\in\mathbb{R}^{d},
\]
the following theorem gives a concise characterization of the Wasserstein
gradient flow. The proof can be found in Appendix \ref{sec:proof-Wass-ODE}.

\begin{theorem}[Particle Wasserstein gradient flow]\label{thm:Wass-ODE}The
	ODE system 
	\begin{align}
		\dot{\mu}_{t}^{(j)} & =\frac{1}{N}\sum_{i=1}^{N}\frac{\phi\left(X_{i}-\mu_{t}^{(j)}\right)}{\sum_{l=1}^{m}\omega_{t}^{(j)}\phi\left(X_{i}-\mu_{t}^{(l)}\right)}\left(X_{i}-\mu_{t}^{(j)}\right),\label{eq:Wass-ODE}
	\end{align}
	with initialization $\mu_{0}^{(1)},\ldots\mu_{0}^{(m)}\overset{\text{i.i.d.}}{\sim}\mathsf{Uniform}(\{X_{i}\}_{1\leq i\leq N})$
	has unique solution on any time interval $[0,T]$. Moreover, the flow
	$(\rho_{t})_{t\geq0}$ defined as
	\begin{equation}
		\rho_{t}\coloneqq\frac{1}{m}\sum_{l=1}^{m}\delta_{\mu_{t}^{(l)}}\label{eq:Wass-discrete}
	\end{equation}
	is the Wasserstein gradient flow, i.e.~ \eqref{eq:Wass-discrete}
	is a distributional solution to the PDE  \eqref{eq:Wass-GF}.
	
\end{theorem}

By time discretization, we also have the Wasserstein gradient descent
algorithm; see Algorithm \ref{alg:Wasserstein-GD}. 
\begin{algorithm}[h]
	\caption{Wasserstein gradient descent.}
	
	\label{alg:Wasserstein-GD}\begin{algorithmic}
		
		\STATE \textbf{{Input}}: data $\{X_{i}\}_{1\leq i\leq n}$, number
		of particles $m$, step sizes $\eta>0$, maximum number of iterations
		$t_{0}$.
		
		\STATE \textbf{{Initialization}}: draw $\mu_{0}^{(1)},\ldots,\mu_{0}^{(m)}\overset{\text{i.i.d.}}{\sim}\mathsf{Unif}(\{X_{i}\}_{1\leq i\leq n})$.
		
		\vspace{0.2em}
		
		\STATE \textbf{{Updates}}: \textbf{for }$t=0,1,\ldots,t_{0}$ \textbf{do}
		
		\STATE \vspace{-1em}
		\begin{subequations} 
			\begin{align*}
				\mu_{t+1}^{(j)} & =\mu_{t}^{(j)}+\eta\frac{1}{N}\sum_{i=1}^{N}\frac{\phi\left(X_{i}-\mu_{t}^{(j)}\right)}{m^{-1}\sum_{l=1}^{m}\phi\left(X_{i}-\mu_{t}^{(l)}\right)}\left(X_{i}-\mu_{t}^{(j)}\right),
			\end{align*}
		\end{subequations} for all $j=1,\ldots,m$. Here $\phi(x)=(2\pi)^{-d/2}\exp(-\Vert x\Vert_{2}^{2}/2)$
		is the probability density function of $\mathcal{N}(0,I_{d})$.
		
		\STATE \textbf{{Output}} $\rho=m^{-1}\sum_{j=1}^{m}\delta_{\mu_{t_{0}}^{(j)}}$
		as the (approximate) NMPLE.
		
	\end{algorithmic}
\end{algorithm}

Unlike the Fisher-Rao gradient descent algorithm, the approximation error of the Wasserstein flow may be mitigated since it allows particles to evolve in space. Nevertheless, this movement can take a long time to move mass from one part of the space at initialization to a distant part of the space. Instead, Wasserstein-Fisher-Rao gradient descent allows for particles to not only evolve in space but also have changing weights, thus greatly improving the performance compared to vanilla Wasserstein gradient descent. This superiority is, again, demonstrated in numerical experiments. In particular, owing to Theorem \ref{thm:Wass-Euclidean-connection}, these experiments indicate that WFR gradient descent dominates the classical gradient descent on the location parameters of the mixing distribution.

The reader will notice here the absence of convergence results analogous to Theorem~\ref{thm-convergence-wfr-iteration} for WFR and Theorem~\ref{thm:FR-convergence} for Fisher-Rao gradient descent. Unfortunately, we were not able to derive such convergence results. The dynamics of the Wasserstein gradient flow are quite intricate and Appendix~\ref{sec:Wass-theory} is devoted to establishing partial results that shed light on them.

%% file: numerical.tex
\section{Numerical experiments \label{sec:numerical}}

In this section, we conduct a series of numerical experiments to validate
and complement our theory. We consider two mixing distributions in
$\mathbb{R}^{d}$: the first one is a continuous isotropic Gaussian
distribution 
\[
\rho_{\mathsf{c}}=\mathcal{N}\left(0,I_{d}\right),
\]
and the other one is a discrete distribution motivated by \citet{jin2016local}
\[
\rho_{\mathsf{d}}=\left(\frac{1}{3}\delta_{-1}+\frac{1}{3}\delta_{1}+\frac{1}{3}\delta_{10}\right)\otimes\left(\delta_{0}\right)^{\otimes(d-1)}.
\]
The second mixing distribution $\rho_{\mathsf{d}}$ is a product distribution
with its first margin being a uniform distribution over $\{-1,1,10\}$
and the rest $d-1$ margins being a degenerate distribution taking
a constant zero. According to \citet{jin2016local}, classical EM and
gradient descent algorithm fail to learn the location of each component
of this Gaussian mixture even with infinite samples and known
weights.

\subsection{Instability of classical algorithms\label{subsec:numerical-1}}

\begin{figure}[t]
\centering

\begin{tabular}{cc}
\qquad(a) bad local minima  & \qquad(b) global minima\tabularnewline
\includegraphics[scale=0.4]{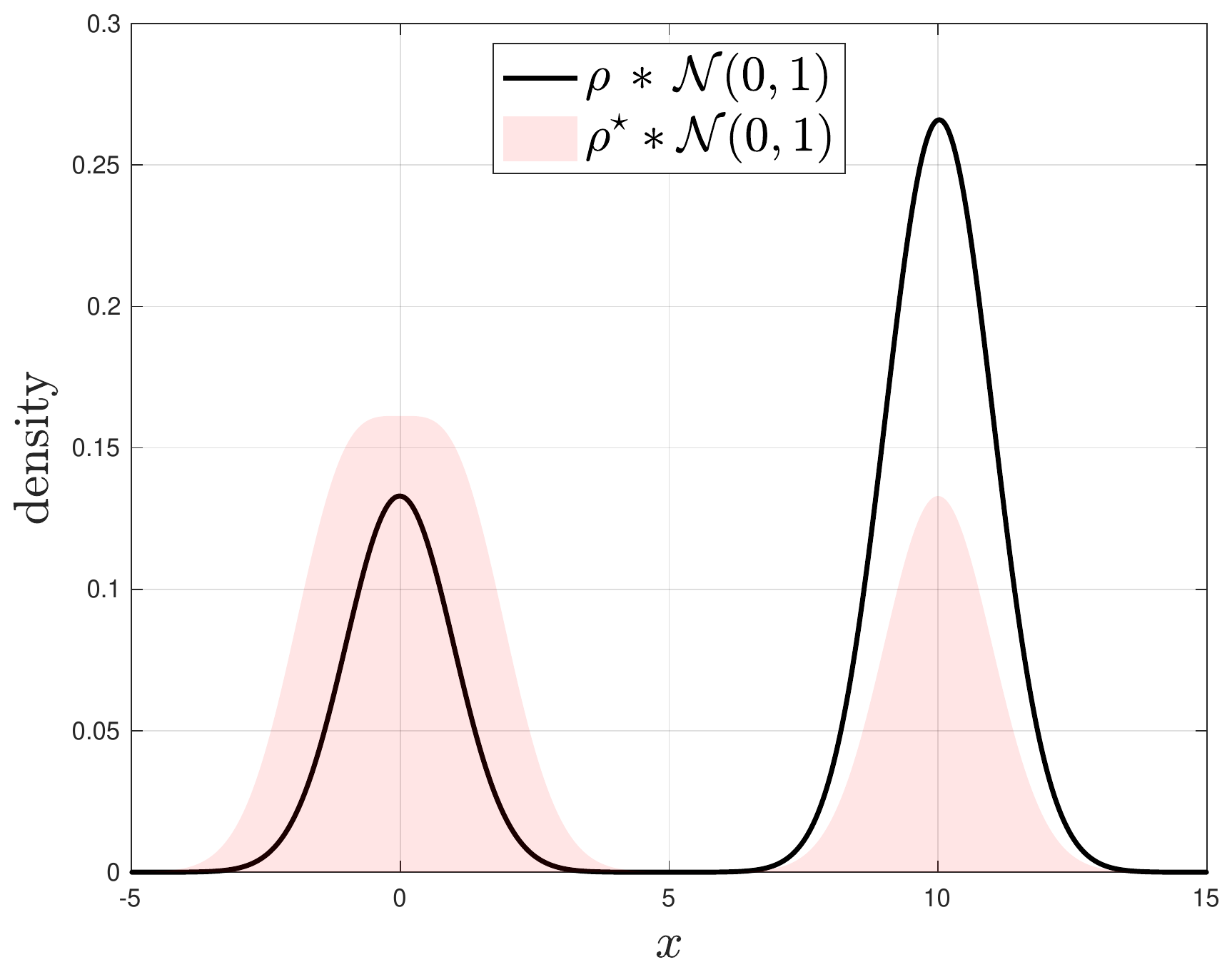} & \includegraphics[scale=0.4]{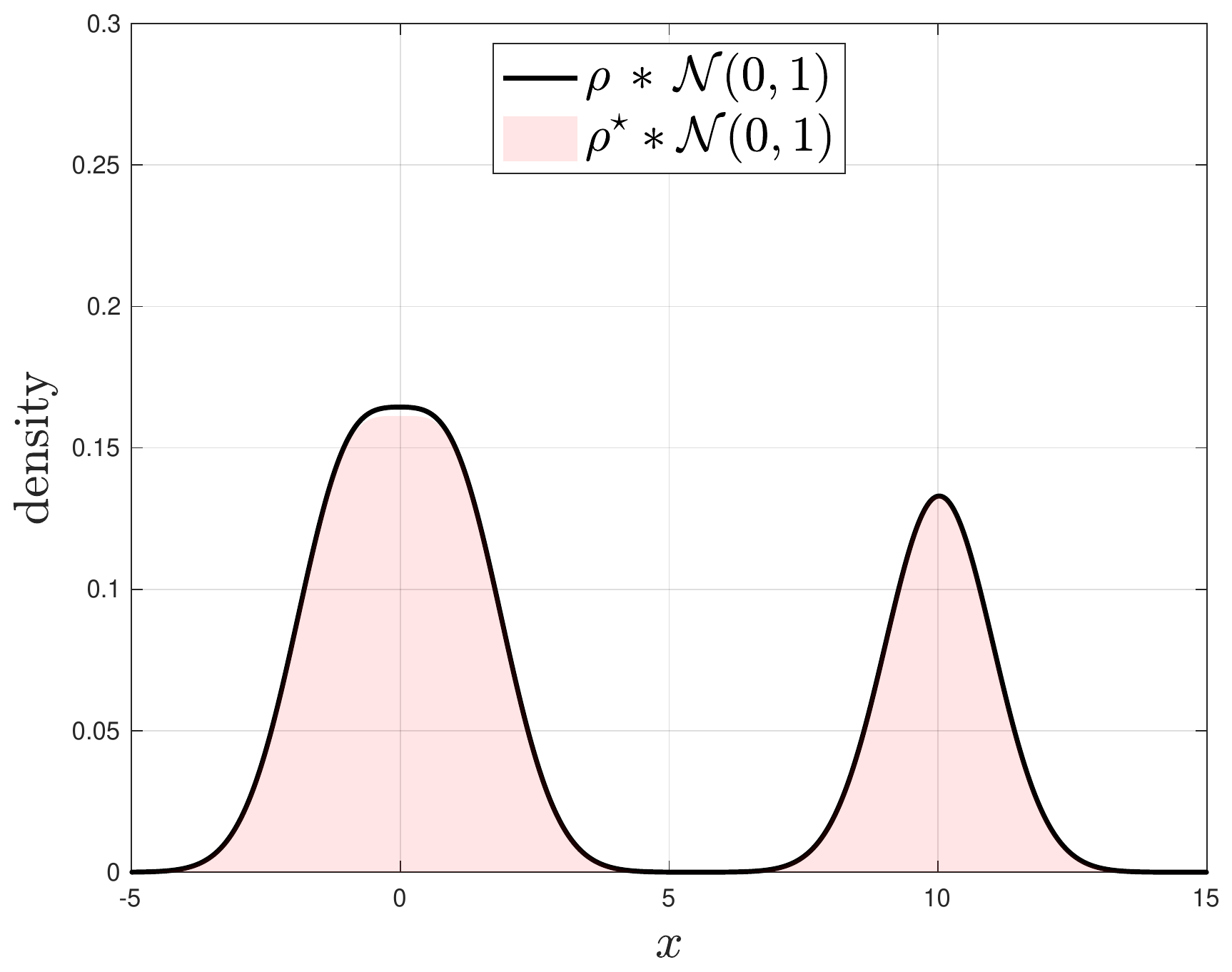}\tabularnewline
\end{tabular}

\caption{The density plots of $\rho*\mathcal{N}(0,1)$ (learned Gaussian mixture)
and $\mathcal{\rho^{\star}*\mathcal{N}}(0,1)$ (true Gaussian mixture),
where $\rho=\frac{1}{3}\delta_{\mu_{1}}+\frac{1}{3}\delta_{\mu_{2}}+\frac{1}{3}\delta_{\mu_{3}}$,
$\rho^{\star}=\rho_{\mathsf{d}}$ and $(\mu_{1},\mu_{2},\mu_{3})$
is the output returned by EM and GD algorithms. Figure \ref{fig:stability}(a)
shows the learned Gaussian mixture when $(\mu_{1},\mu_{2},\mu_{3})$
is a bad local minimum of $\ell$, and Figure \ref{fig:stability}(b)
corresponds to the case when $(\mu_{1},\mu_{2},\mu_{3})$ is a global
minimum of $\ell$. \label{fig:stability}}
\end{figure}

In this section, we compare the two classical algorithms for learning
Gaussian mixture when the mixing distribution $\rho^\star=\rho_{\mathsf{d}}$: (i) expectation--maximization (EM) algorithm, and
(ii) the gradient descent (GD) algorithm, with Wasserstein-Fisher-Rao
gradient descent algorithm (cf.~Algorithm \ref{alg:WFR-GD}) proposed
in this paper. For the first two algorithms, we assume that the number
of mixture components $k=3$ and the weights $\omega_{1}^{\star}=\omega_{2}^{\star}=\omega_{3}^{\star}=1/3$
are known a priori, and implement the standard EM and GD algorithms
for solving the MLE
\[
\min_{\mu_1, \mu_2, \mu_3}\,\ell\left(\mu_{1},\mu_{2},\mu_{3}\right)\coloneqq-\frac{1}{N}\sum_{i=1}^{N}\log\bigg\{\frac{1}{3}\sum_{j=1}^{3}\frac{1}{\left(2\pi\right)^{d/2}}\exp\left[-\frac{1}{2}\left\Vert X_{i}-\mu_{j}\right\Vert _{2}^{2}\right]\bigg\}.
\]
The updating rule for EM algorithm is given by
\begin{equation}
\mu_{j}^{t+1}=\frac{\sum_{i=1}^{N}\omega_{i,j}^{t}X_{i}}{\sum_{i=1}^{N}\omega_{i,j}^{t}}\qquad\text{where}\qquad\omega_{i,j}^{t}=\frac{\omega_{j}^{\star}\phi\left(X_{i};\mu_{j}^{t},I_{d}\right)}{\sum_{l=1}^{3}\omega_{l}^{\star}\phi\left(X_{i};\mu_{l}^{t},I_{d}\right)}\quad\forall\,i\in\left[N\right].\label{eq:EM-update}
\end{equation}
for all $j=1,2,3$ and $t\geq0$, with random initialization from
the samples $\mu_{1}^{0},\mu_{2}^{0},\mu_{3}^{0}\overset{\text{i.i.d.}}{\sim}\mathsf{Uniform}(\{X_{i}\}_{1\leq i\leq N})$.
The GD algorithm coincides with the Wasserstein gradient descent (cf.~Algorithm
\ref{alg:Wasserstein-GD}) with $m=3$ particles.

We consider the one-dimensional setting (i.e.~$d=1$) for simplicity
of visualization, and this turns out to be enough for showing the
instability of EM and GD. We generate $N=1500$ samples $\{X_{i}\}_{1\leq i\leq N}$
from $\rho^{\star}*\mathcal{N}(0,1)$. Then we fix these samples and
run $100$ independent trials of the three algorithms. For EM, we
run $200$ iterations. For GD and Wasserstein-Fisher-Rao gradient
descent, we set all the step sizes to be $\eta=0.1$ and run $t_{0}=1000$
iterations. Both EM and GD have two possible outputs: (i) the first
one is $\mu_{1}\approx\mu_{2}\approx10$ and $\mu_{3}\approx0$ (up
to permutation), which is a bad local minimum of $\ell$; (ii) the
second one is $\mu_{1}\approx-1$, $\mu_{2}\approx1$ and $\mu_{3}\approx10$,
which is the global minimum of $\ell$. Figure \ref{fig:stability}
displays the learned Gaussian mixtures correspond to these two outputs,
and it is clear that both algorighms fail to learn the true Gaussian
mixture when they converge to the bad local minima. In the 100 independent
trials, EM converges to the bad local minimum for $23$ times, while
GD converges to the bad local minimum for $32$ times, both exhibiting instability
vis-\`{a}-vis random initialization. In contrast, as we will show in Section
\ref{subsec:numerical-3}, Wasserstein-Fisher-Rao gradient descent
with number of particles $m=500$ converges to NPMLE stably and learns
the Gaussian mixture as in Figure \ref{fig:optimality}(c). 

\subsection{Superiority of Wasserstein-Fisher-Rao gradient descent \label{subsec:numerical-2}}

\begin{figure}[t]
\centering

\begin{tabular}{cc}
\qquad(a) training error for discrete $\rho^{\star}$ & \qquad(b) test error for discrete $\rho^{\star}$\tabularnewline
\includegraphics[scale=0.4]{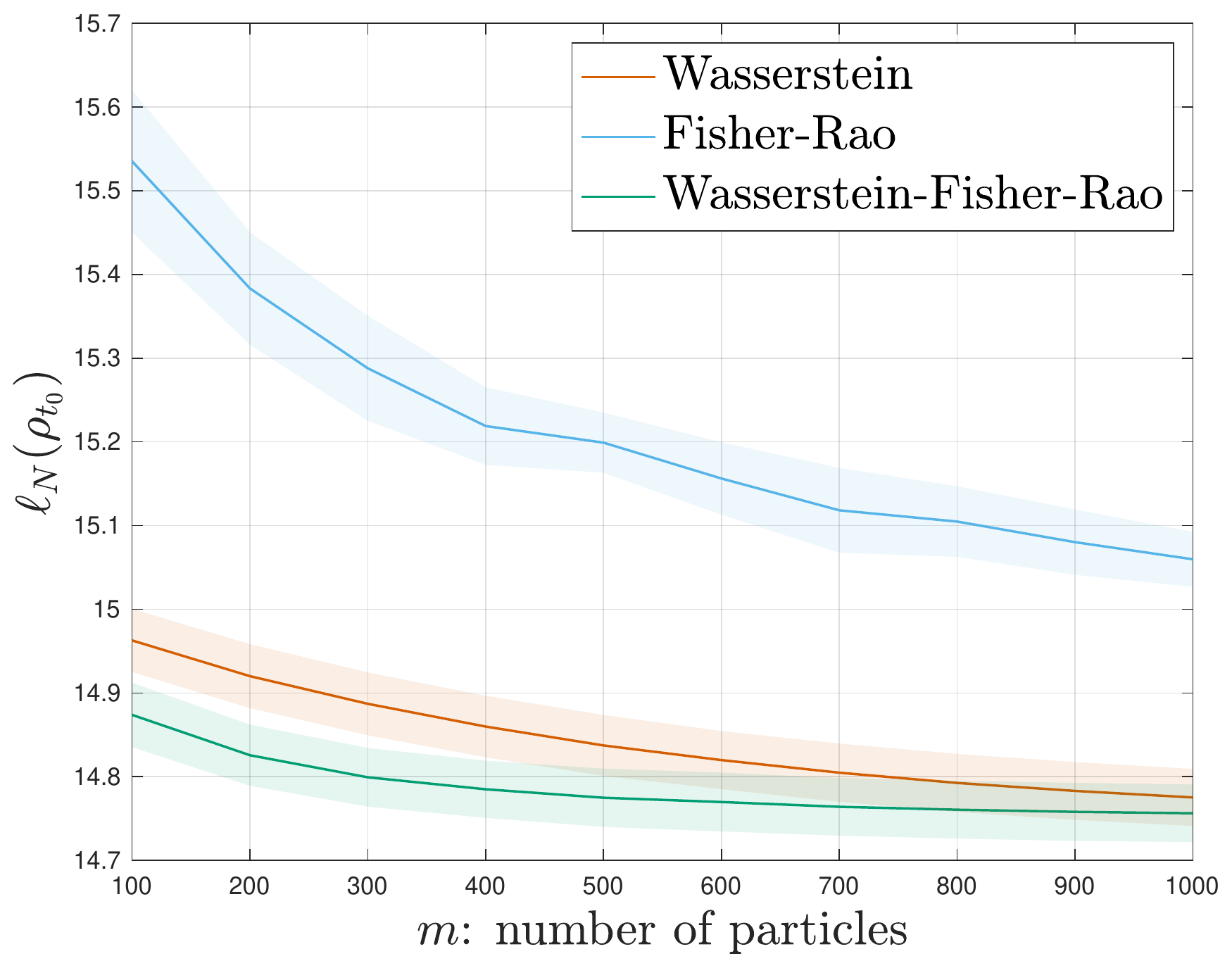} & \includegraphics[scale=0.4]{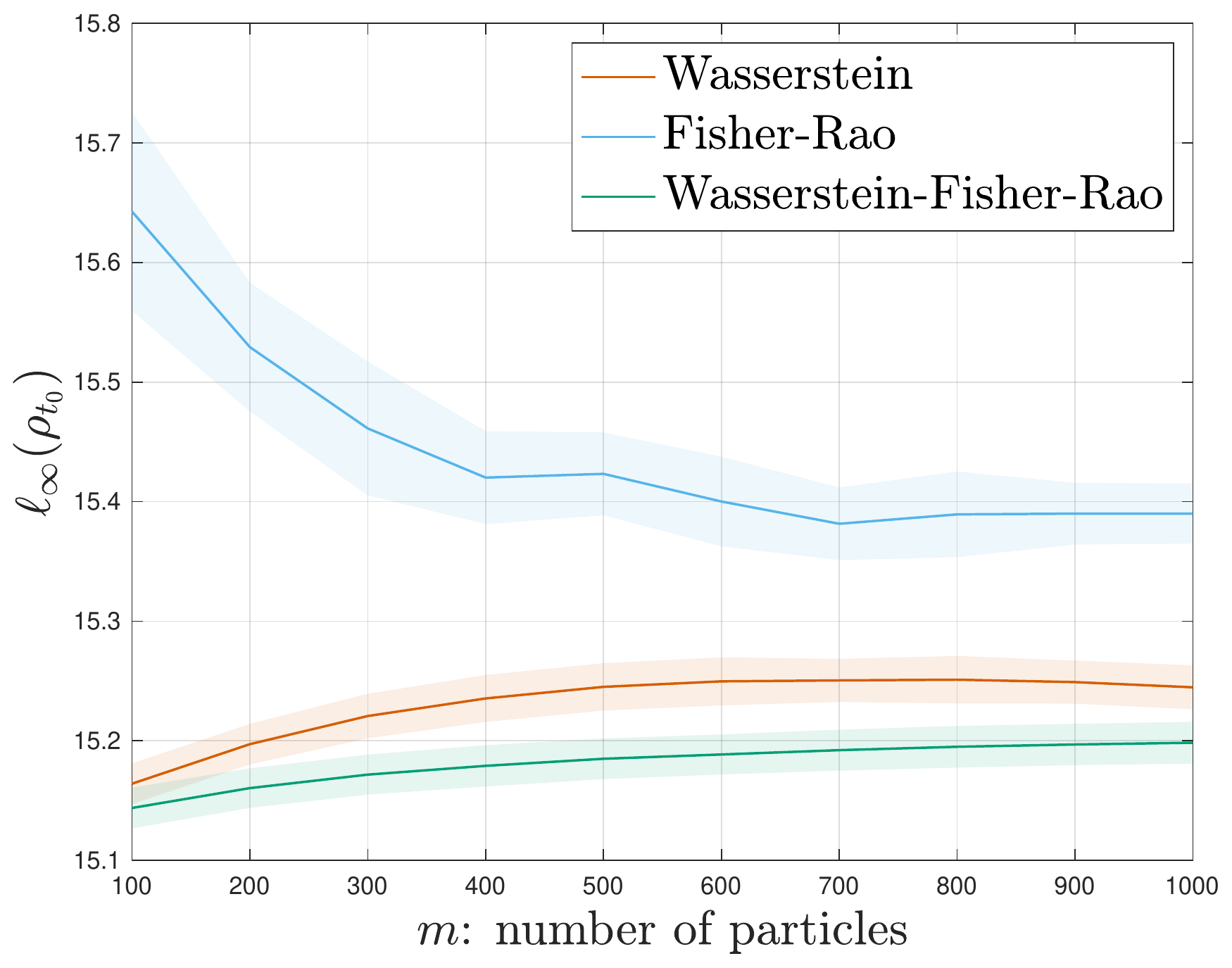}\tabularnewline
\qquad(c) training error for continuous $\rho^{\star}$ & \qquad(d) test error for continuous $\rho^{\star}$\tabularnewline
\includegraphics[scale=0.4]{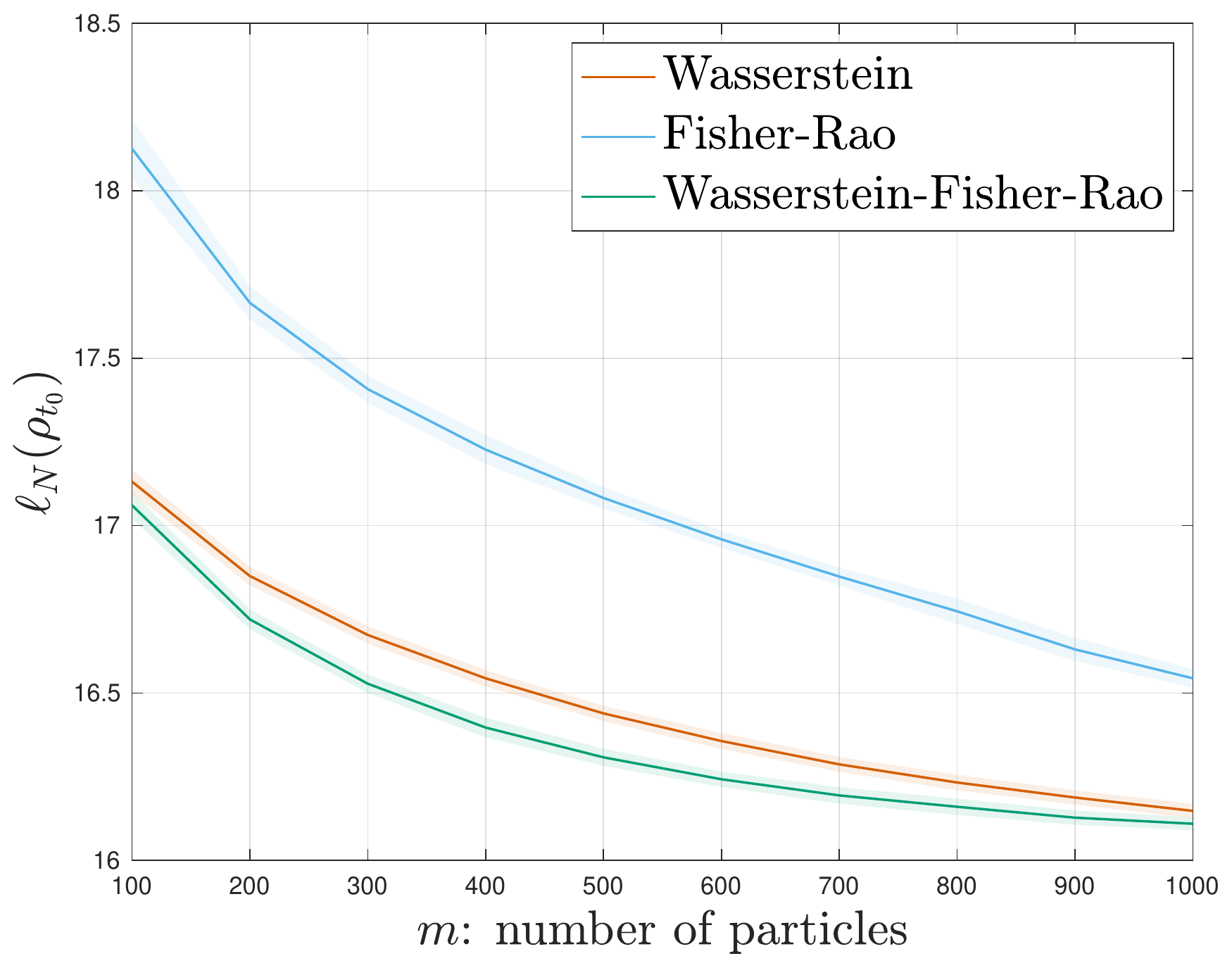} & \includegraphics[scale=0.4]{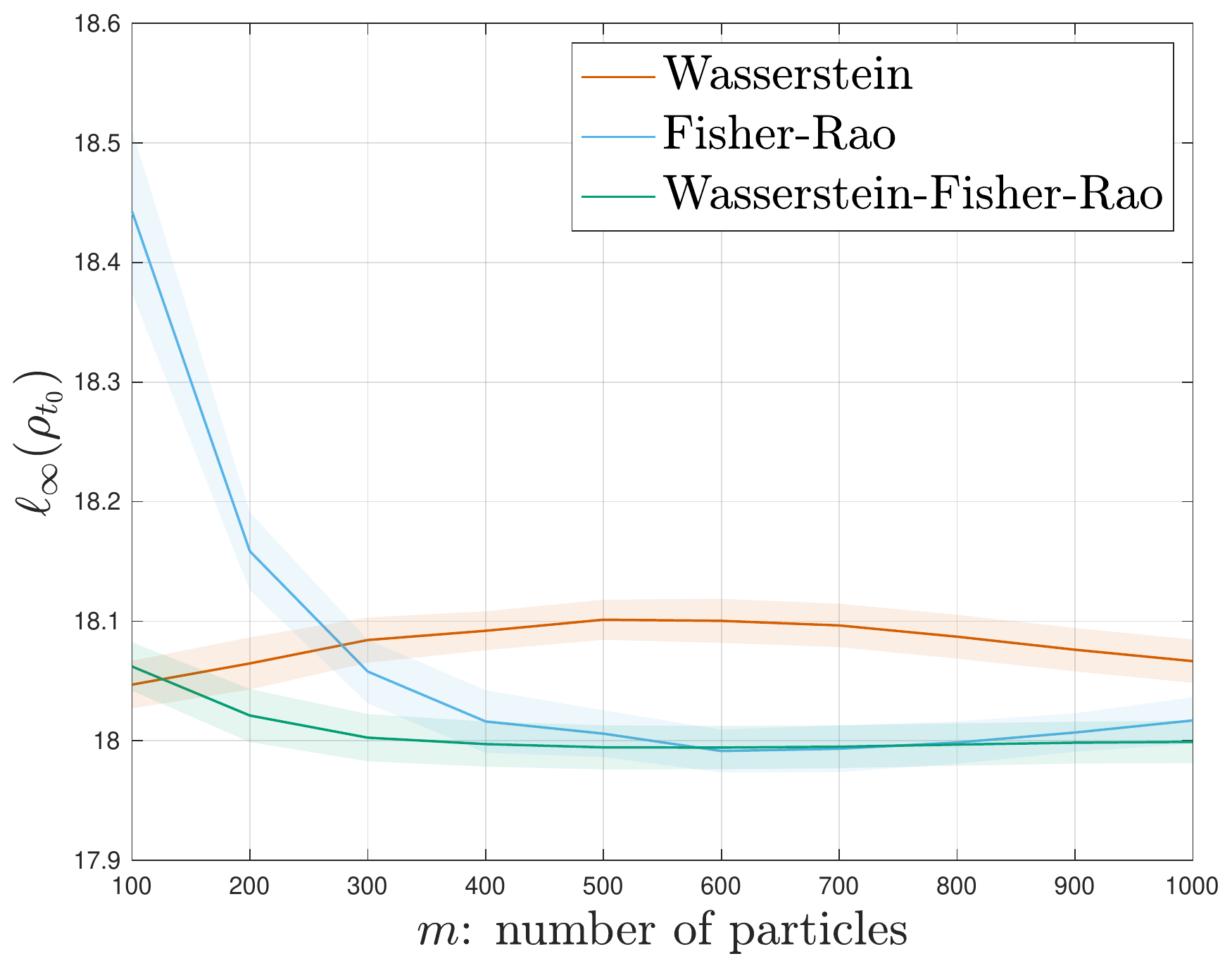}\tabularnewline
\end{tabular}

\caption{Training or testing error (with error bars) of the three algorithms
vs.~the number of particles. The training and testing errors
are evaluated using $\ell_{N}(\rho)$ and $\ell_{\infty}(\rho)$ respectively.
Figure (a) an (b) display the training and testing errors when the
mixing distribution $\rho^{\star}=\rho_{\mathsf{d}}$ is discrete,
while Figure (c) and (d) show the training and testing errors when
the mixing distribution $\rho^{\star}=\rho_{\mathsf{c}}$ is continuous.
The results are reported over $20$ independent trials for $N=1500$,
$d=10$, and $t_{0}=1000$. \label{fig:loss_vs_m}}
\end{figure}

\begin{figure}[t]
\centering

\begin{tabular}{cc}
\qquad(a) training error for discrete $\rho^{\star}$ & \qquad(b) test error for discrete $\rho^{\star}$\tabularnewline
\includegraphics[scale=0.4]{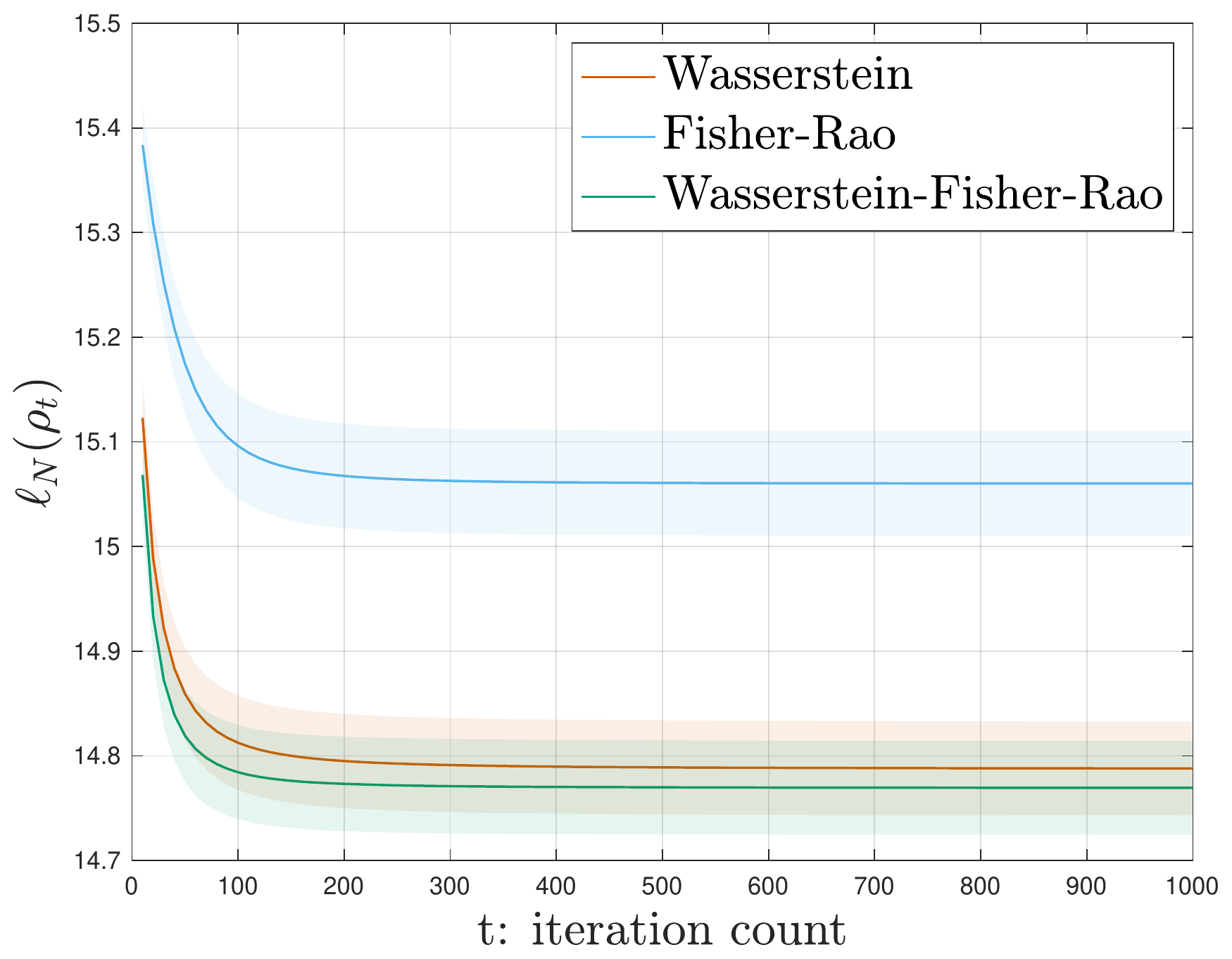} & \includegraphics[scale=0.4]{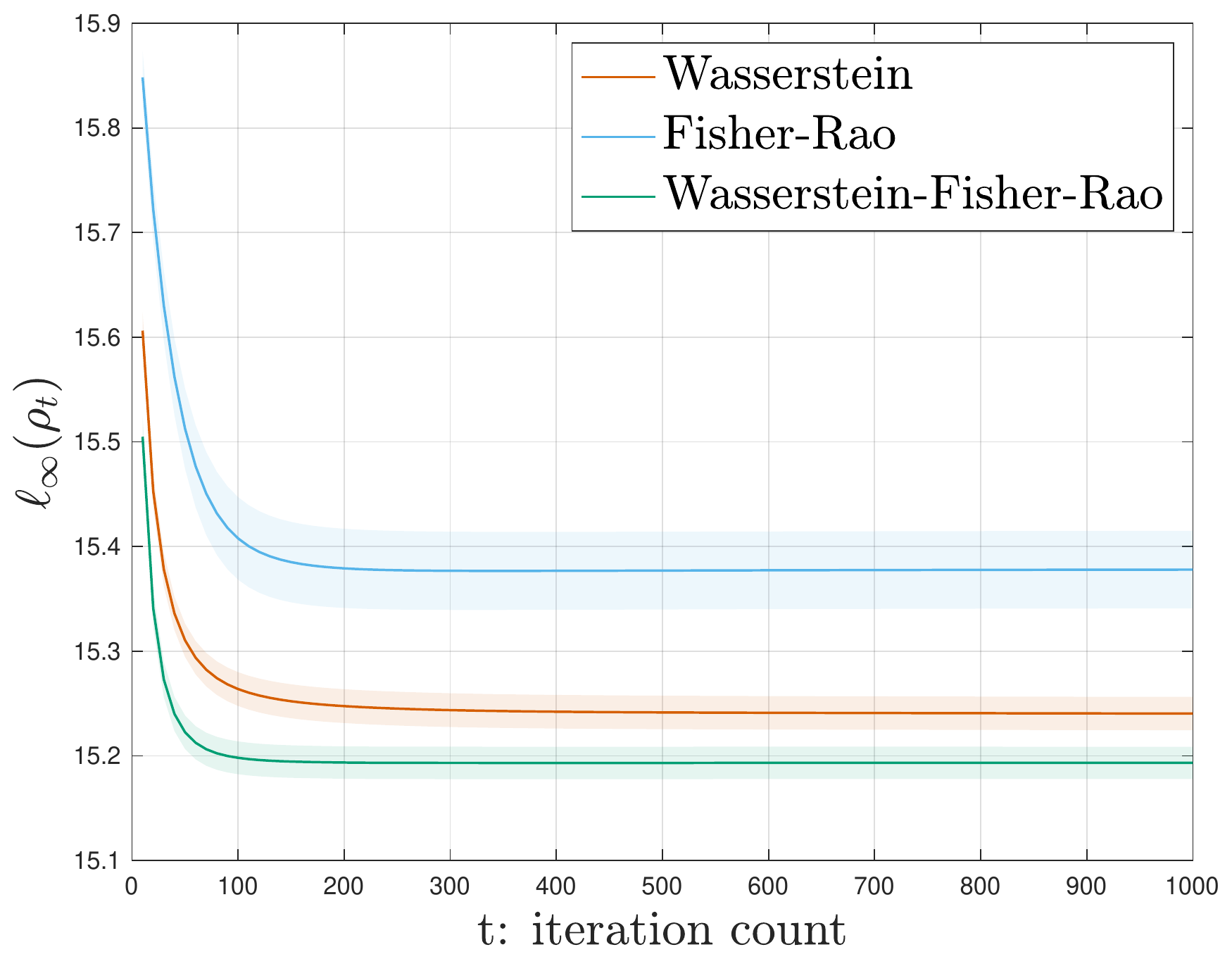}\tabularnewline
\qquad(c) training error for continuous $\rho^{\star}$ & \qquad(d) test error for continuous $\rho^{\star}$\tabularnewline
\includegraphics[scale=0.4]{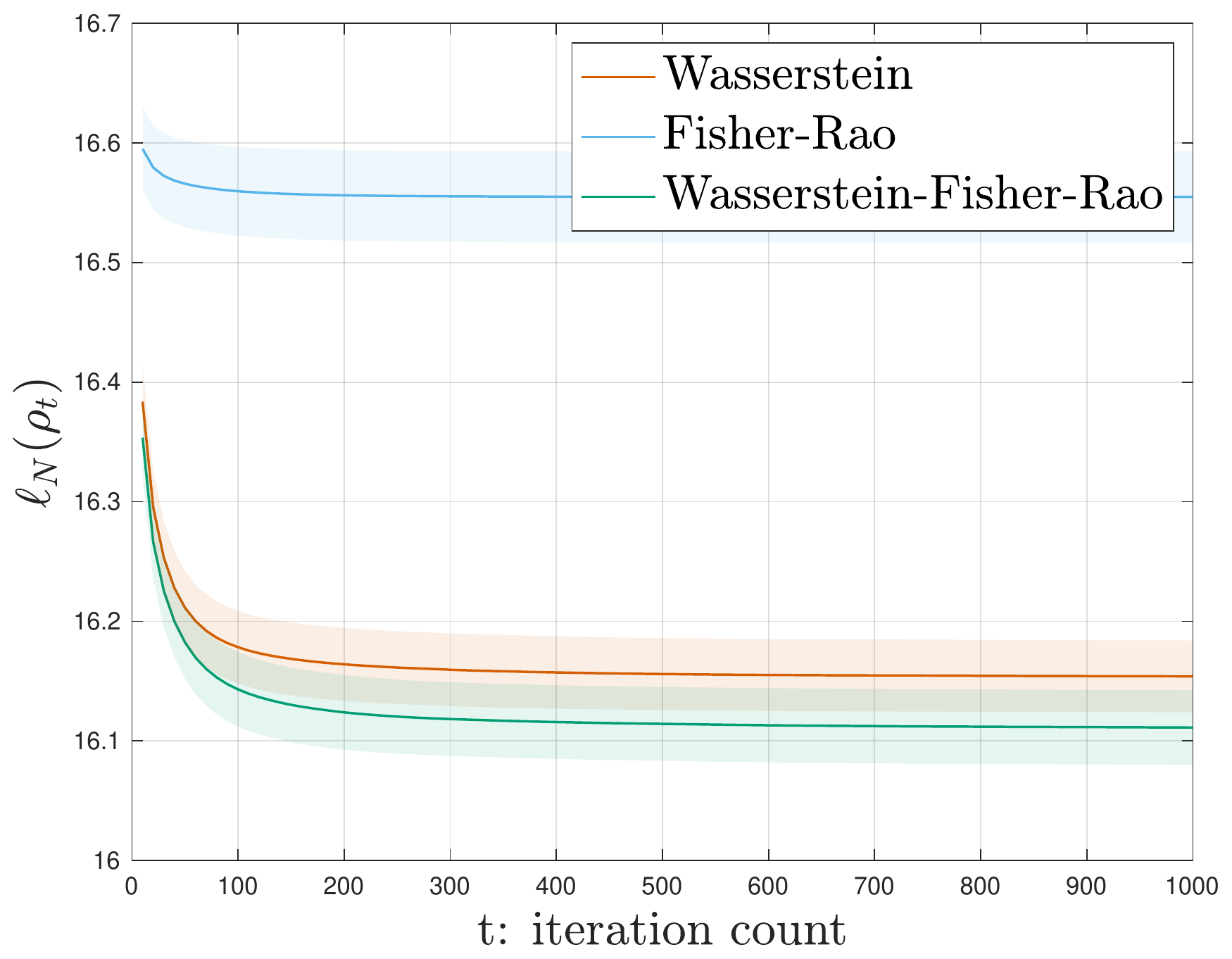} & \includegraphics[scale=0.4]{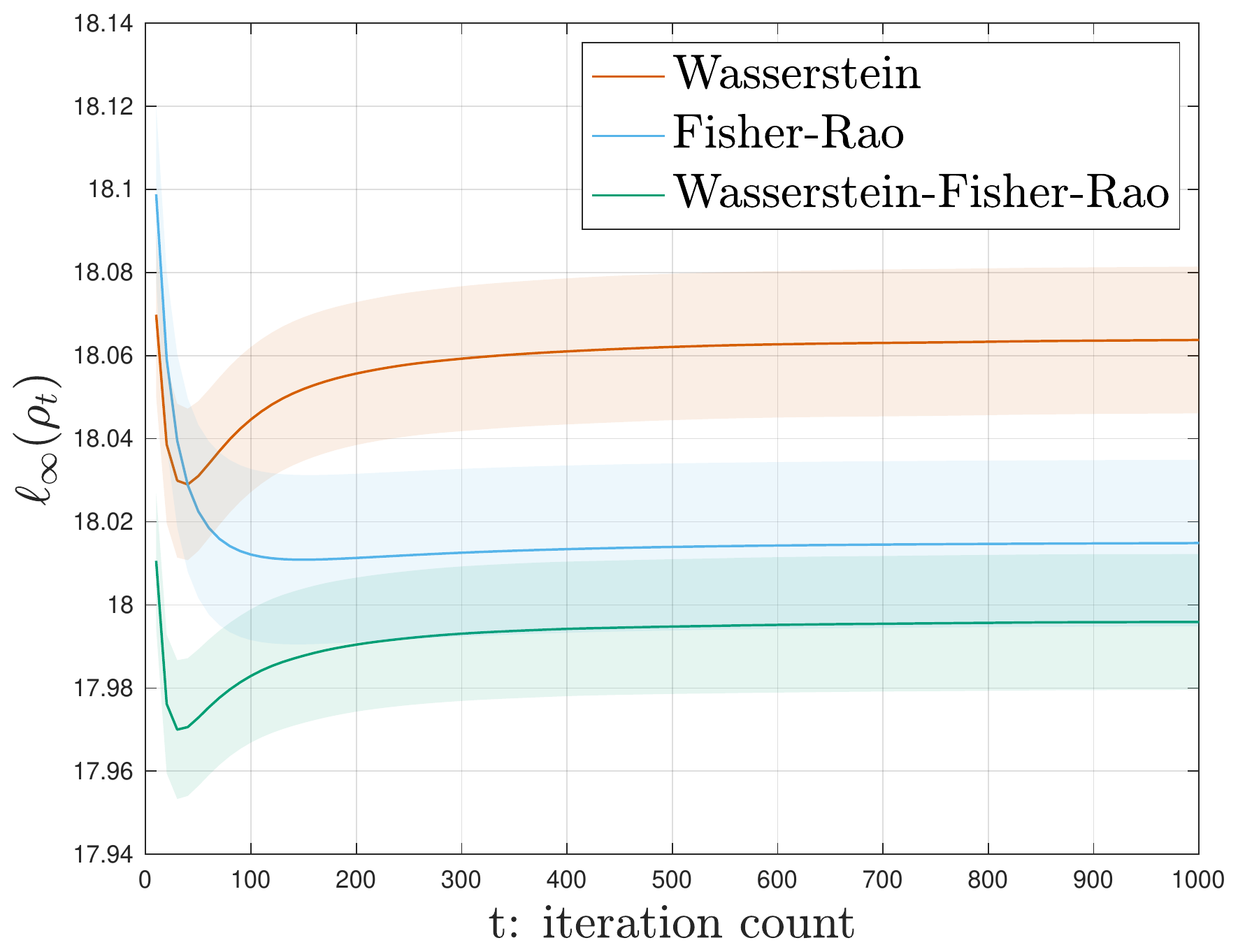}\tabularnewline
\end{tabular}

\caption{Training or testing error (with error bars) of the three algorithms
vs.~the iteration count. The training and the testing errors are
evaluated using $\ell_{N}(\rho_{t})$ and $\ell_{\infty}(\rho_{t})$
respectively. Figure (a) an (b) display training and testing errors
when the mixing distribution $\rho^{\star}=\rho_{\mathsf{d}}$ is
discrete, while Figure (c) and (d) show training and testing errors
when the mixing distribution $\rho^{\star}=\rho_{\mathsf{c}}$ is
continuous. The results are reported over $20$ independent trials
for $N=1500$, $d=10$, and $m=500$. \label{fig:loss_vs_iter}}
\end{figure}

In this section we compare the empirical performance of three gradient
descent algorithms studied in this paper: (i) Fisher-Rao gradient
descent (cf.~Algorithm \ref{alg:FR-GD}), (ii) Wasserstein-Fisher-Rao
gradient descent (cf.~Algorithm \ref{alg:WFR-GD}) and (iii) Wasserstein
gradient descent (cf.~Algorithm \ref{alg:Wasserstein-GD}). We set
the sample size $N=1500$, the dimension $d=10$, maximum number of
iterations $t_{0}=1000$, the step size $\eta=10^{-1}$ for Algorithm
\ref{alg:FR-GD} and Algorithm \ref{alg:Wasserstein-GD}, and $\eta = 10^{-2}$ for Algorithm \ref{alg:WFR-GD}. Figure \ref{fig:loss_vs_m}(a) displays
the negative log-likelihood $\ell_{N}(\rho)$ (with one standard deviation
error bars) vs.~the number of particles $m$ over 20 independent
trials for the three algorithms. Unlike the previous experiment, we
generate fresh samples $\{X_{i}\}_{1\leq i\leq N}$ in each independent
trial. As we can see, the loss decreases as we use more particles,
and Wasserstein-Fisher-Rao gradient descent achieves the smallest
loss uniformly for all $m$. It can also be observed that the marginal
benefit of increasing the number of particles becomes negligible for
Wasserstein-Fisher-Rao gradient descent when $m>500$. Similarly,
Figure \ref{fig:loss_vs_m}(b) depicts the negative log-likelihood
$\ell_{N}(\rho_{t})$ (with one standard deviation error bars) vs.~the
iteration count $t$ over 20 independent trials for the three algorithms.
We can see that Wasserstein-Fisher-Rao gradient descent again achieves
the smallest loss uniformly in all iteration, confirming again the
superiority of the algorithm.

\subsection{Certifying the optimality condition for NPMLE in one dimension \label{subsec:numerical-3}}

\begin{figure}[t]
\centering

\begin{tabular}{cc}
\qquad(a) optimality gap for discrete $\rho^{\star}$ & \qquad(b) optimality gap for continuous $\rho^{\star}$\tabularnewline
\includegraphics[scale=0.4]{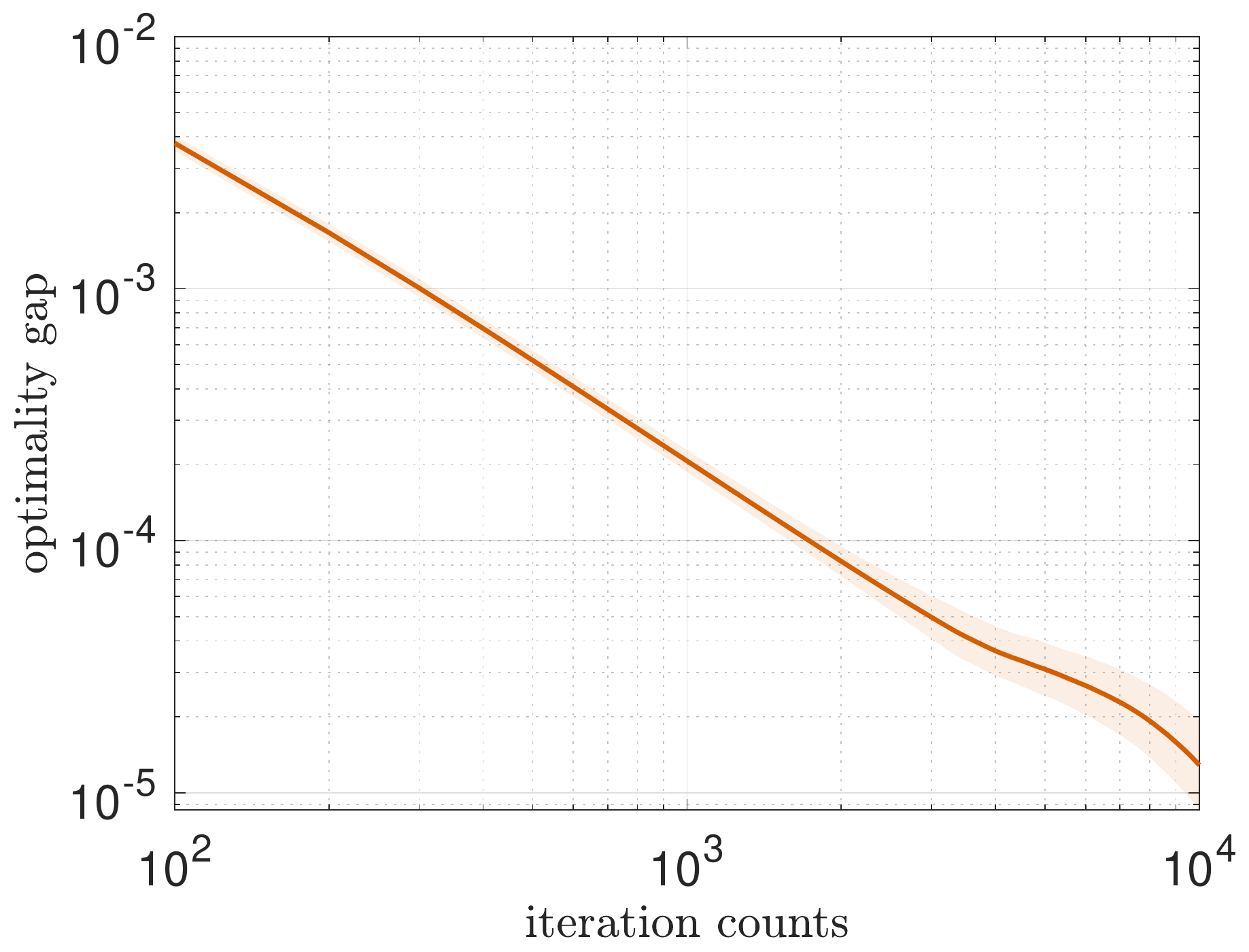} & \includegraphics[scale=0.4]{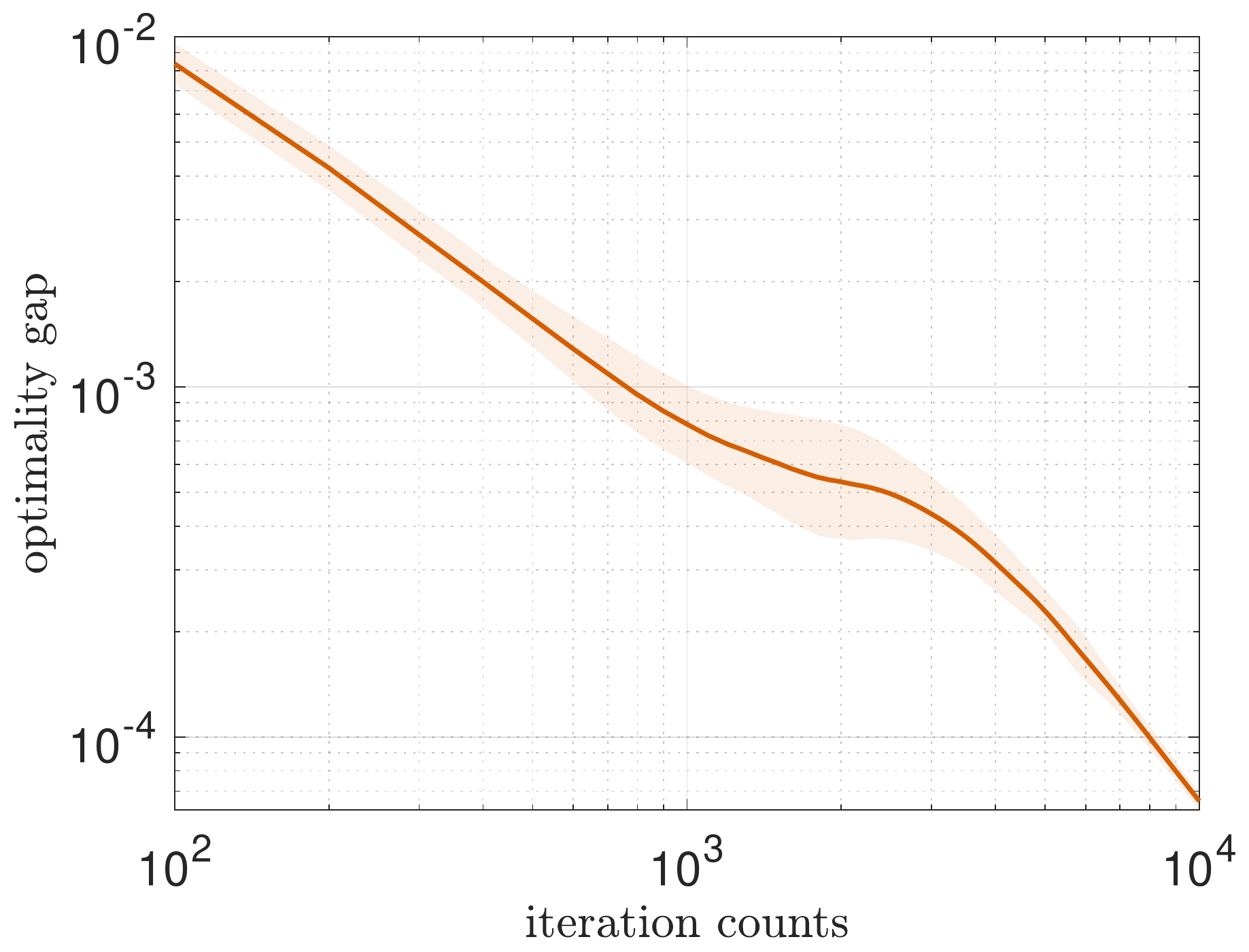}\tabularnewline
\qquad(c) density plot for discrete $\rho^{\star}$ & \qquad(d) density plot for continuous $\rho^{\star}$\tabularnewline
\includegraphics[scale=0.4]{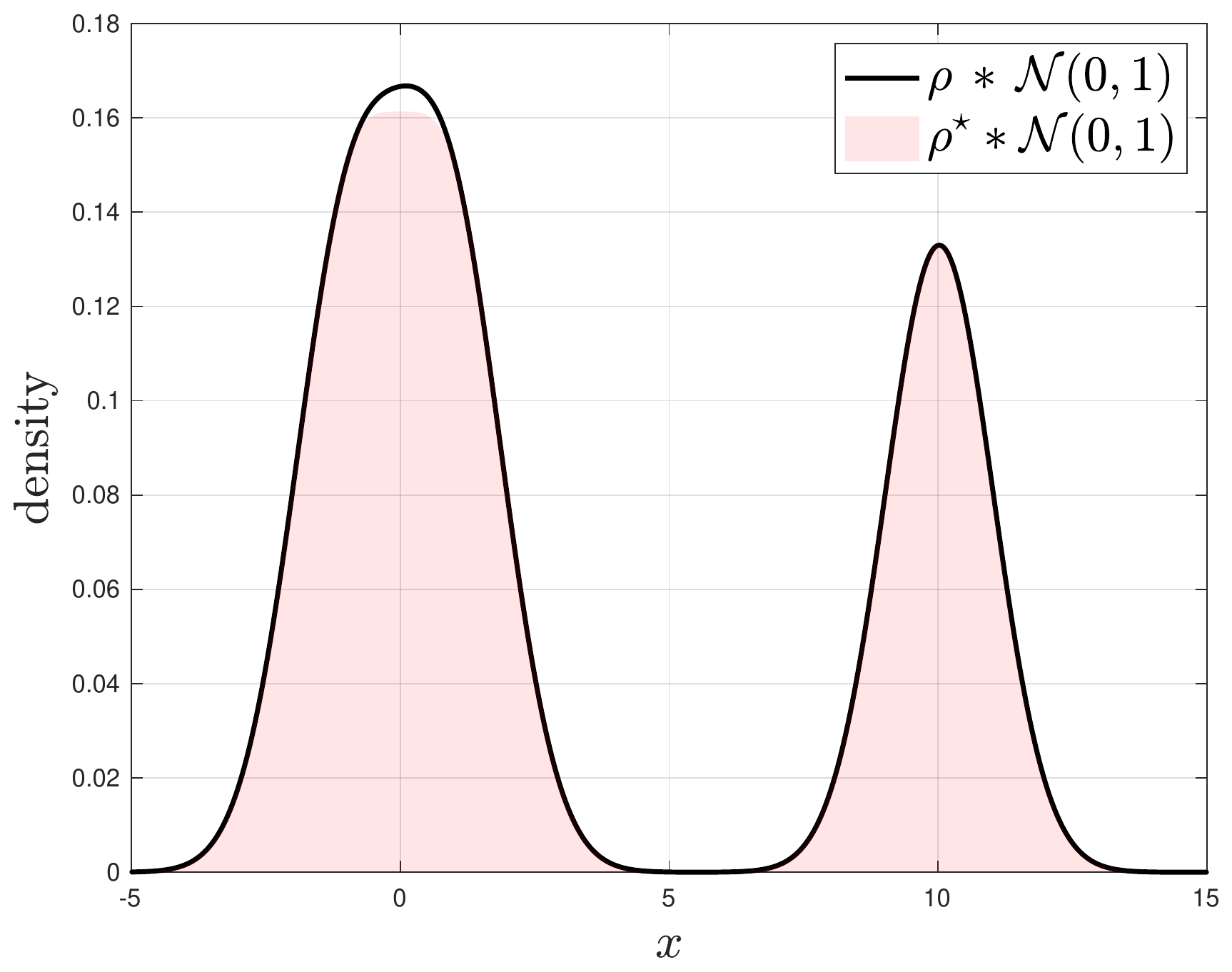} & \includegraphics[scale=0.4]{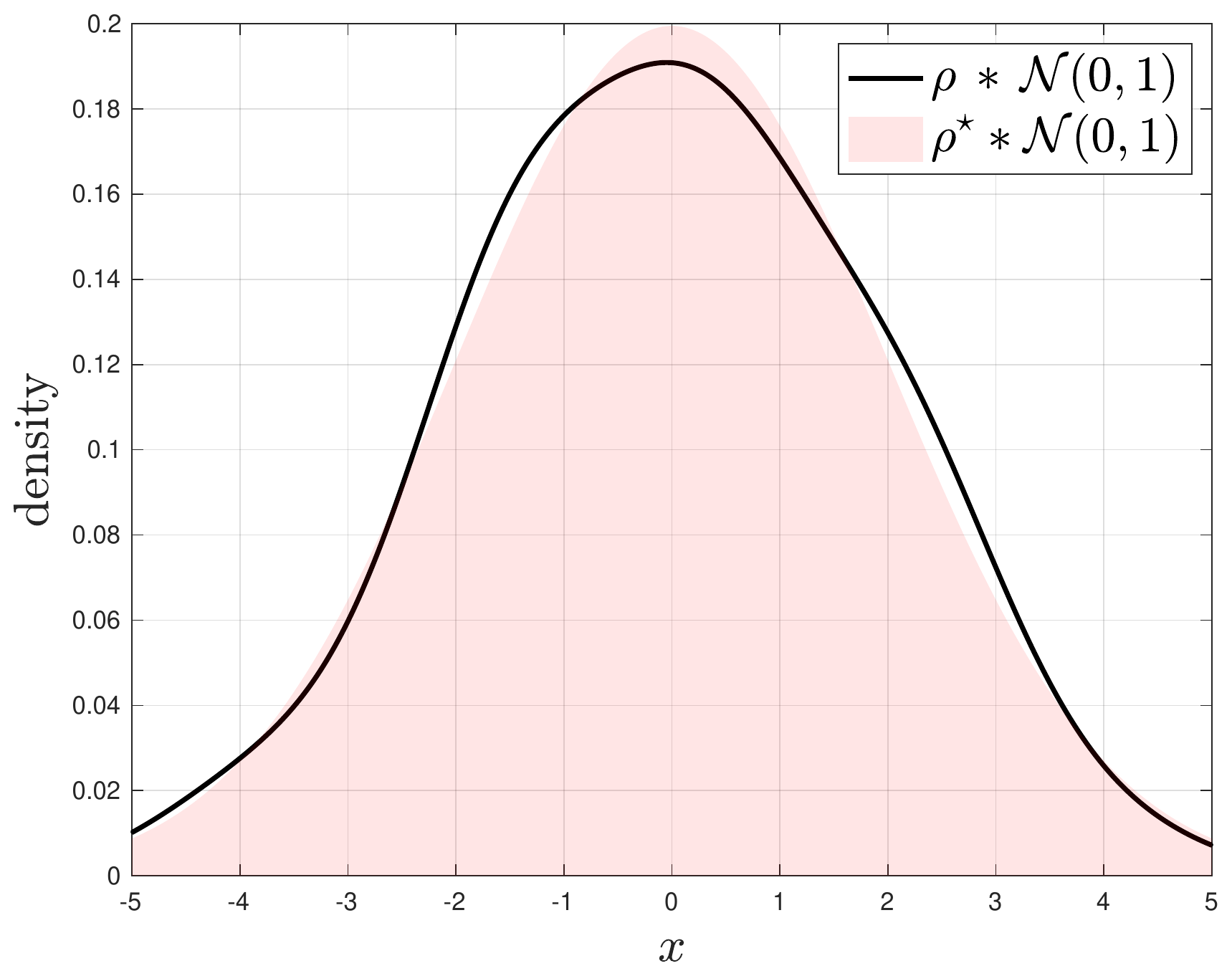}\tabularnewline
\end{tabular}

\caption{Figures (a)-(b) display sub-optimality gaps (cf.~(\ref{eq:suboptimality-gap}))
vs.~iteration count for Wasserstein-Fisher-Rao gradient descent,
with discrete mixing distribution $\rho^{\star}=\rho_{\mathsf{d}}$
in Figure (a) and $\rho^{\star}=\rho_{\mathsf{c}}$ in Figure (b).
The error bars are computed over $20$ independent trials. Figures
(c)-(d) are density plots of $\rho$ and $\rho^{\star}$ convolved
with standard Gaussian, with discrete mixing distribution $\rho^{\star}=\rho_{\mathsf{d}}$
in (c) and continuous mixing distribution $\rho^{\star}=\rho_{\mathsf{c}}$
in (d). The results are reported for $N=1500$, $d=1$ and $m=500$.
\label{fig:optimality}}
\end{figure}

\begin{figure}[t]
\centering

\begin{tabular}{c}
\qquad(a) first variation $\delta\ell_{N}(\rho)$ for discrete $\rho^{\star}=\frac{1}{3}\delta_{-1}+\frac{1}{3}\delta_{1}+\frac{1}{3}\delta_{10}$\tabularnewline
\includegraphics[scale=0.55]{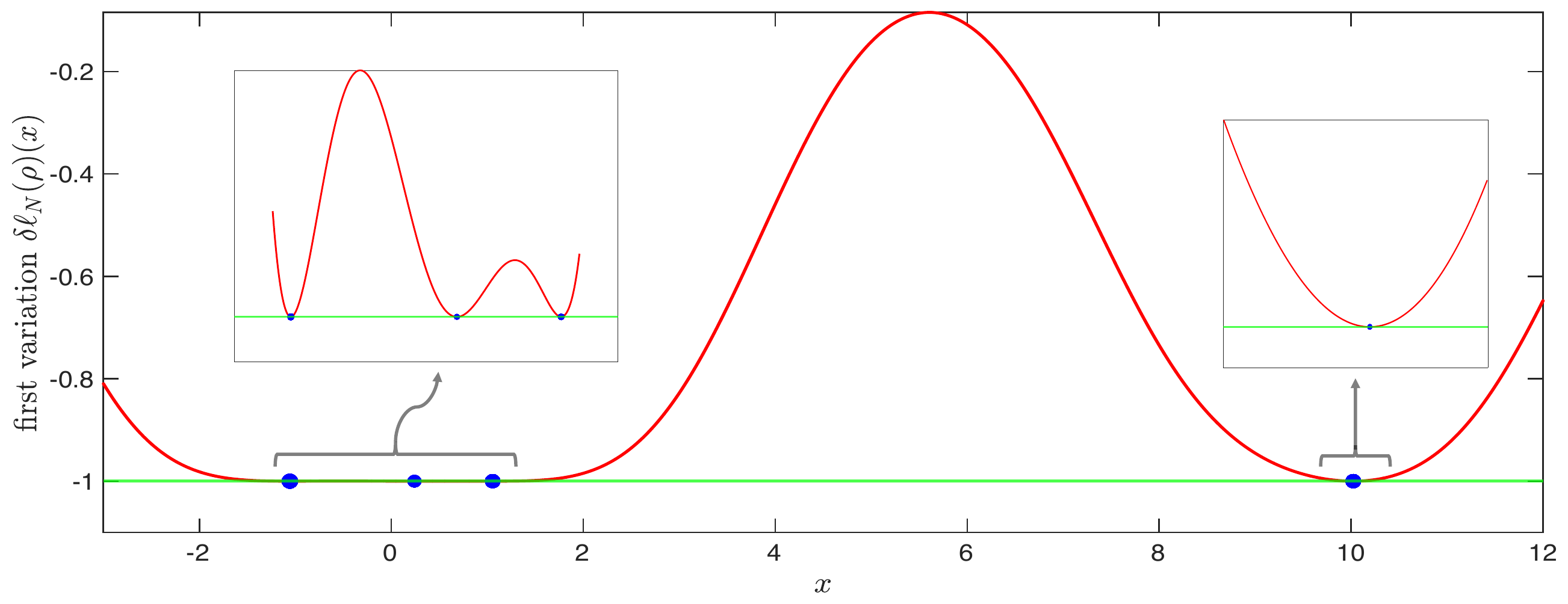}\tabularnewline
\qquad(b) first variation $\delta\ell_{N}(\rho)$ for continuous
$\rho^{\star}=\mathcal{N}(0,1)$\tabularnewline
\includegraphics[scale=0.55]{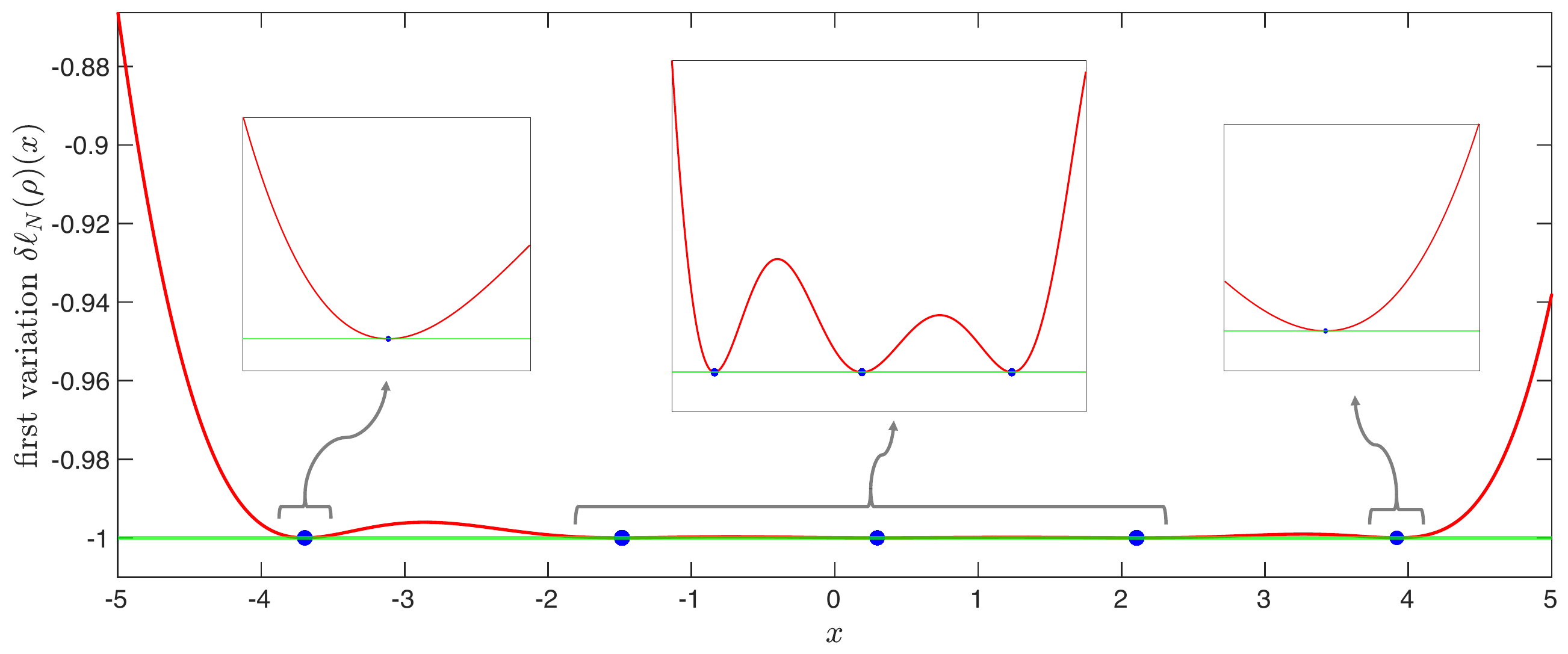}\tabularnewline
\end{tabular}

\caption{The first variation $\delta\ell_{N}(\rho)$ (in red line) for discrete
$\rho^{\star}=\rho_{\mathsf{d}}$ in the upper panel and continuous
$\rho^{\star}=\rho_{\mathsf{c}}$ in the lower panel. The blue dots
are $\{(\mu_{j},\delta\ell_{N}(\rho)(\mu_{j})):1\protect\leq j\protect\leq m\}$,
where the size of these dots are proportional to the weights $\{\omega_{j}\}_{1\protect\leq j\protect\leq m}$.
The green line is a horizontal line $y=-1$. The two subfigures in
the upper panel zoom in the two regions $-1.5\protect\leq x\protect\leq1.5$
and $9.8\protect\leq x\protect\leq10.2$, and the three subfigures
in the lower panel zoom in the three regions $-4\protect\leq x\protect\leq-3.4$,
$-2\protect\leq x\protect\leq3$ and $3.8\protect\leq x\protect\leq4.1$.
\label{fig:Exp3_fv}}
\end{figure}

The convergence guarantees in this paper only cover the infinite-particle
limit of Wasserstein-Fisher-Rao gradient descent (cf.~Algorithm \ref{alg:WFR-GD}).
In this section, we provide numerical evidence that Algorithm \ref{alg:WFR-GD}
converges to (approximate) NPMLE when we use a large number of particles. 

Recall from Theorem \ref{thm:NPMLE-basics} and the following remark that $\rho\in\mathcal{P}(\mathbb{R}^{d})$
is NPMLE if and only if $\delta\ell_{N}(\rho)(x)\geq-1$ holds for
all $x\in\mathbb{R}^{d}$. We focus on the one-dimensional setting (i.e.~$d=1$)
since it is computationally affordable to check the function value
of $\delta\ell_{N}(\rho)(x)$ in one dimension (e.g.~over a fine grid). For a discrete distribution
$\rho=\sum_{j=1}^{m}\omega_{j}\delta_{\mu_{j}}$, we define the following suboptimality gaps:
\[
\mathsf{gap}\left(\rho\right) \coloneqq\sup_{x\in\mathbb{R}}\max\left\{ -1-\delta\ell_{N}\left(\rho\right)\left(x\right),0\right\},\qquad
\widehat{\mathsf{gap}}(\rho)\coloneqq\max_{x\in\mathsf{grid}(\rho)}\max\left\{ -1-\delta\ell_{N}\left(\rho\right)\left(x\right),0\right\}.
\label{eq:suboptimality-gap}
\]
It is clear that when the first optimality gap $\mathsf{gap}(\rho)=0$, Theorem \ref{thm:NPMLE-basics} asserts that $\rho$ is
the NPMLE. However $\mathsf{gap}(\rho)$ is in general difficult to compute, and a practical scheme is to approximatly evaluate the supremum over $\mathbb{R}$ by the maximum over a fine grid $\mathsf{grid}(\rho)\subseteq\mathbb{R}$, which gives the second optimality gap $\widehat{\mathsf{gap}}$. In our experiments, we take $\mathsf{grid}(\rho)$ to be a $0.01$-net over $[\min_{1\leq j\leq m} \mu_j-1,\max_{1\leq j \leq m}\mu_j +1]$.

We set the sample size $N=1500$, the dimension $d=1$ and run Wasserstein-Fisher-Rao
gradient descent (cf.~Algorithm \ref{alg:WFR-GD}) with number of
particles $m=500$, step sizes $\eta_{1}=\eta_{2}=10^{-1}$. Figure
\ref{fig:optimality} illustrates the suboptimality gap $\widehat{\mathsf{gap}}$ in (\ref{eq:suboptimality-gap})
(with one standard deviation error bars) vs.~the iteration count
over $20$ independent trials as well as the density plot of the output
of the algorithm convolved with $\mathcal{N}(0,1)$. Roughly speaking,
both suboptimality gaps decreases inversely proportional to the iteration
counts. Lastly, Figure \ref{fig:Exp3_fv} depicts the first variation
$\delta\ell_{N}(\rho)$, which clearly shows that the optimality condition
is approximately satisfied with high precision; see the caption of
Figure \ref{fig:Exp3_fv} for more details.

%% file: discussion.tex
\section{Discussion}

The current paper proposes to solve the NPMLE for Gaussian mixtures
using an interacting particle system driven by Wasserstein-Fisher-Rao
gradient descent. In the infinite-particle limit, we show that the
proposed algorithm converges to NPMLE under certain conditions. In
practice, we conduct extensive numerical experiments to illustrate
the capability of the proposed algorithm in exactly computing NPMLE
using a finite (or even small) number of particles, and also to demonstrate
the superiority of the proposed algorithm compared to other algorithms.
Moving forward, there are numerous possible extensions that merit
future investigation. For example, our convergence theory only holds
when the Wasserstein-Fisher-Rao gradient descent is initialized from
a distribution supported on the whole space; it would be of interest to extend the current
analysis to the more practical scenario where the algorithm is initialized
from a discrete distribution with a finite number of particles. Another
interesting direction is to develop algorithms for learning Gaussian
mixtures beyond the isotropic case (i.e.~without assuming that the
covariance matrices are identity) using, for example, Wasserstein-Fisher-Rao
gradient flow over the Bures-Wasserstein space \citep{lambert2022variational}.

%% file: appendix_prelim.tex
\section{Preliminaries}

In the main text, we focused on the Gaussian mixture model where $\phi(x)=(2\pi)^{-d/2}\exp(-\Vert x\Vert_{2}^{2}/2)$.
In fact, the algorithms and theorems in the current paper are also
valid for a more general class of probability density $\phi$. In
the appendices, we only assume that $\phi$ satisfies the following
regularity assumption. 

\begin{assumption}[Regularity]\label{assumption-pdf-0} Assume that
the density $\phi(x)>0$ for any $x\in\mathbb{R}^{d}$. Furthermore,
$\phi\in C^{\max\{d,2\}}(\RR^{d})$, $\lim_{\|x\|_{2}\to\infty}\phi(x)=0$,
$\sup_{x\in\mathbb{R}^{d}}\|\nabla\phi(x)\|_{2}<\infty$ and $\sup_{x\in\mathbb{R}^{d}}\|\nabla^{2}\phi(x)\|_{2}<\infty$.
\end{assumption}

It is clear that the Gaussian kernel $\phi(x)=(2\pi)^{-d/2}\exp(-\Vert x\Vert_{2}^{2}/2)$
satisfies Assumption \ref{assumption-pdf-0}. We also define the following
sets and quantites that will be useful throughout the proof.

\begin{definition}Let $\Omega=\mathsf{conv}(\{X_{i}\}_{1\leq i\leq N})$.
For $r\geq0$, define $\Omega_{r}=\{x\in\mathbb{R}^{d}:\,\mathrm{dist}(x,\Omega)\leq r\}$,
$\bar{\phi}(r)=\sup_{\|x\|_{2}\geq r}\phi(x)$ and $\underline{\phi}(r)=\inf_{\|x\|_{2}\leq r}\phi(x)$.
\end{definition}

The following lemma shows that NPMLE is compactly supported. The proof
can be found in Appendix \ref{sec-lem-npmle-support-proof}.

\begin{lemma}[Compact support of NPMLE]\label{lem-npmle-support}
Let Assumption \ref{assumption-pdf-0} hold and $\widehat{\rho}$
be any optimal solution to \eqref{eq:NPMLE}. Define $R_{1}=\inf\{r\geq0:~\bar{\phi}(r)\leq\underline{\phi}[\mathrm{diam}(\Omega)]/2\}$
and 
\[
R=\inf\bigg\{ r\geq0:~\bar{\phi}(r)\leq\frac{\bar{\phi}(R_{1})\underline{\phi}(R_{1}+\mathrm{diam}(\Omega))}{8\bar{\phi}(0)}\bigg\}.
\]
Then, we have $\mathsf{supp}(\widehat{\rho})\subseteq\Omega_{R}$. 

\end{lemma} 

\section{Proof of structual results for NPMLE}

In this section, we prove the two structural results for NPMLE, namely
Theorem \ref{thm:NPMLE-basics} and Lemma \ref{lem-npmle-support} under Assumption \ref{assumption-pdf-0}.

\subsection{Proof of Theorem \ref{thm:NPMLE-basics}\label{sec:proof-thm-NPMLE-basics}}

\paragraph{Part 1: existence of NPMLE.}

Note that the loss function $\ell_{N}$ is lower bounded
\begin{equation}
\ell_{N}\left(\rho\right)=-\frac{1}{N}\sum_{i=1}^{N}\log\left[(\rho*\phi)\left(X_{i}\right)\right]\geq-\log\left\Vert \phi\right\Vert _{\infty},\label{eq:loss_lower_bounded}
\end{equation}
where the last inequality follows from $\rho*\phi(x)=\int_{\mathbb{R}^{d}}\phi\left(y-x\right)\rho\left(\mathrm{d}y\right)\leq\Vert\phi\Vert_{\infty}$
for any $x\in\mathbb{R}^{d}$. Therefore there exists a sequence of
probability distribution $\{\rho_{n}\}$ such that
\begin{equation}
\ell_{N}\left(\rho_{n}\right)\leq\inf_{\rho\in\mathcal{P}(\mathbb{R}^{d})}\ell_{N}\left(\rho\right)+\frac{1}{n}.\label{eq:sequence-approach-inf}
\end{equation}
Now we argue that $\{\rho_{n}\}$ is tight. To that end, we show that there exists $r>0$ such that for any $\varepsilon>0$, it holds $\rho_n(\Omega_r)\geq 1-\varepsilon$ for $n$ large enough.
% we will find a compact set $K_{\varepsilon}$ such that $\rho_{n}(K_{\varepsilon})\geq1-\varepsilon$
% for all $n\geq1$. 

For any $n$ and $r>0$,  define
\[
\rho_{n,r}\coloneqq\rho_n\left(\Omega_{r}\right)\cdot\rho_{n}|_{\Omega_{r}}+\rho_n\left(\Omega_{r}^{\mathrm{c}}\right)\cdot\mathsf{Unif}\left(\Omega\right),
\]
where $\rho_{n}|_{\Omega_{r}}(\cdot)=\rho_n(\cdot|\Omega_r)$ is the conditional  distribution of $\rho_{n}$
given $\Omega_{r}$, and $\mathsf{Unif}(\Omega)$ is the uniform distribution
on $\Omega$. We have
\begin{align*}
\ell_{N}\left(\rho_{n}\right)-\ell_{N}\left(\rho_{n,r}\right) & =-\frac{1}{N}\sum_{i=1}^{N}\log\left[(\rho_{n}*\phi)\left(X_{i}\right)\right]+\frac{1}{N}\sum_{i=1}^{N}\log\left[(\rho_{n,r}*\phi)\left(X_{i}\right)\right]=\frac{1}{N}\sum_{i=1}^{N}\log\left[\frac{(\rho_{n,r}*\phi)\left(X_{i}\right)}{(\rho_{n}*\phi)\left(X_{i}\right)}\right].
\end{align*}
Note that for each $i\in[N]$
\begin{align*}
\log\left[\frac{(\rho_{n,r}*\phi)\left(X_{i}\right)}{(\rho_{n}*\phi)\left(X_{i}\right)}\right] & =\log\left[\frac{\int_{\Omega_{r}}\phi\left(X_{i}-y\right)\rho_{n}\left(\mathrm{d}y\right)+\rho_{n}\left(\Omega_{r}^{\mathrm{c}}\right)\int_{\Omega}\phi\left(X_{i}-y\right)\mathrm{d}y}{\int_{\Omega_{r}}\phi\left(X_{i}-y\right)\rho_{n}\left(\mathrm{d}y\right)+\int_{\Omega_{r}^{\mathrm{c}}}\phi\left(X_{i}-y\right)\rho_{n}\left(\mathrm{d}y\right)}\right]\\
 & =\log\left[1+\frac{\rho_{n}\left(\Omega_{r}^{\mathrm{c}}\right)\int_{\Omega}\phi\left(X_{i}-y\right)\mathrm{d}y-\int_{\Omega_{r}^{\mathrm{c}}}\phi\left(X_{i}-y\right)\rho_{n}\left(\mathrm{d}y\right)}{(\rho_{n}*\phi)\left(X_{i}\right)}\right]\\
 & \geq\log\left[1+\frac{\rho_{n}\left(\Omega_{r}^{\mathrm{c}}\right)\left[\underline{\phi}\left(\mathrm{diam}\left(\Omega\right)\right)-\overline{\phi}\left(r\right)\right]}{\left\Vert \phi\right\Vert _{\infty}}\right].
\end{align*}
We can choose $r>0$ to be sufficiently large so that $\overline{\phi}(r)\leq\underline{\phi}(\mathrm{diam}(\Omega))/2$,
and therefore for each $i\in[N]$
\[
\ell_{N}\left(\rho_{n}\right)-\ell_{N}\left(\rho_{n,r}\right)\geq
% \log\left[1+\frac{\rho_{n}\left(\Omega_{r}^{\mathrm{c}}\right)\left[\underline{\phi}\left(\mathrm{diam}\left(\Omega\right)\right)-\overline{\phi}\left(r\right)\right]}{\left\Vert \phi\right\Vert _{\infty}}\right]\geq
\log\left[1+\frac{\rho_{n}\left(\Omega_{r}^{\mathrm{c}}\right)\underline{\phi}\left(\mathrm{diam}\left(\Omega\right)\right)}{2\left\Vert \phi\right\Vert _{\infty}}\right].
\]
In view of  \eqref{eq:sequence-approach-inf}, we know that 
\[
\ell_{N}\left(\rho_{n}\right)-\ell_{N}\left(\rho_{n,r}\right)\leq\frac{1}{n}.
\]
Taking the above two inequalities collectively give
\[
\rho_{n}\left(\Omega_{r}^{\mathrm{c}}\right)\leq\frac{2\left\Vert \phi\right\Vert _{\infty}\left[\exp\left(\frac{1}{n}\right)-1\right]}{\underline{\phi}\left(\mathrm{diam}\left(\Omega\right)\right)}.
\]
Therefore we have
\[
\rho_{n}\left(\Omega_{r}^{\mathrm{c}}\right)\leq\varepsilon\qquad\text{as long as}\qquad n\geq n_{\varepsilon}\coloneqq\left\lceil 1/\log\left[1+\frac{\varepsilon\underline{\phi}\left(\mathrm{diam}\left(\Omega\right)\right)}{2\left\Vert \phi\right\Vert _{\infty}}\right]\right\rceil ,
\]
which implies that $\{\rho_{n}\}$ is tight. We conclude using Prokhorov's
theorem: there exists a subsequence $\{\rho_{n_{k}}\}$
and $\widehat{\rho}\in\mathcal{P}(\mathbb{R}^{d})$ such that $\rho_{n_{k}}$
converges weakly to $\widehat{\rho}$ which must be a 
minimizer of  \eqref{eq:NPMLE}. 

\paragraph{Part 2: optimality condition.}

First of all, it is straightforward to check that for any $\rho\in\mathcal{M}(\mathbb{R}^{d})$
\[
\int_{\mathbb{R}^{d}}\delta\ell_{N}\left(\rho\right)\left(x\right)\rho\left(\mathrm{d}x\right)=-\frac{1}{N}\sum_{i=1}^{N}\frac{\int\phi\left(x-X_{i}\right)\rho\left(\mathrm{d}x\right)}{(\rho*\phi)\left(X_{i}\right)}=-\frac{1}{N}\sum_{i=1}^{N}\frac{(\rho*\phi)\left(X_{i}\right)}{(\rho*\phi)\left(X_{i}\right)}=-1.
\]

If $\widehat{\rho}\in\mathcal{M}(\mathbb{R}^{d})$ is the optimal
solution to  \eqref{eq:NPMLE}, then for any $x\in\mathbb{R}^{d}$
and any $\varepsilon\in[0,1]$ we have 
\[
\ell_{N}\left(\widehat{\rho}\right)\leq\ell_{N}\left(\left(1-\varepsilon\right)\widehat{\rho}+\varepsilon\delta_{x}\right)=\ell_{N}\left(\rho+\varepsilon\left(\delta_{x}-\widehat{\rho}\right)\right).
\]
As a result we have
\[
\delta\ell_{N}\left(\widehat{\rho}\right)\left(x\right)+1=\int_{\mathbb{R}^{d}}\delta\ell_{N}\left(\widehat{\rho}\right)\mathrm{d}\left(\delta_{x}-\widehat{\rho}\right)=\lim_{\varepsilon\to0}\frac{1}{\varepsilon}\left[\ell_{N}\left(\widehat{\rho}+\varepsilon\left(\delta_{x}-\widehat{\rho}\right)\right)-\ell_{N}\left(\widehat{\rho}\right)\right]\geq0.
\]
Since $x$ is arbitrary, this implies that $\delta\ell_{N}(\widehat{\rho})(x)\geq-1$
for any $x\in\mathbb{R}^{d}$. Combine this with $\int_{\mathbb{R}^{d}}\delta\ell_{N}(\widehat{\rho})\mathrm{d}\widehat{\rho}=-1$
readily gives $\delta\ell_{N}(\widehat{\rho})(x)=-1$ for $\widehat{\rho}$-a.e.~$x$. 

Conversely, if $\widehat{\rho}\in\mathcal{M}(\mathbb{R}^{d})$ satisfies $\delta\ell_{N}(\widehat{\rho})(x)\geq-1$
for all $x\in\mathbb{R}^{d}$, then for any $\rho\in\mathcal{M}(\mathbb{R}^{d})$, it holds
\[
0\le \int_{\mathbb{R}^{d}}\delta\ell_{N}\left(\widehat{\rho}\right)\mathrm{d}\rho+1 
= \int_{\mathbb{R}^{d}}\delta\ell_{N}\left(\widehat{\rho}\right)\mathrm{d}\left(\rho-\widehat{\rho}\right)
=\lim_{\varepsilon\to0}\frac{1}{\varepsilon}\left[\ell_{N}\left(\widehat{\rho}+\varepsilon\left(\rho-\widehat{\rho}\right)\right)-\ell_{N}\left(\widehat{\rho}\right)\right]
\le 
\ell_{N}\left(\rho\right)-\ell_{N}\left(\widehat{\rho}\right)
\]

% \[
% \ell_{N}\left(\rho\right)-\ell_{N}\left(\widehat{\rho}\right)\geq\lim_{\varepsilon\to0}\frac{1}{\varepsilon}\left[\ell_{N}\left(\widehat{\rho}+\varepsilon\left(\rho-\widehat{\rho}\right)\right)-\ell_{N}\left(\widehat{\rho}\right)\right]=\int_{\mathbb{R}^{d}}\delta\ell_{N}\left(\widehat{\rho}\right)\mathrm{d}\left(\rho-\widehat{\rho}\right)=\int_{\mathbb{R}^{d}}\delta\ell_{N}\left(\widehat{\rho}\right)\mathrm{d}\rho+1\geq0,
% \]
where the last inequality follows from convexity of the functional $\rho \mapsto \ell_{N}(\rho)$.
% is $\ell_{2}$-convex
% in $\rho$, and the last inequality follows from $\delta\ell_{N}(\widehat{\rho})(x)\geq-1$
% for all $x\in\mathbb{R}^{d}$.
The above display yields that that $\widehat{\rho}$ is
a global minimizer of  \eqref{eq:NPMLE}. 

\subsection{Proof of Lemma \ref{lem-npmle-support} \label{sec-lem-npmle-support-proof}}

By Theorem \ref{thm:NPMLE-basics}, $\delta\ell_{N}(\widehat{\rho})=-1$
over $\mathsf{supp}(\widehat{\rho})$. We will show that $|\delta\ell_{N}(\widehat{\rho})(y)|<1/2$
when $y$ is too far away from $\Omega$, and then conclude that $\mathsf{supp}(\widehat{\rho})$
must stay close to $\Omega$. To that end, we present some useful
estimates in the following lemma that will also be useful later. 

\begin{lemma}\label{lem-loss-bounds} Let Assumption \ref{assumption-pdf-0}
hold. For any $\rho\in\mathcal{P}(\mathbb{R}^{d})$, $x\in\Omega$
and $r\geq0$, we have 
\[
\rho(\Omega_{r})\underline{\phi}(r+\mathrm{diam}(\Omega))\leq(\rho*\phi)(x)\leq\rho(\Omega_{r}^{\mathrm{c}})\bar{\phi}(r)+\rho(\Omega_{r})\bar{\phi}(0)\leq\bar{\phi}(r)+\rho(\Omega_{r})\bar{\phi}(0).
\]
As a result, 
\begin{align*}
-\log\Big(\bar{\phi}(r)+\rho(\Omega_{r})\bar{\phi}(0)\Big)\leq\ell_{N}(\rho)\leq-\log\Big(\rho(\Omega_{r})\underline{\phi}(r+\mathrm{diam}(\Omega))\Big).
\end{align*}
Take any $R\geq0$ such that $\bar{\phi}(R)\leq e^{-\ell_{N}(\rho)}/2$.
For any $\mu\in\mathcal{P}(\mathbb{R}^{d})$ with $\ell_{N}(\mu)\leq\ell_{N}(\rho)$,
we have 
\begin{align*}
 & \mu(\Omega_{R})\geq e^{-\ell_{N}(\rho)}/[2\bar{\phi}(0)],\\
 & \inf_{x\in\Omega}(\mu*\phi)(x)\geq e^{-\ell_{N}(\rho)}\underline{\phi}(R+\mathrm{diam}(\Omega))/[2\bar{\phi}(0)],\\
 & \sup_{y\in\Omega_{r}^{c}}|\delta\ell_{N}(\mu)(y)|\leq\frac{2e^{\ell_{N}(\rho)}\bar{\phi}(0)}{\underline{\phi}(R+\mathrm{diam}(\Omega))}\cdot\bar{\phi}(r),\qquad\forall r\geq0.
\end{align*}
\end{lemma}

The proof of Lemma \ref{lem-loss-bounds} is deferred to the end of
this section. We come back to proving Lemma \ref{lem-npmle-support}.
Choose any $\rho_{0}\in\mathcal{P}(\mathbb{R}^{d})$ supported on
$\Omega$. By Lemma \ref{lem-loss-bounds}, we have
\begin{align*}
\ell_{N}(\rho_{0})\leq-\log\Big(\underline{\phi}[\mathrm{diam}(\Omega)]\Big),\qquad\forall r\geq0.
\end{align*}
Take $R_{1}=\inf\{r\geq0:~\bar{\phi}(r)\leq\underline{\phi}[\mathrm{diam}(\Omega)]/2\}$.
By the continuity of $\bar{\phi}$, 
\[
\bar{\phi}(R_{1})\leq\underline{\phi}[\mathrm{diam}(\Omega)]/2\leq e^{-\ell_{N}(\rho_{0})}/2.
\]
Lemma \ref{lem-loss-bounds} implies that 
\begin{align*}
\sup_{y\in\Omega_{r}^{\mathrm{c}}}|\delta\ell_{N}(\widehat{\rho})(y)|\leq\frac{2e^{\ell_{N}(\rho_{0})}\bar{\phi}(0)}{\underline{\phi}(R_{1}+\mathrm{diam}(\Omega))}\cdot\bar{\phi}(r)\leq\frac{4\bar{\phi}(0)}{\bar{\phi}(R_{1})\underline{\phi}(R_{1}+\mathrm{diam}(\Omega))}\cdot\bar{\phi}(r),\qquad\forall r\geq0.
\end{align*}
Let 
\[
R=\inf\bigg\{ r\geq0:~\bar{\phi}(r)\leq\frac{\bar{\phi}(R_{1})\underline{\phi}(R_{1}+\mathrm{diam}(\Omega))}{8\bar{\phi}(0)}\bigg\}.
\]
The continuity of $\bar{\phi}$ leads to $\bar{\phi}(R)\leq\bar{\phi}(R_{1})\underline{\phi}(R_{1}+\mathrm{diam}(\Omega))/[8\bar{\phi}(0)]$
and thus $\sup_{y\in\Omega_{R}^{\mathrm{c}}}|\delta\ell_{N}(\widehat{\rho})(y)|\leq1/2<1$.
As a result, $\mathsf{supp}(\widehat{\rho})\cap\Omega_{R}^{\mathrm{c}}=\varnothing$
and $\mathsf{supp}(\widehat{\rho})\subseteq\Omega_{R}$.

\begin{proof}[Proof of Lemma \ref{lem-loss-bounds}]

Note that $\phi(x-y)\leq\bar{\phi}(0)$ for any $x,y\in\mathbb{R}^{d}$.
If $x\in\Omega$ and $y\in\Omega_{r}^{\mathrm{c}}$, then $\|x-y\|_{2}\geq r$
and thus $\phi(x-y)\leq\bar{\phi}(r)$. Hence, 
\begin{align*}
(\rho*\phi)(x) & \leq\int_{\Omega_{r}}\phi(x-y)\rho(\mathrm{d}y)+\int_{\Omega_{r}^{c}}\phi(x-y)\rho(\mathrm{d}y)\leq\bar{\phi}(0)\rho(\Omega_{r})+\bar{\phi}(r)\rho(\Omega_{r}^{\mathrm{c}}).
\end{align*}
If $x\in\Omega$ and $y\in\Omega_{r}$, then $\|x-y\|_{2}\geq r+\mathrm{diam}(\Omega)$
and thus $\phi(x-y)\geq\underline{\phi}(r+\mathrm{diam}(\Omega))$.
Therefore, 
\begin{align*}
(\rho*\phi)(x) & \geq\int_{\Omega_{r}}\phi(x-y)\rho(\mathrm{d}y)\geq\underline{\phi}(r+\mathrm{diam}(\Omega))\rho(\Omega_{r}).
\end{align*}
The desired bounds on $\ell_N(\cdot)$ become obvious. If $\mu\in\mathcal{P}(\mathbb{R}^{d})$
and $\ell_{N}(\mu)\leq\ell_{N}(\rho)$, then our estimates of $\ell$
implies that 
\begin{align*}
 & -\log\Big(\bar{\phi}(r)+\mu(\Omega_{r})\bar{\phi}(0)\Big)\leq\ell_{N}(\mu)\leq\ell_{N}(\rho),\qquad\forall r\geq0.
\end{align*}
Hence, $\bar{\phi}(r)+\mu(\Omega_{r})\bar{\phi}(0)\geq e^{-\ell_{N}(\rho)}$.
By Assumption \ref{assumption-pdf-0}, $\lim\limits _{r\to\infty}\bar{\phi}(r)=0$.
Take any $R\geq0$ such that $\bar{\phi}(R)\leq e^{-\ell_{N}(\rho)}/2$.
Then, 
\begin{align*}
 & \mu(\Omega_{R})\geq[e^{-\ell_{N}(\rho)}-\bar{\phi}(r)]/\bar{\phi}(0)\geq e^{-\ell_{N}(\rho)}/[2\bar{\phi}(0)],\\
 & \inf_{x\in\Omega}(\mu*\phi)(x)\geq\mu(\Omega_{R})\underline{\phi}(R+\mathrm{diam}(\Omega))\geq e^{-\ell_{N}(\rho)}\underline{\phi}(R+\mathrm{diam}(\Omega))/[2\bar{\phi}(0)].
\end{align*}

For any $r\geq0$, we have $\|X-y\|_{2}\geq r$ whenever $X\in\mathsf{supp}(\nu)$
and $y\in\Omega_{r}^{\mathrm{c}}$. Then, 
\begin{align*}
|\delta\ell_{N}(\mu)(y)|=\frac{1}{N}\sum_{i=1}^{N}\frac{\phi(X_{i}-y)}{(\mu*\phi)(X_{i})}\leq\frac{\bar{\phi}(r)}{\inf_{x\in\Omega}(\mu*\phi)(x)},\qquad\forall r\geq0,~~y\in\Omega_{r}^{\mathrm{c}}.
\end{align*}
The proof is finished by combining this and the lower bound on $\inf_{x\in\Omega}(\mu*\phi)(x)$
we have established above.

\end{proof}

%% file: appendix_gf_derivation.tex
\section{Derivation of gradient flows over $\mathcal{P}_{2}(\mathbb{R}^{d})$}

\subsection{First variation \label{subsec:First-variation}}

Recall that the population and finite-sample loss functions are 
\[
\ell_{\infty}\left(\rho\right)=\mathbb{E}_{X\sim(\rho^{\star}*\phi)}\left\{ \log\left[(\rho*\phi)\left(X\right)\right]\right\} ,\qquad\ell_{N}\left(\rho\right)=-\frac{1}{N}\sum_{i=1}^{N}\log\left[(\rho*\phi)\left(X_{i}\right)\right].
\]
The first variation of $\ell_{N}$ is defined as nay measurable function $\delta\ell(\rho):\mathbb{R}^d\to\mathbb{R}$
satisfying 
\[
\lim_{\varepsilon\to0}\frac{\ell\left(\rho+\varepsilon\mathcal{X}\right)-\ell\left(\rho\right)}{\varepsilon}=\int\delta\ell\left(\rho\right)\mathrm{d}\mathcal{X}
\]
for all signed measures $\mathcal{X}$ satisfying $\int\mathrm{d}\mathcal{X}=0$. In particular, it is easy to see that the first variation is defined up to an additive constant.

By direct computation, we have 
\begin{align*}
	\lim_{\varepsilon\to0}\frac{\ell_{N}\left(\rho+\varepsilon\mathcal{X}\right)-\ell_{N}\left(\rho\right)}{\varepsilon} & =-\frac{1}{N}\sum_{i=1}^{N}\lim_{\varepsilon\to0}\frac{\log\left[[\left(\rho+\varepsilon\mathcal{X}\right)*\phi]\left(X_{i}\right)\right]-\log\left[(\rho*\phi)\left(X_{i}\right)\right]}{\varepsilon}\\
	& =-\frac{1}{N}\sum_{i=1}^{N}\lim_{\varepsilon\to0}\frac{1}{\varepsilon}\log\left[1+\varepsilon\frac{(\mathcal{X}*\phi)\left(X_{i}\right)}{(\rho*\phi)\left(X_{i}\right)}\right]=-\frac{1}{N}\sum_{i=1}^{N}\frac{(\mathcal{X}*\phi)\left(X_{i}\right)}{(\rho*\phi)\left(X_{i}\right)}\\
	& =-\frac{1}{N}\sum_{i=1}^{N}\int\frac{\phi\left(x-X_{i}\right)}{(\rho*\phi)\left(X_{i}\right)}\mathcal{X}\left(\mathrm{d}x\right),
\end{align*}
As a result, we have 
\[
\delta\ell_{N}\left(\rho\right):x\to-\frac{1}{N}\sum_{i=1}^{N}\frac{\phi\left(x-X_{i}\right)}{(\rho*\phi)\left(X_{i}\right)}.
\]
Similarly, we can also compute 
\begin{align*}
	\lim_{\varepsilon\to0}\frac{\ell_{\infty}\left(\rho+\varepsilon\mathcal{X}\right)-\ell_{\infty}\left(\rho\right)}{\varepsilon} & =\lim_{\varepsilon\to0}\frac{1}{\varepsilon}\left[-\int\log\left[\frac{\left(\rho+\varepsilon\mathcal{X}\right)*\phi\left(x\right)}{(\rho*\phi)\left(x\right)}\right](\rho^{\star}*\phi)\left(x\right)\mathrm{d}x\right]\\
	& =-\int\frac{(\mathcal{X}*\phi)\left(x\right)}{(\rho*\phi)\left(x\right)}(\rho^{\star}*\phi)\left(x\right)\mathrm{d}x=-\int\frac{(\rho^{\star}*\phi)\left(x\right)}{(\rho*\phi)\left(x\right)}\left[\int\phi\left(x-y\right)\mathcal{X}\left(\mathrm{d}y\right)\right]\mathrm{d}x\\
	& =-\int\left[\int\frac{(\rho^{\star}*\phi)\left(x\right)}{(\rho*\phi)\left(x\right)}\phi\left(x-y\right)\mathrm{d}x\right]\mathcal{X}\left(\mathrm{d}y\right).
\end{align*}
which gives 
\[
\delta\ell_{\infty}\left(\rho\right):x\to-\int\frac{(\rho^{\star}*\phi)\left(y\right)}{(\rho*\phi)\left(y\right)}\phi\left(x-y\right)\mathrm{d}y.
\]

\subsection{Fisher-Rao gradient flow }

\subsubsection{A formal derivation of gradient flow using Riemannian geometry\label{subsec:FR-derivation}}

We first introduce a Riemannian structure over $\mathcal{P}_{2}(\mathbb{R}^{d})$
underlying the Fisher-Rao metric. Define the tangent space at $\rho\in\mathcal{P}_{2}(\mathbb{R}^{d})$
as 
\[
\mathrm{Tan}_{\rho}^{\mathsf{FR}}\mathcal{P}_{2}(\mathbb{R}^{d})\coloneqq\left\{ \zeta:\zeta=\rho\left(\alpha-\int\alpha\mathrm{d}\rho\right)\text{ for some }\alpha\text{ satisfying }\int\alpha^{2}\mathrm{d}\rho<\infty\right\} .
\]
We equip the tangent space $T_{\rho}^{\mathsf{FR}}\mathcal{P}_{2}(\mathbb{R}^{d})$
with the following Riemannian metric tensor $g_{\rho}^{\mathsf{FR}}(\cdot,\cdot):\mathrm{Tan}_{\rho}^{\mathsf{FR}}\mathcal{P}_{2}(\mathbb{R}^{d})\times\mathrm{Tan}_{\rho}^{\mathsf{FR}}\mathcal{P}_{2}(\mathbb{R}^{d})\to\mathbb{R}$
as 
\begin{align*}
	g_{\rho}^{\mathsf{FR}}\left(\zeta_{1},\zeta_{2}\right) & \coloneqq \int \frac{\zeta_1 \cdot\zeta_2}{\rho^2} \ud \rho\\
	&=\int_{\mathbb{R}^{d}}\left[\alpha_{1}\left(x\right)-\int_{\mathbb{R}^{d}}\alpha_{1}\mathrm{d}\rho\right]\left[\alpha_{2}\left(x\right)-\int_{\mathbb{R}^{d}}\alpha_{2}\mathrm{d}\rho\right]\rho\left(\mathrm{d}x\right)\\
	& =\int_{\mathbb{R}^{d}}\alpha_{1}\left(x\right)\alpha_{2}\left(x\right)\rho\left(\mathrm{d}x\right)-\left(\int_{\mathbb{R}^{d}}\alpha_{1}\mathrm{d}\rho\right)\left(\int_{\mathbb{R}^{d}}\alpha_{2}\mathrm{d}\rho\right),
\end{align*}
for any $\zeta_{1}=\rho(\alpha_{1}-\int\alpha_{1}\mathrm{d}\rho)$
and $\zeta_{2}=\rho(\alpha_{2}-\int\alpha_{2}\mathrm{d}\rho)$. The
metric induced by this Riemannian structure, namely the Fisher-Rao
metric $d_{\mathsf{FR}}(\cdot,\cdot)$, satisfies the following property:
\begin{align*}
	d_{\mathsf{FR}}^{2}\left(\rho_{0},\rho_{1}\right) & =\inf\bigg\{\int_{0}^{1}\int\Big[\Big(\alpha_{t}-\int\alpha_{t}\mathrm{d}\rho_{t}\Big)^{2}\Big]\mathrm{d}\rho_{t}\mathrm{d}t:\left(\rho_{t},\alpha_{t}\right)_{t\in[0,1]}\text{ solves }\partial_{t}\rho_{t}=\rho_{t}\alpha_{t}\bigg\}.
\end{align*}
Then we follow \citet{lu2019accelerating,gallouet2017jko} to derive the Fisher-Rao gradient flow with respect to the functional
$\ell_{N}$. Let $(\rho_{t})_{t\geq0}$ be a $C^{1}$ curve satisfying
$\rho_{0}=\rho$ with initial velocity 
\[
\partial_{t}\rho_{t}|_{t=0}=\zeta=\rho\left(\alpha-\int\alpha\mathrm{d}\rho\right).
\]
The Fisher-Rao gradient of $\ell_N$ at $\rho$ is defined as the function $\mathrm{grad}_{\mathsf{FR}}\ell_{N}\left(\rho\right) \in L^2(\rho)$ such that
\[
\frac{\mathrm{d}}{\mathrm{d}t}\ell_{N}\left(\rho_{t}\right)\Big|_{t=0}=g_{\rho}^{\mathsf{FR}}\left(\mathrm{grad}_{\mathsf{FR}}\ell_{N}\left(\rho\right),\zeta\right).
\]
To compute it, observe that the right-hand side of the above identity is given by
\begin{align*}
	\frac{\mathrm{d}}{\mathrm{d}t}\ell_{N}\left(\rho_{t}\right)\Big|_{t=0} 
	&= \int \delta \ell_N (\rho) \cdot\partial_t\rho_t\Big|_{t=0}\\
	& =\int\delta\ell_{N}\left(\rho\right),\zeta=\int\delta\ell_{N}\left(\rho\right)\left(\alpha-\int\alpha\mathrm{d}\rho\right)\mathrm{d}\rho\\
	& =\int\left(\delta\ell_{N}\left(\rho\right)-\int\delta\ell_{N}\left(\rho\right)\mathrm{d}\rho\right)\left(\alpha-\int\alpha\mathrm{d}\rho\right)\mathrm{d}\rho\\
	& =g_{\rho}^{\mathsf{FR}}\left(\rho\left(\delta\ell_{N}\left(\rho\right)-\int\delta\ell_{N}\left(\rho\right)\mathrm{d}\rho\right),\zeta\right).
\end{align*}
Therefore 
\[
g_{\rho}^{\mathsf{FR}}\left(\mathrm{grad}_{\mathsf{FR}}\ell_{N}\left(\rho\right)\cdot\zeta\right)=g_{\rho}^{\mathsf{FR}}\left(\rho\left(\delta\ell_{N}\left(\rho\right)-\int\delta\ell_{N}\left(\rho\right)\mathrm{d}\rho\right),\zeta\right)
\]
holds for any $\zeta\in\mathrm{Tan}_{\rho}^{\mathsf{FR}}\mathcal{P}_{2}(\mathbb{R}^{d})$,
and as a result 
\[
\mathrm{grad}_{\mathsf{FR}}\ell_{N}\left(\rho\right)=\rho\left[\delta\ell_{N}\left(\rho\right)-\int\delta\ell_{N}\left(\rho\right)\mathrm{d}\rho\right]=\rho\left[\delta\ell_{N}\left(\rho\right)+1\right],
\]
where we have used the fact that $\int\delta\ell_{N}(\rho)\mathrm{d}\rho=-1$.
Hence, the gradient flow of $\ell_{N}$ with respect to the
Fisher-Rao metric $d_{\mathsf{FR}}$ is given by 
\[
\partial_{t}\rho_{t}=-\mathrm{grad}_{\mathsf{FR}}\ell_{N}\left(\rho_{t}\right)=-\rho_{t}\left[\delta\ell_{N}\left(\rho_{t}\right)+1\right].
\]

\subsubsection{Other perspectives of Fisher-Rao gradient flow\label{subsec:FR-equivalence}}

In this section, we formally illustrate the connection between Fisher-Rao
gradient flow  \eqref{eq:Fisher-Rao-GF} with proximal gradient descent
and mirror descent. For simplicity, we focus on the case when $\rho_{t}$
is continuous; the case when $\rho_{t}$ is discrete is similar. We
also show the connection between the particle Fisher-Rao gradient
descent  \eqref{alg:FR-GD} and the EM algorithm.

\paragraph{Fisher-Rao gradient flow as proximal gradient flow.}

Consider the proximal gradient update in  \eqref{eq:FR-proximal}.
Recall that for $\mu,\nu\in\mathcal{P}_{\mathsf{ac}}(\mathbb{R}^{d})$,
the Fisher-Rao distance can be expressed as 
\[
d_{\mathsf{FR}}^{2}\left(\mu,\nu\right)=4\int\big|\sqrt{\mu\left(x\right)}-\sqrt{\nu\left(x\right)}\big|^{2}\mathrm{d}x.
\]
Note that for the purpose of defining a gradient flow, the metric
only matters up to its second-order local expansion 
\[
d_{\mathsf{FR}}^{2}\left(\mu,\nu\right)=\int\frac{\Delta^{2}(x)}{\nu\left(x\right)}\mathrm{d}x+\mathsf{higher}\text{-}\mathsf{order}\text{ }\mathsf{terms},
\]
where $\Delta=\mu-\nu$. Therefore we obtain an asymptotically (as
$\eta\to0$) equivalent problem 
\[
\rho_{t}^{\eta}\coloneqq\underset{\rho\in\mathcal{P}_{\mathsf{ac}}(\mathbb{R}^{d})}{\arg\min}\left\{ \int_{\mathbb{R}^{d}}\delta\ell_{N}\left(\rho_{t}\right)\mathrm{d}\left(\rho-\rho_{t}\right)+\frac{1}{2\eta}\int\frac{\left[\rho\left(x\right)-\rho_{t}\left(x\right)\right]^{2}}{\rho_{t}\left(x\right)}\mathrm{d}x\right\} .
\]
The first-order optimality condition is 
\[
\delta\ell_{N}\left(\rho_{t}\right)\left(x\right)+\frac{1}{\eta}\cdot\frac{\rho\left(x\right)-\rho_{t}\left(x\right)}{\rho_{t}\left(x\right)}=c
\]
for some constant $c\in\mathbb{R}$. This gives 
\[
\rho_{t}^{\eta}\left(x\right)=\rho_{t}\left(x\right)\left[1+x\eta-\eta\delta\ell_{N}\left(\rho_{t}\right)\left(x\right)\right].
\]
Since $\rho_{t}^{\eta}$ is a probability density, we have 
\[
1=\int_{\mathbb{R}^{d}}\rho_{t}^{\eta}\left(x\right)\mathrm{d}x=\int_{\mathbb{R}^{d}}\rho_{t}\left(x\right)\left[1+c\eta-\eta\delta\ell_{N}\left(\rho_{t}\right)\left(x\right)\right]\mathrm{d}x=1+c\eta-\eta,
\]
where we use the fact that $\int_{\mathbb{R}^{d}}\delta\ell_{N}(\rho)\mathrm{d}\rho=-1$
for any $\rho\in\mathcal{P}_{2}(\mathbb{R}^{d})$. This gives $c=-1$.
Therefore 
\[
\rho_{t}^{\eta}\left(x\right)=\rho_{t}\left(x\right)\left[1-\eta-\eta\delta\ell_{N}\left(\rho_{t}\right)\left(x\right)\right],
\]
and as a result, 
\[
\partial_{t}\rho_{t}=\lim_{\eta\to0+}\frac{\rho_{t}^{\eta}-\rho_{t}}{\eta}=-\left[1+\delta\ell_{N}\left(\rho_{t}\right)\right],
\]
which recovers the Fisher-Rao gradient flow  \eqref{eq:Fisher-Rao-GF}.

\paragraph{Fisher-Rao gradient flow as mirror flow.}

Recall that the mirror descent update is defined as

\[
\rho_{t}^{\eta}\coloneqq\underset{\rho\in\mathcal{P}_{\mathsf{ac}}(\mathbb{R}^{d})}{\arg\min}\int_{\mathbb{R}^{d}}\delta\ell_{N}\left(\rho_{t}\right)\mathrm{d}\left(\rho-\rho_{t}\right)+\frac{1}{\eta}\mathsf{KL}\left(\rho\,\Vert\,\rho_{t}\right).
\]
The first variation of $f(\cdot)\coloneqq\mathsf{KL}(\cdot\,\Vert\,\rho_{t})$
is given by
\[
\delta f\left(\rho\right)\left(x\right)=\log\left[\frac{\rho\left(x\right)}{\rho_{t}\left(x\right)}\right],
\]
therefore the first-order optimality condition reads
\[
\delta\ell_{N}\left(\rho_{t}\right)\left(x\right)+\frac{1}{\eta}\log\left[\frac{\rho\left(x\right)}{\rho_{t}\left(x\right)}\right]=c
\]
for some constant $c>0$. This gives
\[
\frac{\rho\left(x\right)}{\rho_{t}\left(x\right)}=\exp\left\{ \eta\left[c-\delta\ell_{N}\left(\rho_{t}\right)\left(x\right)\right]\right\} \propto\exp\left[-\eta\delta\ell_{N}\left(\rho_{t}\right)\left(x\right)\right].
\]
Since $\int_{\mathbb{R}^{d}}\rho_{t}^{\eta}(x)\mathrm{d}x=1$, we
know that the closed-form solution is given by

\begin{equation}
	\rho_{t}^{\eta}\left(x\right)=\frac{\rho_{t}\left(x\right)\exp\left[-\eta\delta\ell_{N}\left(\rho_{t}\right)\left(x\right)\right]}{\int\rho_{t}\left(y\right)\exp\left[-\eta\delta\ell_{N}\left(\rho_{t}\right)\left(y\right)\right]\mathrm{d}y}.\label{eq:MD-update-continuous}
\end{equation}
Then as $\eta\to0$, we can compute 
\begin{align*}
	\rho_{t}^{\eta}\left(x\right) & =\frac{\rho_{t}\left(x\right)\left[1-\eta\delta\ell_{N}\left(\rho_{t}\right)\left(x\right)+O\left(\eta^{2}\right)\right]}{\int\rho_{t}\left(y\right)\left[1-\eta\delta\ell_{N}\left(\rho_{t}\right)\left(y\right)+O\left(\eta^{2}\right)\right]\mathrm{d}y}=\frac{\rho_{t}\left(x\right)\left[1-\eta\delta\ell_{N}\left(\rho_{t}\right)\left(x\right)+O\left(\eta^{2}\right)\right]}{1+\eta+O\left(\eta^{2}\right)}\\
	& =\rho_{t}\left(x\right)\left\{ 1-\eta\left[1+\delta\ell_{N}\left(\rho_{t}\right)\left(x\right)\right]+O\left(\eta^{2}\right)\right\} ,
\end{align*}
where we use the fact that $\int\delta\ell_{N}(\rho)\mathrm{d}\rho=-1$
for any $\rho\in\mathcal{P}_{2}(\mathbb{R}^{d})$. Therefore the continuous-time
limit of mirror descent is 
\[
\partial_{t}\rho_{t}=\lim_{\eta\to0+}\frac{\rho_{t}^{\eta}\left(x\right)-\rho_{t}\left(x\right)}{\eta}=-\left[1+\delta\ell_{N}\left(\rho_{t}\right)\left(x\right)\right],
\]
which recovers the Fisher-Rao gradient flow  \eqref{eq:Fisher-Rao-GF}.

\paragraph{Fisher-Rao gradient descent as EM algorithm.}

Now we consider fitting a $m$-component Gaussian mixture model with
unknown weights $\{\omega^{(j)}\}_{1\leq j\leq m}$, known location
parameters $\{\mu_{j}\}_{1\leq j\leq m}$ and isotropic covariance.
Given the data $\{X_{i}\}_{1\leq i\leq N}$, the MLE is given by 
\[
\underset{\omega\in\Delta^{m-1}}{\arg\max}\ell\left(\omega\right)=\frac{1}{N}\sum_{i=1}^{N}\log\left[\sum_{j=1}^{m}\omega^{(j)}\phi\left(X_{i}-\mu_{j}\right)\right].
\]
The Expectation-Maximization algorithm for solving the above MLE is
given as follows. We first introduce the latent i.i.d.~random variables
$\{J_{i}\}_{1\leq i\leq N}$ distributed $\mathbb{P}(J_{i}=j)=\omega^{(j)}$
for $1\leq j\leq m$, then the distribution of the observed samples
is $X_{i}|J_{i}=j\sim\mathcal{N}(\mu_{j},I_{d})$. The joint distribution
of $(X_{i},J_{i})$ is given by 
\[
p_{\omega}\left(x,j\right)=\phi\left(X_{i}-\mu_{j}\right)\omega^{(j)},
\]
and conditional on $X_{i}=x$, the conditional distribution of $J_{i}$
is given by 
\[
p_{\omega}\left(j|x\right)=\frac{p_{\omega}\left(x,j\right)}{\sum_{l=1}^{m}p_{\omega}\left(x,l\right)}=\frac{\phi\left(X_{i}-\mu_{j}\right)\omega^{(j)}}{\sum_{l=1}^{m}\phi\left(X_{i}-\mu_{l}\right)\omega^{(l)}}.
\]
Given the current estimate $\omega_{t}$, the E-step amounts to computing
\begin{align*}
	Q\left(\omega|\omega_{t}\right) & =\frac{1}{N}\sum_{i=1}^{N}\sum_{j=1}^{m}p_{\omega_{t}}\left(j|X_{i}\right)\log p_{\omega}\left(X_{i},j\right)\\
	& =\frac{1}{N}\sum_{i=1}^{N}\sum_{j=1}^{m}\frac{\phi\left(X_{i}-\mu_{j}\right)\omega_{t}^{(j)}}{\sum_{l=1}^{m}\phi\left(X_{i}-\mu_{l}\right)\omega_{t}^{(l)}}\log\left[\phi\left(X_{i}-\mu_{j}\right)\omega^{(j)}\right].
\end{align*}
The M-step is to update 
\[
\omega_{t+1}\coloneqq\underset{\omega\in\Delta^{m-1}}{\arg\max}Q\left(\omega|\omega_{t}\right),
\]
which is given by 
\[
\omega_{t+1}^{(j)}=\frac{1}{N}\sum_{i=1}^{N}\frac{\phi\left(X_{i}-\mu_{j}\right)\omega_{t}^{(j)}}{\sum_{l=1}^{m}\phi\left(X_{i}-\mu_{l}\right)\omega_{t}^{(l)}}\qquad\forall\,1\leq j\leq m.
\]
This is equivalent to Algorithm \ref{alg:FR-GD} with step size $\eta=1$.

\subsection{Wasserstein gradient flow \label{subsec:Wass-derivation}}

We introduce the Riemannian structure over $\mathcal{P}_{2}(\mathbb{R}^{d})$
underlying the quadratic Wasserstein distance. We define the tangent
space at $\rho\in\mathcal{P}_{2}(\mathbb{R}^{d})$ to be 
\[
\mathrm{Tan}_{\rho}^{\mathsf{W}}\mathcal{P}_{2}(\mathbb{R}^{d})=\left\{ \zeta:\zeta=-\mathsf{div}\left(\rho\nabla u\right)\text{ for some }u\text{ satisfying }\int\Vert\nabla u\Vert_{2}^{2}\mathrm{d}\rho<\infty\right\} .
\]
We equip this tangent space with the $L^{2}(\rho)$ metric, namely
we define the Riemannian metric tensor $g_{\rho}^{\mathsf{W}}(\cdot,\cdot):\mathrm{Tan}_{\rho}^{\mathsf{W}}\mathcal{P}_{2}(\mathbb{R}^{d})\times\mathrm{Tan}_{\rho}^{\mathsf{W}}\mathcal{P}_{2}(\mathbb{R}^{d})\to\mathbb{R}$
as 
\begin{align*}
	g_{\rho}^{\mathsf{W}}\left(\zeta_{1},\zeta_{2}\right) & \coloneqq\int_{\mathbb{R}^{d}}\int_{\mathbb{R}^{d}}\left\langle \nabla u_{1},\nabla u_{2}\right\rangle \rho\left(\mathrm{d}x\right)
\end{align*}
for any $\zeta_{1}=-\mathsf{div}(\rho\nabla u_{1})$ and $\zeta_{2}=-\mathsf{div}(\rho\nabla u_{2})$.
The metric induced by this Riemannian structure recovers the quadratic
Wasserstein distance, namely 
\begin{align*}
	d_{\mathsf{W}}^{2}\left(\rho_{0},\rho_{1}\right) & =\inf\bigg\{\int_{0}^{1}\int\left\Vert v_{t}\right\Vert _{2}^{2}\mathrm{d}\rho_{t}\mathrm{d}t:\left(\rho_{t},v_{t}\right)_{t\in[0,1]}\text{ solves }\partial_{t}\rho_{t}+\mathsf{div}(\rho_{t}v_{t})=0\bigg\}\\
	& =\inf_{\pi\in\Pi(\rho_{0},\rho_{1})}\int\left\Vert x-y\right\Vert _{2}^{2}\pi\left(\mathrm{d}x,\mathrm{d}y\right),
\end{align*}
where $\Pi(\rho_{0},\rho_{1})$ is the set of couplings of $\rho_{0}$
and $\rho_{1}$. This is known as the Benamou-Brenier formula for
the Wasserstein distance. Then we derive the Wasserstein gradient
flow with respect to the functional $\ell_{N}$. Interested readers are referred to \citet{ambrosio2008gradient} for detailed introduction to Wasserstein gradient flow. Let $(\rho_{t})_{t\geq0}$
be a $C^{1}$ curve satisfying $\rho_{0}=\rho$ with initial velocity
\[
\partial_{t}\rho_{t}|_{t=0}=\zeta=-\mathsf{div}\left(\rho\nabla u\right).
\]
Then it should hold that 
\[
\frac{\mathrm{d}}{\mathrm{d}t}\ell_{N}\left(\rho_{t}\right)\Big|_{t=0}=g_{\rho}^{\mathsf{W}}\left(\mathrm{grad}_{\mathsf{W}}\ell_{N}\left(\rho\right),\zeta\right).
\]
The left hand side of the above equation equals to 
\begin{align*}
	\frac{\mathrm{d}}{\mathrm{d}t}\ell_{N}\left(\rho_{t}\right)\Big|_{t=0} & =\int\delta\ell_{N}\left(\rho\right)\zeta\mathrm{d}x=-\int\delta\ell_{N}\left(\rho\right)\mathsf{div}\left(\rho\nabla u\right)\mathrm{d}x\\
	& =-\int\left\langle \nabla\delta\ell_{N}\left(\rho\right),\nabla u\right\rangle \mathrm{d}\rho\\
	& =g_{\rho}^{\mathsf{W}}\left(-\mathsf{div}\left(\nabla\delta\ell_{N}\left(\rho\right)\rho\right),\zeta\right).
\end{align*}
Therefore 
\[
g_{\rho}^{\mathsf{W}}\left(\mathrm{grad}_{\mathsf{W}}\ell_{N}\left(\rho\right),\zeta\right)=g_{\rho}^{\mathsf{W}}\left(-\mathsf{div}\left(\nabla\delta\ell_{N}\left(\rho\right)\rho\right),\zeta\right)
\]
holds for any $\zeta\in\mathrm{Tan}_{\rho}^{\mathsf{W}}\mathcal{P}_{2}(\mathbb{R}^{d})$,
and as a result 
\[
\mathrm{grad}_{\mathsf{W}}\ell_{N}\left(\rho\right)=-\mathsf{div}\left(\nabla\delta\ell_{N}\left(\rho\right)\rho\right).
\]
This shows that the gradient flow of $\ell_{N}$ with respect to the
quadratic Wasserstein distance $d_{\mathsf{W}}$ is given by 
\[
\partial_{t}\rho_{t}=-\mathrm{grad}_{\mathsf{W}}\ell_{N}\left(\rho_{t}\right)=\mathsf{div}\left(\nabla\delta\ell_{N}\left(\rho_{t}\right)\rho_{t}\right).
\]

\subsection{Wasserstein-Fisher-Rao gradient flow \label{subsec:WFR-derivation}}

We introduce the Riemannian structure over $\mathcal{P}_{2}(\mathbb{R}^{d})$
underlying the Wasserstein-Fisher-Rao metric. Define the tangent space
at $\rho\in\mathcal{P}_{2}(\mathbb{R}^{d})$ to be 
\begin{align*}
	\mathrm{Tan}_{\rho}^{\mathsf{WFR}}\mathcal{P}_{2}(\mathbb{R}^{d}) & =\Big\{\zeta:\zeta=-\mathsf{div}\left(\rho\nabla u\right)+\rho\left(\alpha-\int\alpha\mathrm{d}\rho\right)\text{ for some }u,\alpha:\mathbb{R}^{d}\to\mathbb{R}\\
	& \qquad\quad\;\:\text{satisfying }\int(\alpha^{2}+\Vert\nabla u\Vert_{2}^{2})\mathrm{d}\rho<\infty\Big\}.
\end{align*}
We equip this tangent space with the Riemannian metric tensor $g_{\rho}^{\mathsf{WFR}}(\cdot,\cdot):\mathrm{Tan}_{\rho}^{\mathsf{WFR}}\mathcal{P}_{2}(\mathbb{R}^{d})\times\mathrm{Tan}_{\rho}^{\mathsf{WFR}}\mathcal{P}_{2}(\mathbb{R}^{d})\to\mathbb{R}$
defined as 
\begin{align*}
	g_{\rho}^{\mathsf{WFR}}\left(\zeta_{1},\zeta_{2}\right) & \coloneqq\int_{\mathbb{R}^{d}}\left\langle \nabla u_{1},\nabla u_{2}\right\rangle \rho\left(\mathrm{d}x\right)+\int_{\mathbb{R}^{d}}\left[\alpha_{1}\left(x\right)-\int_{\mathbb{R}^{d}}\alpha_{1}\mathrm{d}\rho\right]\left[\alpha_{2}\left(x\right)-\int_{\mathbb{R}^{d}}\alpha_{2}\mathrm{d}\rho\right]\rho\left(\mathrm{d}x\right)\\
	& =\int_{\mathbb{R}^{d}}\left\langle \nabla u_{1},\nabla u_{2}\right\rangle \rho\left(\mathrm{d}x\right)+\int_{\mathbb{R}^{d}}\alpha_{1}\left(x\right)\alpha_{2}\left(x\right)\rho\left(\mathrm{d}x\right)-\left(\int_{\mathbb{R}^{d}}\alpha_{1}\mathrm{d}\rho\right)\left(\int_{\mathbb{R}^{d}}\alpha_{2}\mathrm{d}\rho\right)
\end{align*}
for any $\zeta_{1}=-\mathsf{div}(\rho\nabla u_{1})+\rho(\alpha_{1}-\int\alpha_{1}\mathrm{d}\rho)$
and $\zeta_{2}=-\mathsf{div}(\rho\nabla u_{2})+\rho(\alpha_{2}-\int\alpha_{2}\mathrm{d}\rho)$.
The metric induced by the above Riemannian structure, namely the Wasserstein-Fisher-Rao
metric $d_{\mathsf{WFR}}(\cdot,\cdot)$, is defined as 
\begin{align*}
	d_{\mathsf{WFR}}^{2}\left(\rho_{0},\rho_{1}\right) & =\inf\bigg\{\int_{0}^{1}\int\Big[\left\Vert v_{t}\right\Vert ^{2}+\Big(\alpha_{t}-\int\alpha_{t}\mathrm{d}\rho_{t}\Big)^{2}\Big]\mathrm{d}\rho_{t}\mathrm{d}t:\left(\rho_{t},v_{t},\alpha_{t}\right)_{0\leq t\leq1}\\
	& \qquad\qquad\text{solves }\partial_{t}\rho_{t}=-\mathsf{div}(\rho_{t}v_{t})+\rho_{t}\alpha_{t}\bigg\}.
\end{align*}
Then we follow \citet{lu2019accelerating,gallouet2017jko} to derive the Wasserstein-Fisher-Rao gradient flow with respect
to the functional $\ell_{N}$. Let $(\rho_{t})_{t\geq0}$ be a $C^{1}$
curve satisfying $\rho_{0}=\rho$ with initial velocity 
\[
\partial_{t}\rho_{t}|_{t=0}=\zeta=-\mathsf{div}\left(\rho\nabla u\right)+\rho\left(\alpha-\int\alpha\mathrm{d}\rho\right).
\]
Then it should hold that 
\[
\frac{\mathrm{d}}{\mathrm{d}t}\ell_{N}\left(\rho_{t}\right)\Big|_{t=0}=g_{\rho}^{\mathsf{WFR}}\left(\mathrm{grad}_{\mathsf{WFR}}\ell_{N}\left(\rho\right),\zeta\right).
\]
The left hand side of the above equation equals to 
\begin{align*}
	\frac{\mathrm{d}}{\mathrm{d}t}\ell_{N}\left(\rho_{t}\right)\Big|_{t=0} & =\int\delta\ell_{N}\left(\rho\right)\zeta\mathrm{d}x=\int\delta\ell_{N}\left(\rho\right)\left[-\mathsf{div}\left(\rho\nabla u\right)+\rho\left(\alpha-\int\alpha\mathrm{d}\rho\right)\right]\mathrm{d}x\\
	& =-\int\left\langle \nabla\delta\ell_{N}\left(\rho\right),\nabla u\right\rangle \mathrm{d}\rho+\int\left(\delta\ell_{N}\left(\rho\right)-\int\delta\ell_{N}\left(\rho\right)\mathrm{d}\rho\right)\left(\alpha-\int\alpha\mathrm{d}\rho\right)\mathrm{d}\rho\\
	& =g_{\rho}^{\mathsf{WFR}}\left(-\mathsf{div}\left(\nabla\delta\ell_{N}\left(\rho\right)\rho\right)+\rho\left(\delta\ell_{N}\left(\rho\right)-\int\delta\ell_{N}\left(\rho\right)\mathrm{d}\rho\right),\zeta\right).
\end{align*}
Therefore 
\[
g_{\rho}^{\mathsf{WFR}}\left(\mathrm{grad}_{\mathsf{W}}\ell_{N}\left(\rho\right),\zeta\right)=g_{\rho}^{\mathsf{WFR}}\left(-\mathsf{div}\left(\nabla\delta\ell_{N}\left(\rho\right)\rho\right)+\rho\left(\delta\ell_{N}\left(\rho\right)-\int\delta\ell_{N}\left(\rho\right)\mathrm{d}\rho\right),\zeta\right)
\]
holds for any $\zeta\in\mathrm{Tan}_{\rho}^{\mathsf{WFR}}\mathcal{P}_{2}(\mathbb{R}^{d})$,
and as a result 
\begin{align*}
	\mathrm{grad}_{\mathsf{WFR}}\ell_{N}\left(\rho\right) & =-\mathsf{div}\left(\nabla\delta\ell_{N}\left(\rho\right)\rho\right)+\rho\left(\delta\ell_{N}\left(\rho\right)-\int\delta\ell_{N}\left(\rho\right)\mathrm{d}\rho\right)\\
	& =-\mathsf{div}\left(\nabla\delta\ell_{N}\left(\rho\right)\rho\right)+\rho\left[1+\delta\ell_{N}\left(\rho\right)\right],
\end{align*}
where we have used the fact that $\int\delta\ell_{N}(\rho)\mathrm{d}\rho=-1$.
This shows that the gradient flow of $\ell_{N}$ with respect to the
Wasserstein-Fisher-Rao metric $d_{\mathsf{WFR}}$ is given by 
\[
\partial_{t}\rho_{t}=-\mathrm{grad}_{\mathsf{WFR}}\ell_{N}\left(\rho_{t}\right)=\mathsf{div}\left(\nabla\delta\ell_{N}\left(\rho_{t}\right)\rho_{t}\right)-\rho_{t}\left[1+\delta\ell_{N}\left(\rho_{t}\right)\right].
\]

%% file: appendix_gd_theory.tex
\section{Convergence theory (Proof of Theorem \ref{thm:FR-convergence} and
\ref{thm:WFR-convergence})}

\label{sec-convergence-proofs}

In this section, we provide a systematic treatment to the convergence
of Fisher-Rao gradient descent and Wasserstein-Fisher-Rao gradient
descent. Instead of proving Theorem \ref{thm:FR-convergence} and
\ref{thm:WFR-convergence} separately, we prove a more general convergence
result in Theorem \ref{thm-convergence-wfr-iteration}, which admits
Theorem \ref{thm:FR-convergence} and \ref{thm:WFR-convergence} as
its special cases.

To begin with, we define the Wasserstein gradient descent update. 

\begin{definition}[Wasserstein gradient descent]\label{defn-gd}
For any $\rho\in\mathcal{P}(\mathbb{R}^{d})$ and $\eta\geq0$, we
define $\rho^{\mathrm{W},\eta}=(\mathrm{id}-\eta\nabla\delta\ell_{N}(\rho))_{\#}\rho$.
Given an initial distribution $\rho_{0}\in\mathcal{P}(\mathbb{R}^{d})$,
the Wasserstein gradient descent for solving \eqref{eq:NPMLE} is
defined recursively by $\rho_{n+1}=\rho_{n}^{\mathrm{W},\eta}$, $\forall n\geq0$.
\end{definition}

In words, $\rho^{\mathrm{W},\eta}$ is the push forward of $\rho$
by the mapping $x\mapsto x-\eta\nabla\delta\ell_{N}(\rho)(x)$. For
any $p\geq1$, define the $p$-Wasserstein distance between $\rho_{0}$
and $\rho_{1}$ in $\mathcal{P}(\mathbb{R}^{d})$ as 
\[
W_{p}\left(\rho_{0},\rho_{1}\right)=\left[\inf_{\pi\in\Pi(\rho_{0},\rho_{1})}\int\left\Vert x-y\right\Vert _{p}^{p}\pi\left(\mathrm{d}x,\mathrm{d}y\right)\right]^{1/p}.
\]
The following lemma characterizes the decrease in loss function $\ell_{N}$
by running one step of Wasserstein gradient descent, which can be
lower bounded by the squared Wasserstein distance between the two
iterates. 

\begin{lemma}[Wasserstein gradient descent]\label{lem-gd} Let
Assumption \ref{assumption-pdf-0} hold. Choose any $\rho_{0}\in\mathcal{P}(\mathbb{R}^{d})$
and define 
\[
c_{0}=\frac{e^{-\ell_{N}(\rho_{0})}\underline{\phi}(\inf\{r\geq0:~\bar{\phi}(r)\leq e^{-\ell_{N}(\rho_{0})}/2\}+\mathrm{diam}(\Omega))}{2\bar{\phi}(0)}.
\]
Suppose that 
\[
0\leq\eta<\frac{c_{0}}{\sup_{x\in\mathbb{R}^{d}}\|\nabla^{2}\phi(x)\|_{2}+\sup_{x\in\mathbb{R}^{d}}\|\nabla\phi(x)\|_{2}^{2}/c_{0}}.
\]
Define $\rho_{n+1}=\rho_{n}^{\mathrm{W},\eta}$ for $n\geq0$. Then,
for any $n\geq0$, 
\begin{align*}
\ell_{N}(\rho_{n+1})-\ell_{N}(\rho_{n}) & \leq-\frac{\eta}{2}\mathbb{E}_{Y\sim\rho_{n}}\|\nabla\delta\ell_{N}(\rho_{n})(Y)\|_{2}^{2}\leq-\frac{1}{2\eta}W_{2}^{2}(\rho_{n+1},\rho_{n}).
\end{align*}
In addition, if $\mathsf{supp}(\rho_{0})=\mathbb{R}^{d}$, then $\mathsf{supp}(\rho_{n})=\mathbb{R}^{d}$
for all $n\geq0$. \end{lemma}

\begin{proof}See Appendix \ref{sec-lem-gd-proof}. \end{proof}

Then we define the Fisher-Rao gradient descent update.

\begin{definition}[Fisher-Rao gradient descent]\label{defn-fr}
For any $\rho\in\mathcal{P}(\mathbb{R}^{d})$ and $\gamma\in[0,1]$,
we define $\rho^{\mathrm{FR},\gamma}\in\mathcal{P}(\mathbb{R}^{d})$
through 
\[
\frac{\mathrm{d}\rho^{\mathrm{FR},\gamma}}{\mathrm{d}\rho}=1-\gamma\left[\delta\ell_{N}(\rho)+1\right].
\]
Given an initial distribution $\rho_{0}\in\mathcal{P}(\mathbb{R}^{d})$,
the Fisher-Rao gradient descent for solving \eqref{eq:NPMLE} is defined
recursively by $\rho_{n+1}=\rho_{n}^{\mathrm{FR},\gamma}$, $\forall n\geq0$.
\end{definition}

It is easily seen that $\rho^{\mathrm{FR},\gamma}=(1-\gamma)\rho+\gamma\rho^{\mathrm{FR},1}$
and 
\[
\frac{\mathrm{d}\rho^{\mathrm{FR},1}}{\mathrm{d}\rho}(x)=-\delta\ell_{N}(\rho)(x)=\frac{1}{N}\sum_{i=1}^{N}\frac{\phi(X_{i}-x)}{(\rho*\phi)(X_{i})},\qquad\forall x\in\mathbb{R}^{d}.
\]
From Appendix \ref{subsec:FR-derivation} we know that Fisher-Rao
gradient descent with $\gamma=1$ can be viewed as fixed-location
EM algorithm. The following lemma shows that by running one step of
Fisher-Rao gradient descent, the decrease in loss function can be
lower bounded by the KL divergence between the two iterates.

\begin{lemma}[Fisher-Rao gradient descent]\label{lem-FR-KL} Choose
any $\rho\in\mathcal{P}(\mathbb{R}^{d})$ and $\gamma\in[0,1]$. We
have 
\[
\ell_{N}(\mu)\leq\ell_{N}(\rho)+\mathsf{KL}(\rho^{\mathrm{FR},1}\Vert\mu)-\mathsf{KL}(\rho^{\mathrm{FR},1}\Vert\rho),\qquad\forall\,\mu\ll\rho,
\]
and
\[
\ell_{N}(\rho^{\mathrm{FR},\gamma})-\ell_{N}(\rho)\leq-\mathsf{KL}(\rho^{\mathrm{FR},\gamma}\Vert\rho).
\]
\end{lemma}

\begin{proof} See Appendix \ref{sec-thm-FR-KL-proof}. \end{proof}

Finally we define Wasserstein-Fisher-Rao gradient descent, which can
be viewed as iteratively applying one step of Wasserstein gradient
descent (cf.~Definition \ref{defn-gd}) and one step of Fisher-Rao
gradient descent (cf.~Definition \ref{defn-fr}).

\begin{definition}[Wasserstein-Fisher-Rao gradient descent] Let
$\rho_{0}\in\mathcal{P}(\mathbb{R}^{d})$, $\eta\geq0$ and $\gamma\in[0,1]$.
The Wasserstein-Fisher-Rao gradient descent is defined through $\widetilde{\rho}_{n}=\rho_{n}^{\mathrm{W},\eta}$
and $\rho_{n+1}=\widetilde{\rho}_{n}^{\mathrm{FR},\gamma}$ for all
$n\geq0$. \end{definition}

The Wasserstein gradient descent and Fisher-Rao gradient descent are
special cases of Wasserstein-Fisher-Rao gradient descent with $\gamma=0$
and $\eta=0$, respectively. The following theorem shows that if Wasserstein-Fisher-Rao
gradient descent converges weakly to a limit distribution, then this
weak limit is the NPMLE.

\begin{theorem}\label{thm-convergence-wfr-iteration} Suppose that
Assumption \ref{assumption-pdf-0} holds. Let $\{\rho_{n}\}_{n=0}^{\infty}$
be the iterates of Wasserstein-Fisher-Rao gradient descent, and 
\[
c_{0}=\frac{e^{-\ell_{N}(\rho_{0})}\underline{\phi}(\inf\{r\geq0:~\bar{\phi}(r)\leq e^{-\ell_{N}(\rho_{0})}/2\}+\mathrm{diam}(\Omega))}{2\bar{\phi}(0)}.
\]
If $\mathsf{supp}(\rho_{0})=\mathbb{R}^{d}$, $\gamma\in(0,1]$, and
\[
0\leq\eta<\frac{c_{0}}{\sup_{{x}\in\mathbb{R}^{d}}\|\nabla^{2}\phi(x)\|_{2}+\sup_{{x}\in\mathbb{R}^{d}}\|\nabla\phi({x})\|_{2}^{2}/c_{0}},
\]
then for all $n\geq0$, $\ell_{N}(\rho_{n+1})\leq\ell_{N}(\rho_{n})$
holds, and $\mathsf{supp}(\rho_{n})=\mathbb{R}^{d}$. Furthermore,
if $\{\rho_{n}\}_{n=0}^{\infty}$ converges weakly to $\rho_{\infty}\in\mathcal{P}(\mathbb{R}^{d})$,
then this limit $\rho_{\infty}$ is NPMLE, i.e.~an optimal solution
to \eqref{eq:NPMLE}. \end{theorem}

\begin{proof} See Appendix \ref{sec-thm-convergence-wfr-iteration-proof}.
\end{proof}

Theorem \ref{thm-convergence-wfr-iteration} directly implies that
Theorem \ref{thm:WFR-convergence} holds. By taking $\eta=0$, this
also implies that Theorem \ref{thm:FR-convergence} holds. Note that
Theorem \ref{thm-convergence-wfr-iteration} requires $\gamma>0$,
therefore it does not provide convergence guarantee for Wasserstein
gradient descent.

Here are some key ideas for showing the optimality of the weak limit
$\rho_{\infty}$. If $\rho_{\infty}$ is not an optimal
solution, then Theorem \ref{thm:NPMLE-basics} implies that $\delta\ell_{N}(\rho_{\infty})(x_{0})<-1-\varepsilon$
holds for some $x_{0}\in\mathbb{R}^d$ and $\varepsilon>0$. We can
find some appropriate $y\in(-1-\varepsilon,-1-\varepsilon/2)$ and
 study the sublevel set $\bar{S}=\{x\in\mathbb{R}^{d}:\,\delta\ell_{N}(\rho_{\infty})\leq y\}$. Note that for $n$ large, we have $\nabla\delta\ell_{N}(\rho_{n})\approx\nabla\delta\ell_{N}(\rho_{\infty})$.
We analyze the Wasserstein step and the Fisher-Rao step separately:
\begin{itemize}
\item For the Wasserstein step we show that
 for any $x\in\bar{S}$, the gradient descent step  $x-\eta\nabla\delta\ell_{N}(\rho_{n})(x)$ remains in $\bar S$; this follows from the definition of the gradient step and the fact that the step size is chosen small enough.
A crucial step therein is to choose $y$ so that the gradient $\nabla\delta\ell_{N}(\rho_{\infty})$
does not vanish on the level set $\{x\in\mathbb{R}^{d}:~\delta\ell_{N}(\rho_{\infty})=y\}$.
The existence of such a $y$ follows readily from Sard's lemma. Since $\widetilde{\rho}_{n}=(\mathrm{id}-\nabla\delta\ell_{N}(\rho_{n}))_{\#}\rho_{n}$,
we get $\widetilde{\rho}_{n}(\bar{S})\geq\rho_{n}(\bar{S})$. 
\item For the Fisher-Rao step, we first establish
% $\delta\ell_{N}(\widetilde{\rho}_{n})\approx\delta\ell_{N}(\rho_{\infty})$
% holds for large $n$, and thus 
that $\delta\ell_{N}(\widetilde{\rho}_{n})(x)<-1-\varepsilon/4$ for all $x \in \bar S$. 
Recalling that
\begin{align*}
\frac{\mathrm{d}\rho_{n+1}}{\mathrm{d}\widetilde{\rho}_{n}}(\cdot)=(1-\gamma)+\gamma\cdot[-\delta\ell(\widetilde{\rho}_{n})(\cdot)],
\end{align*}
it readily yields$\rho_{n+1}(\bar{S})\geq(1+\varepsilon/4)\widetilde{\rho}_{n}(\bar{S})$. 
\end{itemize}
Putting both steps together, we get that there exists $N$ such that $\rho_{n+1}(\bar{S})\geq(1+\varepsilon/4)\rho_{n}(\bar{S})$
holds all $n>N$. Note that this geometric improvement is entirely driven by the Fisher-Rao part of the proof. To conclude, we use $\mathsf{supp}(\rho_{N})=\mathbb{R}^{d}$
to get $\widetilde{\rho}_{N}(\bar{S})>0$, and thus $\lim_{n\to\infty}\widetilde{\rho}_{n}(\bar{S})=\infty$,
which leads to a contradiction so that $\rho_\infty$ must be an optimal solution to~\eqref{eq:NPMLE}.

\subsection{Proof of Lemma \ref{lem-gd}\label{sec-lem-gd-proof}}

We invoke a descent lemma Wasserstein gradient
descent, whose proof is deferred to Appendix \ref{sec-lem-preservation-proof}. Such a lemma is standard in convex optimization optimization~\citep[see, e.g.,][eq. (3.5)]{Bub15}. It has appeared for optimization over the Wasserstein space in~\cite{SalKorLui20} under for functionals that are convex along generalized geodesics, an assumption that does not hold for the negative log-likelihood $\ell_N$. 

\begin{lemma}[A descent lemma]\label{lem-preservation}
Let Assumption \ref{assumption-pdf-0} hold. Choose any $\rho\in\mathcal{P}(\mathbb{R}^{d})$.
Define $c=\inf_{{x}\in\Omega}(\rho*\phi)({x})$, $G=\sup_{x\in\mathbb{R}^{d}}\|\nabla\phi(x)\|_{2}$
and $H=\sup_{x\in\mathbb{R}^{d}}\|\nabla^{2}\phi(x)\|_{2}$. 
\begin{itemize}
\item We have 
\begin{align*}
\ell_{N}(\rho^{\mathrm{W},\eta})-\ell_{N}(\rho) & \leq-\eta\bigg[1-\frac{\eta}{2c}\bigg(H+\frac{G^{2}}{c}\bigg)\bigg]\mathbb{E}_{Y\sim\rho}\|\nabla\delta\ell_{N}(\rho)(Y)\|_{2}^{2}.
\end{align*}
In addition, we have $\sup_{x\in\mathbb{R}^{d}}\|\nabla^{2}\delta\ell_{N}(\rho)(x)\|_{2}\leq H/c$. 
\item If $0\leq\eta<c/H$ and $\mathsf{supp}(\rho)=\mathbb{R}^{d}$, then
$\mathsf{supp}(\rho^{\mathrm{W},\eta})=\mathbb{R}^{d}$. 
\end{itemize}
\end{lemma}

We now come back to Lemma \ref{lem-gd}. Let $R=\inf\{r\geq0:~\bar{\phi}(r)\leq e^{-\ell_{N}(\rho_{0})}/2\}$.
Lemma \ref{lem-loss-bounds} implies that for any $\mu\in\mathcal{P}(\mathbb{R}^{d})$
with $\ell_{N}(\mu)\leq\ell_{N}(\rho_{0})$, we have 
\[
\inf_{x\in\Omega}(\mu*\phi)(x)\geq e^{-\ell_{N}(\rho_{0})}\underline{\phi}(R+\mathrm{diam}(\Omega))/[2\bar{\phi}(0)]=c_{0}.
\]
In particular, $\inf_{x\in\Omega}(\rho_{0}*\phi)(x)\geq c_{0}$. When
$\eta\leq c_{0}/(H+G^{2}/c_{0})$, Lemma \ref{lem-preservation} and
the definition $\rho_{1}=(\mathrm{id}-\eta\nabla\delta\ell_{N}(\rho_{0}))_{\#}\rho_{0}$
together yield 
\begin{align*}
\ell_{N}(\rho_{1})-\ell_{N}(\rho_{0}) & \leq-\frac{\eta}{2}\mathbb{E}_{Y\sim\rho_{0}}\|\nabla\delta\ell_{N}(\rho_{0})(Y)\|_{2}^{2}\leq-\frac{1}{2\eta}W_{2}^{2}(\rho_{1},\rho_{0})\leq0.
\end{align*}
Also, if $\mathsf{supp}(\rho_{0})=\mathbb{R}^{d}$, then $\mathsf{supp}(\rho_{1})=\mathbb{R}^{d}$.
From $\ell_{N}(\rho_{1})\leq\ell_{N}(\rho_{0})$ we obtain that $\inf_{x\in\Omega}(\rho_{1}*\phi)(x)\geq c_{0}$.
Then, the proof is completed by induction.

\subsection{Proof of Lemma \ref{lem-preservation} \label{sec-lem-preservation-proof}}

We prove the two results in Lemma \ref{lem-preservation} in sequence.

\paragraph{Part 1.}

Let $\nu=N^{-1}\sum_{i=1}^{N}\delta_{X_{i}}$ be the empirical data
distribution, and let $h(x)=-\log x$ for $x>0$. Then we can wirte
\begin{align*}
\ell_{N}(\rho) & =-\frac{1}{N}\sum_{i=1}^{N}\log\left[\rho*\phi\left(X_{i}\right)\right]=\mathbb{E}_{X\sim\nu}\left[h\left((\rho*\phi)(X)\right)\right]=\mathbb{E}_{X\sim\nu}\left[h\left(\mathbb{E}_{Y\sim\rho}\left[\phi(X-Y)\right]\right)\right]
\end{align*}
as well as
\[
\ell_{N}(\rho^{\mathrm{W},\eta})=\mathbb{E}_{X\sim\nu}\left[h\left(\mathbb{E}_{Y\sim\rho^{\mathrm{W},\eta}}\left[\phi(X-Y)\right]\right)\right]=\mathbb{E}_{X\sim\nu}\left[h\left(\mathbb{E}_{Y\sim\rho}\left[\phi\left(X-Y+\eta\nabla\delta\ell_{N}(\rho)(Y)\right)\right]\right)\right].
\]
Define $a(x)=\mathbb{E}_{Y\sim\rho}[\phi(x-Y+\eta\nabla\delta\ell_{N}(\rho)(Y))]$
and $b(x)=\mathbb{E}_{Y\sim\rho}[\phi(x-Y)]=(\rho*\phi)(x)$ for any
$x\in\mathbb{R}^{d}$. Then we can write
\[
\ell_{N}(\rho)=\mathbb{E}_{X\sim\nu}\left[h\left(b(X)\right)\right],\qquad\ell_{N}(\rho^{\mathrm{W},\eta})=\mathbb{E}_{X\sim\nu}\left[h\left(a(X)\right)\right].
\]
 Note that $h'(x)=-x^{-1}<0$, $h''(x)=x^{-2}>0$ and $h'''(x)=-2x^{-3}<0$.
For any $a,b>0$, by Taylor's theorem
\[
h(a)\leq h(b)+h'(b)(a-b)+\frac{h''(b)}{2}(a-b)^{2}=h(b)-\frac{a-b}{b}+\frac{(a-b)^{2}}{2b^{2}}.
\]
Taking the above two equations collectively gives
\begin{align}
\ell_{N}(\rho^{\mathrm{W},\eta})-\ell_{N}(\rho) & =\mathbb{E}_{X\sim\nu}\left[h\left(a(X)\right)-h\left(b(X)\right)\right]\nonumber \\
 & \leq\underbrace{-\mathbb{E}_{X\sim\nu}\left[\frac{a(X)-b(X)}{b(X)}\right]}_{\eqqcolon\alpha_{1}}+\underbrace{\frac{1}{2}\mathbb{E}_{X\sim\nu}\left[\frac{[a(X)-b(X)]^{2}}{b(X)^{2}}\right]}_{\eqqcolon\alpha_{2}}.\label{eq:proof-preserve-1}
\end{align}
Now derive upper bounds for $\alpha_{1}$ and $\alpha_{2}$ respectively. 
\begin{itemize}
\item To control $\alpha_{1}$, let $G=\sup_{x\in\mathbb{R}^{d}}\|\nabla\phi(x)\|_{2}$
and observe that
\[
\left|\phi\left(x-y+\eta\nabla\delta\ell_{N}(\rho)(y)\right)-\phi(x-y)\right|\leq G\eta\left\Vert \nabla\delta\ell_{N}(\rho)(y)\right\Vert _{2}
\]
for all $x,y\in\mathbb{R}^{d}$, therefore 
\begin{equation}
\left|a\left(x\right)-b\left(x\right)\right|\leq G\eta\mathbb{E}_{Y\sim\rho}\left[\left\Vert \nabla\delta\ell_{N}\left(\rho\right)\left(Y\right)\right\Vert _{2}\right]\label{eq:proof-preserve-2}
\end{equation}
holds for all $x\in\mathbb{R}^{d}$. Then we have
\begin{align}
\alpha_{1} & \overset{\text{(i)}}{\leq}\frac{G^{2}\eta^{2}\mathbb{E}_{Y\sim\rho}^{2}\left[\left\Vert \nabla\delta\ell_{N}\left(\rho\right)\left(Y\right)\right\Vert \right]}{c^{2}}\overset{\text{(ii)}}{\leq}\frac{G^{2}\eta^{2}}{c^{2}}\mathbb{E}_{Y\sim\rho}\left[\left\Vert \nabla\delta\ell_{N}\left(\rho\right)\left(Y\right)\right\Vert _{2}^{2}\right],\label{eq:proof-preserve-3}
\end{align}
where (i) follows from  \eqref{eq:proof-preserve-2} and the fact 
\begin{equation}
c=\inf_{x\in\Omega}(\rho*\phi)(x)=\inf_{x\in\Omega}b(x)\label{eq:proof-preserve-3.5}
\end{equation}
 and (ii) follows from Jensen's inequality. 
\item Regarding $\alpha_{2}$, let $H=\sup_{x\in\mathbb{R}^{d}}\|\nabla^{2}\phi(x)\|_{2}$
and we have
\begin{align*}
\phi\left(x-y+\eta\nabla\delta\ell_{N}(\rho)(y)\right)-\phi(x-y) & \geq\left\langle \nabla\phi(x-y),\eta\nabla\delta\ell_{N}(\rho)(y)\right\rangle -\frac{H}{2}\eta^{2}\|\nabla\delta\ell_{N}(\rho)(y)\|_{2}^{2},
\end{align*}
for any $x,y\in\mathbb{R}^{d}$, and therefore
\begin{equation}
a(x)-b(x)\geq\mathbb{E}_{Y\sim\rho}\left[\left\langle \nabla\phi(x-Y),\eta\nabla\delta\ell_{N}(\rho)(Y)\right\rangle \right]-\frac{H}{2}\eta^{2}\mathbb{E}_{Y\sim\rho}\left[\left\Vert \nabla\delta\ell_{N}(\rho)(Y)\right\Vert _{2}^{2}\right]\label{eq:proof-preserve-4}
\end{equation}
for any $x\in\mathbb{R}^{d}$. Since $b(x)>0$, we have 
\begin{align}
\alpha_{2} & \overset{\text{(i)}}{\leq}-\mathbb{E}_{X\sim\nu}\left[\frac{\mathbb{E}_{Y\sim\rho}\left[\left\langle \nabla\phi(X-Y),\eta\nabla\delta\ell_{N}(\rho)(Y)\right\rangle \right]-\frac{H}{2}\eta^{2}\mathbb{E}_{Y\sim\rho}\left[\left\Vert \nabla\delta\ell_{N}(\rho)(Y)\right\Vert _{2}^{2}\right]}{b\left(X\right)}\right]\nonumber \\
 & =-\eta\mathbb{E}_{Y\sim\rho}\left[\left\langle \mathbb{E}_{X\sim\nu}\bigg(\frac{\nabla\phi(X-Y)}{b(X)}\bigg),\nabla\delta\ell_{N}(\rho)(Y)\right\rangle \right]+\frac{H\eta^{2}}{2}\mathbb{E}_{X\sim\nu}\left[\frac{1}{b(X)}\right]\mathbb{E}_{Y\sim\rho}\left[\left\Vert \nabla\delta\ell_{N}(\rho)(Y)\right\Vert _{2}^{2}\right]\nonumber \\
 & \overset{\text{(ii)}}{=}-\eta\mathbb{E}_{Y\sim\rho}\left[\left\Vert \nabla\delta\ell_{N}(\rho)(Y)\right\Vert _{2}^{2}\right]+\frac{H\eta^{2}}{2}\mathbb{E}_{X\sim\nu}\left[\frac{1}{b(X)}\right]\mathbb{E}_{Y\sim\rho}\left[\left\Vert \nabla\delta\ell_{N}(\rho)(Y)\right\Vert _{2}^{2}\right]\nonumber \\
 & \overset{\text{(iii)}}{\leq}\left(-\eta+\frac{H\eta^{2}}{2c}\right)\mathbb{E}_{Y\sim\rho}\left[\left\Vert \nabla\delta\ell_{N}(\rho)(Y)\right\Vert _{2}^{2}\right].\label{eq:proof-preserve-5}
\end{align}
Here (i) utilizes  \eqref{eq:proof-preserve-4}; (ii) holds since 
\begin{equation}
\delta\ell_{N}(\rho)(y)=-\mathbb{E}_{X\sim\nu}\left[\frac{\phi(X-y)}{b(X)}\right],\qquad\nabla\delta\ell_{N}(\rho)(y)=\mathbb{E}_{X\sim\nu}\left[\frac{\nabla\phi(X-y)}{b(X)}\right],\label{eq:proof-preserve-6}
\end{equation}
and therefore 
\[
\mathbb{E}_{Y\sim\rho}\left[\bigg\langle\mathbb{E}_{X\sim\nu}\bigg(\frac{\nabla\phi(X-Y)}{b(X)}\bigg),\nabla\delta\ell_{N}(\rho)(Y)\bigg\rangle\right]=\mathbb{E}_{Y\sim\rho}\left[\left\Vert \nabla\delta\ell_{N}(\rho)(Y)\right\Vert _{2}^{2}\right];
\]
and (iii) follows from  \eqref{eq:proof-preserve-3.5}.
\end{itemize}
Taking \eqref{eq:proof-preserve-1}, \eqref{eq:proof-preserve-3}
and \eqref{eq:proof-preserve-5} collectively gives
\begin{align*}
\ell_{N}(\rho^{\mathrm{W},\eta})-\ell_{N}(\rho) & \leq\alpha_{1}+\alpha_{2}\leq\bigg(-\eta+\frac{H\eta^{2}}{2c}+\frac{G^{2}\eta^{2}}{2c^{2}}\bigg)\mathbb{E}_{Y\sim\rho}\left[\left\Vert \nabla\delta\ell_{N}(\rho)(Y)\right\Vert _{2}^{2}\right]\\
 & =-\eta\bigg[1-\frac{\eta}{2c}\bigg(H+\frac{G^{2}}{c}\bigg)\bigg]\mathbb{E}_{Y\sim\rho}\left[\left\Vert \nabla\delta\ell_{N}(\rho)(Y)\right\Vert _{2}^{2}\right].
\end{align*}
Finally, we learn from  \eqref{eq:proof-preserve-6} and  \eqref{eq:proof-preserve-3.5}
that
\[
\left\Vert \nabla^{2}\delta\ell_{N}(\rho)(x)\right\Vert _{2}=\left\Vert \mathbb{E}_{X\sim\nu}\left[\frac{\nabla^{2}\phi(X-x)}{b\left(X\right)}\right]\right\Vert _{2}\leq\frac{\sup_{x\in\mathbb{R}^{d}}\|\nabla^{2}\phi(x)\|_{2}}{\inf_{x\in\mathbb{R}^{d}}b\left(x\right)}=\frac{H}{c}.
\]

\paragraph{Part 2.}

When $\eta<c/H$, the mapping ${x}\mapsto\eta\nabla\delta\ell_{N}(\rho)({x})$
is a contraction. Let $\varphi(x)=x-\eta\nabla\delta\ell_{N}(\rho)(x)$.
By Lemma \ref{lem-lipschitz-perturbation}, $\varphi:\mathbb{R}^{d}\to\mathbb{R}^{d}$
is a bijection and $\varphi^{-1}$ is Lipschitz. The second-order
differentiability of $\delta\ell_{N}(\rho)$ implies the differentiability
of $\varphi$ and thus $\varphi^{-1}$. If $\mathsf{supp}(\rho)=\mathbb{R}^{d}$,
then $\mathsf{supp}(\rho^{\mathrm{W},\eta})=\mathsf{supp}(\varphi_{\#}\rho)=\mathbb{R}^{d}$.

\subsection{Proof of Lemma \ref{lem-FR-KL} \label{sec-thm-FR-KL-proof}}

Let $\nu=N^{-1}\sum_{i=1}^{N}\delta_{X_{i}}$ be the empirical data
distribution. For any $\mu\ll\rho$, we have 
\begin{align*}
\ell_{N}(\mu) & =-\mathbb{E}_{X\sim\nu}\left[\log\left((\mu*\phi)(X)\right)\right]=-\mathbb{E}_{X\sim\nu}\left[\log\left(\mathbb{E}_{Y\sim\mu}\left[\phi(X-Y)\right]\right)\right]\\
 & =-\mathbb{E}_{X\sim\nu}\left[\log\left(\mathbb{E}_{Y\sim\rho}\left[\frac{\mathrm{d}\mu}{\mathrm{d}\rho}(Y)\cdot\phi(X-Y)\right]\right)\right]\\
 & =-\mathbb{E}_{X\sim\nu}\left[\log\left(\mathbb{E}_{Y\sim\rho}\left[\frac{\mathrm{d}\mu}{\mathrm{d}\rho}(Y)\cdot(\rho*\phi)(X)\cdot\frac{\phi(X-Y)}{(\rho*\phi)(X)}\right]\right)\right].
\end{align*}
For any $x\in\mathbb{R}^{d}$, we can define a new probability measure
$\rho_{x}^{\mathrm{FR},1}\in\mathcal{P}(\mathbb{R}^{d})$ through
\begin{equation}
\frac{\mathrm{d}\rho_{{x}}^{\mathrm{FR},1}}{\mathrm{d}\rho}(\cdot)=\frac{\phi(x-\cdot)}{(\rho*\phi)(x)},\label{eq:proof-FR-1}
\end{equation}
and we can check that $\rho_{x}^{\mathrm{FR},1}$ is indeed a probability
measure since
\[
\int_{\mathbb{R}^{d}}\mathrm{d}\rho_{x}^{\mathrm{FR},1}=\int_{\mathbb{R}^{d}}\frac{\phi(x-y)}{(\rho*\phi)(x)}\rho\left(\mathrm{d}x\right)=1.
\]
Then we have, by the convexity of $t\mapsto-\log t$ and Jensen's
inequality, 
\begin{align}
\ell_{N}(\mu) & \overset{\text{(i)}}{=}-\mathbb{E}_{X\sim\nu}\left[\log\left(\mathbb{E}_{Y\sim\rho_{X}^{\mathrm{FR},1}}\left[\frac{\mathrm{d}\mu}{\mathrm{d}\rho}(Y)\cdot(\rho*\phi)(X)\right]\right)\right]\nonumber \\
 & \overset{\text{(ii)}}{\leq}-\mathbb{E}_{X\sim\nu}\left[\mathbb{E}_{Y\sim\rho_{X}^{\mathrm{FR},1}}\left[\log\bigg(\frac{\mathrm{d}\mu}{\mathrm{d}\rho}(Y)\cdot(\rho*\phi)(X)\bigg)\right]\right]\nonumber \\
 & \overset{\text{(iii)}}{=}-\mathbb{E}_{X\sim\nu}\left[\mathbb{E}_{Y\sim\rho}\bigg[\log\bigg(\frac{\mathrm{d}\mu}{\mathrm{d}\rho}(Y)\cdot(\rho*\phi)(X)\bigg)\cdot\frac{\phi(X-Y)}{(\rho*\phi)(X)}\bigg]\right]\nonumber \\
 & =-\underbrace{\mathbb{E}_{X\sim\nu,Y\sim\rho}\left[\log\left(\frac{\mathrm{d}\mu}{\mathrm{d}\rho}(Y)\right)\cdot\frac{\phi(X-Y)}{(\rho*\phi)(X)}\right]}_{\eqqcolon\beta_{1}}-\underbrace{\mathbb{E}_{X\sim\nu,Y\sim\rho}\left[\log\left((\rho*\phi)(X)\right)\cdot\frac{\phi(X-Y)}{(\rho*\phi)(X)}\right]}_{\eqqcolon\beta_{2}}.\label{eq:proof-FR-2}
\end{align}
Here (i) and (iii) utilizes  \eqref{eq:proof-FR-1}, and (ii) follows
from Jensen's inequality and the convexity of $t\mapsto-\log t$.
Then we study the two terms $\beta_{1}$ and $\beta_{2}$ respectively.
Regarding $\beta_{1}$, we have
\begin{align}
\beta_{1} & =\mathbb{E}_{Y\sim\rho}\left[\log\left(\frac{\mathrm{d}\mu}{\mathrm{d}\rho}(Y)\right)\mathbb{E}_{X\sim\nu}\left[\frac{\phi(X-Y)}{(\rho*\phi)(X)}\right]\right]\nonumber \\
 & =\mathbb{E}_{Y\sim\rho}\left[\log\left(\frac{\mathrm{d}\mu}{\mathrm{d}\rho}(Y)\right)\cdot\frac{\mathrm{d}\rho^{\mathrm{FR},1}}{\mathrm{d}\rho}(Y)\right]=\mathbb{E}_{Y\sim\rho^{\mathrm{FR},1}}\left[\log\left(\frac{\mathrm{d}\mu}{\mathrm{d}\rho}(Y)\right)\right]\nonumber \\
 & =\mathbb{E}_{Y\sim\rho^{\mathrm{FR},1}}\left[\log\left(\frac{\mathrm{d}\mu}{\mathrm{d}\rho^{\mathrm{FR},1}}(Y)\right)\right]+\mathbb{E}_{Y\sim\rho^{\mathrm{FR},1}}\left[\log\left(\frac{\mathrm{d}\rho^{\mathrm{FR},1}}{\mathrm{d}\rho}(Y)\right)\right]\nonumber \\
 & =-\mathsf{KL}\left(\rho^{\mathrm{FR},1}\,\Vert\,\mu\right)+\mathsf{KL}\left(\rho^{\mathrm{FR},1}\,\Vert\,\rho\right).\label{eq:proof-FR-3}
\end{align}
Regarding $\beta_{2}$, we have
\begin{equation}
\beta_{2}=\mathbb{E}_{X\sim\nu}\left[\log\left((\rho*\phi)(X)\right)\mathbb{E}_{Y\sim\rho}\left[\frac{\phi(X-Y)}{(\rho*\phi)(X)}\right]\right]=\mathbb{E}_{X\sim\nu}\left[\log\left((\rho*\phi)(X)\right)\right]=-\ell_{N}\left(\rho\right).\label{eq:proof-FR-4}
\end{equation}
Taking  \eqref{eq:proof-FR-2},  \eqref{eq:proof-FR-3} and  \eqref{eq:proof-FR-4}
collectively yields
\[
\ell_{N}(\mu)\leq\ell_{N}(\rho)+\mathsf{KL}\left(\rho^{\mathrm{FR},1}\,\Vert\,\mu\right)-\mathsf{KL}\left(\rho^{\mathrm{FR},1}\,\Vert\,\rho\right),\qquad\forall\,\mu\ll\rho.
\]
By taking $\mu=\rho^{\mathrm{FR},1}$, we get 
\begin{equation}
\ell_{N}\left(\rho^{\mathrm{FR},1}\right)\leq\ell_{N}\left(\rho\right)-\mathsf{KL}\left(\rho^{\mathrm{FR},1}\,\Vert\,\rho\right).\label{eq:proof-FR-5}
\end{equation}
Recall that for any $\gamma\in(0,1)$ we have $\rho^{\mathrm{FR},\gamma}=(1-\gamma)\rho+\gamma\rho^{\mathrm{FR},1}$.
Therefore
\begin{equation}
\ell_{N}\left(\rho^{\mathrm{FR},\gamma}\right)\overset{\text{(i)}}{\leq}(1-\gamma)\ell_{N}(\rho)+\gamma\ell_{N}(\rho^{\mathrm{FR},1})\overset{\text{(ii)}}{\leq}\ell_{N}(\rho)-\gamma\mathsf{KL}\left(\rho^{\mathrm{FR},1}\,\Vert\,\rho\right),\label{eq:proof-FR-6}
\end{equation}
where (i) holds since $\ell_{N}(\rho)$ is $\ell_{2}$-convex in $\rho$,
and (ii) follows from  \eqref{eq:proof-FR-5}. By the $\ell_{2}$-convexity
of $\mathsf{KL}(\cdot\,\Vert\,\rho)$, we have
\begin{equation}
\mathsf{KL}\left(\rho^{\mathrm{FR},\gamma}\,\Vert\,\rho\right)\leq(1-\gamma)\mathsf{KL}\left(\rho\,\Vert\,\rho\right)+\gamma\mathsf{KL}\left(\rho^{\mathrm{FR},1}\,\Vert\,\rho\right)=\gamma\mathsf{KL}\left(\rho^{\mathrm{FR},1}\,\Vert\,\rho\right).\label{eq:proof-FR-7}
\end{equation}
Combine  \eqref{eq:proof-FR-6} and  \eqref{eq:proof-FR-7} to achieve
\[
\ell_{N}\left(\rho^{\mathrm{FR},\gamma}\right)\leq\ell_{N}(\rho)-\mathsf{KL}\left(\rho^{\mathrm{FR},\gamma}\,\Vert\,\rho\right).
\]

\subsection{Proof of Theorem \ref{thm-convergence-wfr-iteration} \label{sec-thm-convergence-wfr-iteration-proof}}

Suppose that for some $n\geq0$, $\ell_{N}(\rho_{n})\leq\ell_{N}(\rho_{0})$
holds, and $\mathsf{supp}(\rho_{n})=\mathbb{R}^{d}$. Define 
\[
c_{n}=\frac{e^{-\ell_{N}(\rho_{n})}\underline{\phi}(\inf\{r\geq0:~\bar{\phi}(r)\leq e^{-\ell_{N}(\rho_{n})}/2\}+\mathrm{diam}(\Omega))}{2\bar{\phi}(0)}.
\]
Then, $c_{n}\geq c_{0}$ and thus 
\[
0\leq\eta<\frac{c_{n}}{\sup_{x\in\mathbb{R}^{d}}\|\nabla^{2}\phi(x)\|_{2}+\sup_{x\in\mathbb{R}^{d}}\|\nabla\phi(x)\|_{2}^{2}/c_{n}}.
\]
Theorem \ref{lem-gd} and the upper bound on $\eta$ immediately gives
$\ell(\widetilde{\rho}_{n})\leq\ell(\rho_{n})$ and $\mathsf{supp}(\widetilde{\rho}_{n})=\mathbb{R}^{d}$.
Let $\nu=N^{-1}\sum_{i=1}^{N}\delta_{X_{i}}$ be the empirical data
distribution. From Lemma \ref{lem-FR-KL}, the updating rule 
\begin{align}
\frac{\mathrm{d}\rho_{n+1}}{\mathrm{d}\widetilde{\rho}_{n}}(\cdot)=(1-\gamma)+\gamma\mathbb{E}_{X\sim\nu}\left[\frac{\phi(X-\cdot)}{(\widetilde{\rho}_{n}*\phi)(X)}\right],\label{eqn-thm-wfr-5}
\end{align}
and the positivity of $\phi$, we see that $\ell(\rho_{n+1})\leq\ell(\widetilde{\rho}_{n})$
and $\mathsf{supp}(\rho_{n+1})=\mathbb{R}^{d}$. Therefore we can
use induction to show that, for all $n\geq0$, the inequality $\ell(\rho_{n+1})\leq\ell(\widetilde{\rho}_{n})\leq\ell(\rho_{n})$
holds, and $\mathsf{supp}(\rho_{n})=\mathbb{R}^{d}$. In view of \eqref{eq:loss_lower_bounded},
both sequences $\{\ell(\rho_{n})\}_{n=0}^{\infty}$ and $\{\ell(\widetilde{\rho}_{n})\}_{n=0}^{\infty}$
converge and have the same limit. Consequently, 
\begin{align}
\ell(\rho_{n})-\ell(\widetilde{\rho}_{n})\to0.\label{eqn-thm-wfr-0}
\end{align}
Now, suppose that $\{\rho_{n}\}_{n=0}^{\infty}$ converges weakly
to some $\rho_{\infty}\in\mathcal{P}(\mathbb{R}^{d})$.

\begin{claim}\label{claim-uniform-0} $\{\delta\ell_{N}(\rho_{n})\}_{n=0}^{\infty}$
converges uniformly to $\delta\ell_{N}(\rho_{\infty})$ over compact
sets. \end{claim}

\begin{proof}

It is easily seen that 
\[
\sup_{x\in\mathbb{R}^{d}}\|\nabla\delta\ell_{N}(\rho_{n})(x)\|_{2}=\sup_{x\in\mathbb{R}^{d}}\left\Vert \mathbb{E}_{{X}\sim\nu}\left[\frac{\nabla\phi({X}-x)}{(\rho_{n}*\phi)({X})}\right]\right\Vert _{2}\leq\frac{\sup_{x\in\mathbb{R}^{d}}\|\nabla\phi(x)\|_{2}}{\inf_{x\in\Omega}(\rho_{n}*\phi)(x)}.
\]
Assumption \ref{assumption-pdf-0} forces $\sup_{x\in\mathbb{R}^{d}}\|\nabla\phi(x)\|_{2}<\infty$.
By Lemma \eqref{lem-loss-bounds} and the fact that $\ell(\rho_{n})\leq\ell(\rho_{0})$,
we have
\begin{align}
 & \inf_{x\in\Omega}(\rho_{n}*\phi)(x)\geq c_{0},\qquad\forall n\geq0.\label{eqn-thm-wfr-1}
\end{align}
Hence, $\{\delta\ell_{N}(\rho_{n})\}_{n=0}^{\infty}$ are uniformly
equicontinuous. Therefore, it suffices to prove that $\{\delta\ell_{N}(\rho_{n})\}_{n=0}^{\infty}$
converges pointwise to $\delta\ell_{N}(\rho_{\infty})$.

Since $\phi$ is bounded and continuous (cf.~Assumption \ref{assumption-pdf-0}),
$(\rho_{n}*\phi)(x)\to(\rho_{\infty}*\phi)(x)$ holds for every $x\in\mathbb{R}^{d}$.
Recall that 
\begin{equation}
\delta\ell(\rho)(x)=-\mathbb{E}_{X\sim\nu}\left[\frac{\phi(X-x)}{(\rho*\phi)(X)}\right],\qquad\forall\rho\in\mathcal{P}(\mathbb{R}^{d}).\label{eqn-thm-wfr-2}
\end{equation}
Based on the boundedness of $\phi$ and the lower bound \eqref{eqn-thm-wfr-1},
we use the bounded convergence theorem to derive $\delta\ell_{N}(\rho_{n})(x)\to\delta\ell_{N}(\rho_{\infty})(x)$
for every $x\in\mathbb{R}^{d}$. This proves the claim. \end{proof}

\begin{claim}\label{claim-uniform-grad} $\{\nabla\delta\ell_{N}(\rho_{n})\}_{n=0}^{\infty}$
converges uniformly to $\nabla\delta\ell_{N}(\rho_{\infty})$ over
compact sets. \end{claim}

\begin{proof}The proof is similar to that of Claim \ref{claim-uniform-0}
and is thus omitted. \end{proof}

\begin{claim}\label{claim-uniform} $\{\delta\ell_{N}(\widetilde{\rho}_{n})\}_{n=0}^{\infty}$
converges uniformly to $\delta\ell_{N}(\rho_{\infty})$ over compact
sets. \end{claim}

\begin{proof}Lemma \ref{lem-gd} implies that $W_{2}^{2}(\widetilde{\rho}_{n},\rho_{n})\leq2\eta[\ell_{N}(\rho_{n})-\ell_{N}(\widetilde{\rho}_{n})]$.
By \eqref{eqn-thm-wfr-0}, $W_{2}(\widetilde{\rho}_{n},\rho_{n})\to0$
and thus $W_{1}(\widetilde{\rho}_{n},\rho_{n})\to0$. Since $\phi$
is Lipschitz, we have 
\[
\sup_{x\in\mathbb{R}^{d}}\left|(\widetilde{\rho}_{n}*\phi)(x)-(\rho_{n}*\phi)(x)\right|\to0.
\]
From the above uniform bound, \eqref{eqn-thm-wfr-1}, and \eqref{eqn-thm-wfr-2}
we obtain that 
\[
\sup_{x\in\mathbb{R}^{d}}\left|\delta\ell_{N}(\widetilde{\rho}_{n})(x)-\delta\ell_{N}(\rho_{n})(x)\right|\to0.
\]
Then, the proof is completed by applying Claim \ref{claim-uniform-0}.
\end{proof}

We now come back to Theorem \eqref{thm-convergence-wfr-iteration}.
Suppose that $\rho_{\infty}$ is not an optimal solution. Then in
view of Theorem \eqref{thm:NPMLE-basics}, there exists $\varepsilon>0$
such that $\delta\ell_{N}(\rho_{\infty})(x)<-1-\varepsilon$ for some
$x\in\mathbb{R}^{d}$.Similar to showing the pointwise convergence
of $\delta\ell_{N}(\rho_{n})$ to $\delta\ell_{N}(\rho_{\infty})$
in Claim \eqref{claim-uniform-0}, we can use the bounded convergence
theorem to show that $\ell(\rho_{n})\to\ell(\rho_{\infty})$. Hence,
$\ell(\rho_{\infty})\leq\ell(\rho_{0})$. Lemma \eqref{lem-loss-bounds}
immdeiately implies that 
\begin{align}
\inf_{x\in\Omega}(\rho_{\infty}*\phi)(x)\geq c_{0},\qquad\forall n\geq0,\label{eqn-thm-wfr-10}
\end{align}
and $\lim_{\|x\|_{2}\to\infty}\delta\ell(\rho_{\infty})(x)=0$. Therefore,
the function $\delta\ell(\rho_{\infty})$ achieves its minimum value
at some $x_{0}\in\mathbb{R}^{d}$. We have $\delta\ell_{N}(\rho_{\infty})(x_{0})<-1-\varepsilon$.

For notational simplicity, let $f_{n}(x)=\delta\ell_{N}(\rho_{n})(x)$
and $f(x)=\delta\ell_{N}(\rho_{\infty})(x)$. By Assumption \ref{assumption-pdf-0},
$f\in C^{d}(\mathbb{R}^{d})$. The second part of Lemma \eqref{lem-level-set}
guarantees the existence of some $y\in(-1-\varepsilon,-1-\varepsilon/2)$
such that 
\[
S(y)=\left\{ x\in\mathbb{R}^{d}:\,f(x)=y\right\} 
\]
is compact and $\inf_{x\in S(y)}\|\nabla f(x)\|_{2}>0$. Denote by
$\xi=\inf_{x\in S(y)}\|\nabla f(x)\|_{2}$ and $\bar{S}=\{x\in\mathbb{R}^{d}:\,f(x)\leq y\}$.
The sublevel set $\bar{S}$ is compact. According to the fact that
$f(x_{0})<y$ and the continuity of $f$, $\bar{S}$ has positive
Lebesgue measure. Therefore we have
\begin{align}
\rho_{n}(\bar{S})>0,\qquad\forall\,n\geq0.\label{eqn-thm-wfr-11}
\end{align}

\begin{itemize}
\item We first show that $\widetilde{\rho}_{n}(\bar{S})\geq\rho_{n}(\bar{S})$
holds for sufficently large $n$. By Claim \ref{claim-uniform-grad},
there exists $N>0$ such that 
\begin{align*}
\sup_{x\in\bar{S}}\|\nabla f_{n}(x)-\nabla f(x)\|_{2}\leq\xi/11,\qquad\forall\,n>N.
\end{align*}
By the assumed upper bound on $\eta$ and the estimate \eqref{eqn-thm-wfr-10},
we have 
\[
\eta\leq\frac{c_{0}}{\sup_{{x}\in\mathbb{R}^{d}}\|\nabla^{2}\phi(x)\|_{2}}\leq\frac{1}{\sup_{{x}\in\mathbb{R}^{d}}\|\nabla^{2}f(x)\|_{2}}.
\]
On top of the above, the first part of Lemma \eqref{eqn-lem-level-set}
implies that 
\[
\{x-\eta\nabla f_{n}(x):~x\in\bar{S}\}\subseteq\bar{S},\qquad\forall\,n>N.
\]
Since $\widetilde{\rho}_{n}=(\mathrm{id}-\eta\nabla f_{n})_{\#}\rho_{n}$,
we have 
\begin{align}
\widetilde{\rho}_{n}(\bar{S})\geq\rho_{n}(\bar{S}),\qquad\forall\,n>N.\label{eqn-thm-wfr-12}
\end{align}
\item Then we prove that $\rho_{n+1}(\bar{S})>(1+\varepsilon/4)\widetilde{\rho}_{n}(\bar{S})$
for sufficiently large $n$. By Claim \ref{claim-uniform}, there
exists $N'>0$ such that 
\begin{align*}
 & \sup_{x\in\bar{S}}|\delta\ell(\widetilde{\rho}_{n})(x)-f(x)|\leq\varepsilon/4,\qquad\forall\,n>N'.
\end{align*}
Consequently, 
\begin{align*}
\sup_{x\in\bar{S}}\delta\ell(\widetilde{\rho}_{n})(x)\leq\sup_{x\in\bar{S}}f(x)+\varepsilon/4\leq-1-\varepsilon/4,\qquad\forall\,n>N'.
\end{align*}
The expression \eqref{eqn-thm-wfr-5} implies that 
\[
\frac{\mathrm{d}\rho_{n+1}}{\mathrm{d}\widetilde{\rho}_{n}}(x)=(1-\gamma)+\gamma\cdot[-\delta\ell(\widetilde{\rho}_{n})(x)]\geq1+\gamma\varepsilon/4,\qquad\forall\,n>N',\,x\in\bar{S}.
\]
Therefore we have
\begin{align}
\rho_{n+1}(\bar{S})\geq(1+\gamma\varepsilon/4)\widetilde{\rho}_{n}(\bar{S}),\qquad\forall n>N'.\label{eqn-thm-wfr-13}
\end{align}
\end{itemize}
Taking \eqref{eqn-thm-wfr-12} and \eqref{eqn-thm-wfr-13} collectively
yields
\begin{align*}
\rho_{n+1}(\bar{S})\geq(1+\gamma\varepsilon/4)\rho_{n}(\bar{S}),\qquad\forall n>\max\{N,N'\}.
\end{align*}
This combined with \eqref{eqn-thm-wfr-11} immediately leads to $\lim_{n\to\infty}\rho_{n}(\bar{S})=\infty$,
which contradicts $\rho_{n}(\mathbb{R}^{d})=1$ for all $n\geq0$.
Therefore $\rho_{\infty}$ must be an optimal solution to \eqref{eq:NPMLE},
i.e.~$\rho_{\infty}$ is the NPMLE.

\subsection{Technical lemmas}

Here is a standard result about the Lipschitz perturbation of identity
mapping.

\begin{lemma}\label{lem-lipschitz-perturbation} Let $\mathbb{B}$
be a Banach space with norm $\|\cdot\|$. If $f:~\mathbb{B}\to\mathbb{B}$
has Lipschitz constant $c<1$, then $\varphi:~\mathbb{B}\to\mathbb{B}$,
$x\mapsto x+f(x)$ is bijective and $\varphi^{-1}$ is $(1-c)^{-1}$-Lipschitz.
\end{lemma}

\begin{proof} Choose any $y\in\mathbb{B}$. The mapping $\psi:~x\mapsto y-f(x)$
has Lipschitz constant $c<1$ (in terms of the norm $\|\cdot\|$).
By the Banach fixed-point theorem, there exists a unique $z$ such
that $z=\psi(z)$, which implies that $y=\varphi(z)$. Hence $\varphi$
is bijective and $\varphi^{-1}:~\mathbb{B}\to\mathbb{B}$ is well-defined.

For any $y_{1},y_{2}\in\mathbb{B}$, define $x_{i}=\varphi^{-1}(y_{i})$.
Then, $y_{i}=x_{i}+f(x_{i})$ and 
\begin{align*}
 & \|y_{2}-y_{1}\|_{2}=\|\varphi(x_{2})-\varphi(x_{1})\|_{2}\geq\|x_{2}-x_{1}\|_{2}-\|f(x_{2})-f(x_{1})\|_{2}\geq(1-c)\|x_{2}-x_{1}\|_{2},\\
 & \|\varphi^{-1}(y_{2})-\varphi^{-1}(y_{1})\|_{2}=\|x_{2}-x_{1}\|_{2}\leq(1-c)^{-1}\|y_{2}-y_{1}\|_{2}.
\end{align*}
This proves the Lipschitz smoothness of $\varphi^{-1}$. \end{proof}

Below we show that one step of approximate gradient descent cannot
expand certain sub-level sets of a function.

\begin{lemma}\label{lem-level-set} Assume that $f\in C^{2}(\mathbb{R}^{d})$
has a finite minimum value $y_{0}$. Define $S(y)=\{x\in\mathbb{R}^{d}:~f(x)=y\}$
for $y\in\mathbb{R}$. We have the followings. 
\begin{itemize}
\item Suppose that $\sup_{{x}\in\mathbb{R}^{d}}\|\nabla^{2}f(x)\|_{2}\leq L<\infty$.
Choose any $\eta\in[0,1/L]$ and $y>y_{0}$. Define $\bar{S}=\{x\in\mathbb{R}^{d}:\,f(x)\leq y\}$.
If $g:\mathbb{R}^{d}\to\mathbb{R}^{d}$ satisfies 
\[
\sup_{x\in\bar{S}}\|g(x)-\nabla f(x)\|_{2}\leq\frac{1}{11}\inf_{x\in S(y)}\|\nabla f(x)\|_{2},
\]
then 
\[
\{x-\eta g(x):~x\in\bar{S}\}\subseteq\bar{S}.
\]
\item Suppose that $\inf_{\|x\|_{2}\geq R}f(x)>y_{0}$ holds for some $R$
and in addition, $f\in C^{d}(\mathbb{R}^{d})$. Then, for any $\varepsilon>0$,
there exists $y\in(y_{0},y_{0}+\varepsilon)$ such that $S(y)$ is
compact and 
\[
\inf_{x\in S(y)}\|\nabla f(x)\|_{2}>0.
\]
\end{itemize}
\end{lemma}

\begin{proof} To prove the first part, we choose any $x\in\bar{S}$
and we will show that $x-\eta g(x)\in\bar{S}$. Let $\delta=\sup_{z\in\bar{S}}\|g(z)-\nabla f(z)\|_{2}$.
By $\sup_{{x}\in\mathbb{R}^{d}}\|\nabla^{2}f(x)\|_{2}\leq L$ and
$0\leq\eta\leq1/L$, 
\begin{align}
f(x-\eta g(x)) & \leq f(x)+\langle\nabla f(x),-\eta g(x)\rangle+\frac{L}{2}\|\eta g(x)\|_{2}^{2}\nonumber \\
 & \leq f(x)+\langle\nabla f(x),-\eta\nabla f(x)-\eta[g(x)-\nabla f(x)]\rangle+\frac{L\eta^{2}}{2}[\|\nabla f(x)\|_{2}+\|g(x)-\nabla f(x)\|_{2}]^{2}\nonumber \\
 & \leq f(x)-\eta\|\nabla f(x)\|_{2}^{2}+\eta\delta\|\nabla f(x)\|_{2}+\frac{L\eta^{2}}{2}[\|\nabla f(x)\|_{2}^{2}+2\delta\|\nabla f(x)\|_{2}+\delta^{2}]\nonumber \\
 & =f(x)+\eta\|\nabla f(x)\|_{2}^{2}\bigg[-\bigg(1-\frac{L\eta}{2}\bigg)+(1+L\eta)\frac{\delta}{\|\nabla f(x)\|_{2}}+\frac{L\eta}{2}\bigg(\frac{\delta}{\|\nabla f(x)\|_{2}}\bigg)^{2}\bigg]\nonumber \\
 & \leq f(x)+\eta\|\nabla f(x)\|_{2}^{2}\bigg[-\frac{1}{2}+\frac{2\delta}{\|\nabla f(x)\|_{2}}+\frac{1}{2}\bigg(\frac{\delta}{\|\nabla f(x)\|_{2}}\bigg)^{2}\bigg].\label{eqn-lem-level-set}
\end{align}
Define $\xi=\inf_{z\in S(y)}\|\nabla f(z)\|_{2}$. If $\xi=0$, then
$\delta=0$. The bound \eqref{eqn-lem-level-set} yields 
\[
f(x-\eta g(x))\leq f(x)-\frac{\eta}{2}\|\nabla f(x)\|_{2}^{2}\leq f(x)\leq y
\]
and thus $x-\eta g(x)\in\bar{S}$. From now on we assume that $\xi>0$.

\vspace{1em}

\noindent \textbf{Case 1. } When $\|\nabla f(x)\|_{2}>(\sqrt{5}+2)\delta$,
we have 
\[
\frac{\delta}{\|\nabla f(x)\|_{2}}<\frac{1}{\sqrt{5}+2}=\sqrt{5}-2.
\]
By \eqref{eqn-lem-level-set}, 
\begin{align*}
f(x-\eta g(x))-f(x) & \leq\eta\|\nabla f(x)\|_{2}^{2}\bigg[\frac{1}{2}\bigg(\frac{\delta}{\|\nabla f(x)\|_{2}}+2\bigg)^{2}-\frac{5}{2}\bigg]\leq0.
\end{align*}
Hence, $f(x-\eta g(x))\leq f(x)\leq y$ and $x-\eta g(x)\in\bar{S}$.

\vspace{1em}

\noindent \textbf{Case 2. } When $\|\nabla f(x)\|_{2}\leq(\sqrt{5}+2)\delta$,
we use $\delta\leq\xi/11$ and $\sqrt{5}<5/2$ to get 
\begin{align*}
\|\nabla f(x)\|_{2}\leq(\sqrt{5}+2)\delta\leq\frac{\sqrt{5}+2}{11}\xi<\frac{9\xi}{22}<\frac{\xi}{2}.
\end{align*}
Therefore, $x\notin S(y)$. As $x\in\bar{S}$, we must have $f(x)<y$.

Note that for any $z\in S(y)$, we have $\|\nabla f(z)\|_{2}\geq\xi>0$.
By $\sup_{{x}\in\mathbb{R}^{d}}\|\nabla^{2}f(x)\|_{2}\leq L$, we
have 
\[
\|z-x\|_{2}\geq\frac{\|\nabla f(z)-\nabla f(x)\|_{2}}{L}\geq\frac{\|\nabla f(z)\|_{2}-\|\nabla f(x)\|_{2}}{L}>\frac{\xi-\xi/2}{L}=\frac{\xi}{2L}.
\]
Therefore, $\inf_{z\in S(y)}\|z-x\|_{2}\geq\frac{\xi}{2L}$. We have
\begin{align}
B\cap S(y)=\varnothing,\qquad\text{where}\qquad B=\bigg\{ x'\in\mathbb{R}^{d}:\,\|x'-x\|_{2}<\frac{\xi}{2L}\bigg\}.\label{eqn-lem-level-set-1}
\end{align}
We claim that $B\subseteq\bar{S}$. If this is not true, then $f(x')>y$
holds for some $x'\in B$. Since $f(x)<y$ and $f$ is continuous,
there exists $t\in(0,1)$ such that $f((1-t)x+tx')=y$, i.e. $(1-t)x+tx'\in S(y)$.
The fact $(1-t)x+tx'\in B$ leads to $B\cap S(y)\neq\varnothing$,
which contradicts \eqref{eqn-lem-level-set-1}.

On the other hand, we have 
\begin{align*}
\|[x-\eta g(x)]-x\|_{2} & =\eta\|g(x)\|_{2}\leq\eta\|\nabla f(x)\|_{2}+\eta\delta\leq\frac{(\sqrt{5}+3)\delta}{L}\leq\frac{(\sqrt{5}+3)\xi}{11L}<\frac{\xi}{2L}.
\end{align*}
Therefore, $x-\eta g(x)\in B\subseteq\bar{S}$. This proves the first
part.

We now come to the second part. Let $y_{1}=\inf_{\|x\|_{2}\geq R}f(x)$.
Thanks to the continuity of $f$, the image set $\{f(x):\,x\in\mathbb{R}^{d}\}$
contains the interval $(y_{0},y_{1})$. Since $f:\mathbb{R}^{d}\to\mathbb{R}$
is $C^{d}$, Sard's lemma asserts that the set of critical
values, i.e. the image set of critical points $\{f(x):~\nabla f(x)=0\}$,
has Lebesgue measure 0. Consequently, for any $\varepsilon>0$, there
exists a regular value $y\in(y_{0},\min\{y_{0}+\varepsilon,~y_{1}\})$,
i.e. $\nabla f(x)\neq0$ so long as $f(x)=y$. Because of $y<y_{1}$,
$S(y)\subseteq\{x\in\mathbb{R}^{d}:\,\|x\|_{2}\geq R\}$ must be compact.
The continuity of $\nabla f$ implies that $\inf_{x\in S(y)}\|\nabla f(x)\|_{2}>0$.

\end{proof}

%% file: appendix_ode.tex
\section{Well-posedness of particle gradient flows}

\subsection{Fisher-Rao gradient flow (Proof of Theorem \ref{thm:FR-ODE}) \label{sec:proof-FR-ODE}}

In this section, we will show that for the ODE system  \eqref{eq:FR-ODE}
\[
\dot{\omega}_{t}^{(j)}=-\omega_{t}^{(j)}\left[1-\frac{1}{N}\sum_{i=1}^{N}\frac{\phi\left(X_{i}-\mu^{(j)}\right)}{\sum_{l=1}^{m}\omega_{t}^{(l)}\phi\left(X_{i}-\mu^{(l)}\right)}\right],\quad\forall\,t\geq0,\,j\in[m]
\]
with initial value $\omega_{0}\in\Delta^{m-1}$, the solution exists,
is unique, and $(\rho_{t})_{t\geq0}$ where $\rho_{t}\coloneqq\sum_{l=1}^{m}\omega_{t}^{(l)}\delta_{\mu^{(l)}}$
is a Fisher-Rao gradient flow in the sense of  \eqref{eq:Fisher-Rao-GF}. 

First of all, we will use Picard-Lindel\"of theorem to prove existence
and uniqueness of the solution. Define a function $f:\mathbb{R}^{m}\to\mathbb{R}^{m}$
as
\[
f\left(y\right)=\left[f_{j}\left(y\right)\right]_{1\leq j\leq m},\qquad f_{j}\left(y\right)=-y_{j}\left[1-\frac{1}{N}\sum_{i=1}^{N}\frac{\phi\left(X_{i}-\mu^{(j)}\right)}{\sum_{l=1}^{m}y_{l}\phi\left(X_{i}-\mu^{(l)}\right)}\right].
\]
Then we can rewrite the ODE system as $\dot{\omega}_{t}=f(\omega_{t})$.
For any $y\in\mathbb{R}^{m}$ such that $\Vert y-\omega_{0}\Vert_{2}\leq\varepsilon$
for some $\delta>0$ to be specified later, we know that
\begin{align*}
\sum_{l=1}^{m}y_{l}\phi\left(X_{i}-\mu^{(l)}\right) & \geq\sum_{l=1}^{m}\omega_{0}^{(l)}\phi\left(X_{i}-\mu^{(l)}\right)-\sum_{l=1}^{m}\left(\omega_{0}^{(l)}-y_{l}\right)\phi\left(X_{i}-\mu^{(l)}\right)\\
 & \overset{\text{(i)}}{\geq}\min_{i\in[N],l\in[m]}\phi\left(X_{i}-\mu^{(l)}\right)-\left\Vert y-\omega_{0}\right\Vert _{2}\sum_{l=1}^{m}\phi^{2}\left(X_{i}-\mu^{(l)}\right)\\
 & \overset{\text{(ii)}}{\geq}\min_{i\in[N],l\in[m]}\phi\left(X_{i}-\mu^{(l)}\right)-\frac{\varepsilon m}{\left(2\pi\right)^{d}}
\end{align*}
for any $i\in[N]$, where (i) follows from $\omega_{0}^{(l)}\in\Delta^{m-1}$
and Cauchy-Schwarz inequality, while (ii) holds since $\Vert\phi\Vert_{\infty}\leq(2\pi)^{-d/2}$.
Therefore by taking
\[
\varepsilon\coloneqq\frac{\left(2\pi\right)^{d}}{2m}\min_{i\in[N],l\in[m]}\phi\left(X_{i}-\mu^{(l)}\right),
\]
we know that for any $\Vert y-\omega_{0}\Vert_{2}\leq\varepsilon$
it holds that
\begin{equation}
\sum_{l=1}^{m}y_{l}\phi\left(X_{i}-\mu^{(l)}\right)\geq\frac{1}{2}\min_{i\in[N],l\in[m]}\phi\left(X_{i}-\mu^{(l)}\right)\triangleq\delta.\label{eq:proof-FR-ODE-1}
\end{equation}
Therefore we can check that
\begin{align*}
\left|\left[\nabla f_{j}\left(y\right)\right]_{j}\right| & =\left|-\left[1-\frac{1}{N}\sum_{i=1}^{N}\frac{\phi\left(X_{i}-\mu^{(j)}\right)}{\sum_{l=1}^{m}y_{l}\phi\left(X_{i}-\mu^{(l)}\right)}\right]-\left[\frac{1}{N}\sum_{i=1}^{N}\frac{y_{j}\phi^{2}\left(X_{i}-\mu^{(j)}\right)}{\left[\sum_{l=1}^{m}y_{l}\phi\left(X_{i}-\mu^{(l)}\right)\right]^{2}}\right]\right|\\
 & \leq1+\frac{\left\Vert \phi\right\Vert _{\infty}}{\delta}+\frac{\left(1+\varepsilon\right)\left\Vert \phi\right\Vert _{\infty}^{2}}{\delta^{2}}=1+\frac{1}{\left(2\pi\right)^{d/2}\delta}+\frac{1+\varepsilon}{\left(2\pi\right)^{d}\delta^{2}},
\end{align*}
and for $l\neq j$
\begin{align*}
\left|\left[\nabla f_{j}\left(y\right)\right]_{l}\right| & =\left|-y_{j}\left[\frac{1}{N}\sum_{i=1}^{N}\frac{\phi\left(X_{i}-\mu^{(j)}\right)\phi\left(X_{i}-\mu^{(l)}\right)}{\left[\sum_{l=1}^{m}y_{l}\phi\left(X_{i}-\mu^{(l)}\right)\right]^{2}}\right]\right|\leq\frac{\left(1+\varepsilon\right)\left\Vert \phi\right\Vert _{\infty}^{2}}{\delta^{2}}=\frac{1+\varepsilon}{\left(2\pi\right)^{d}\delta^{2}}.
\end{align*}
As a result, we know that for any 
\begin{equation}
\max_{y:\Vert y-\omega_{0}\Vert_{2}\leq\varepsilon}\left\Vert \nabla f_{j}\left(y\right)\right\Vert _{2}\leq\sqrt{m}\left(1+\frac{1}{\left(2\pi\right)^{d/2}\delta}+\frac{1+\varepsilon}{\left(2\pi\right)^{d}\delta^{2}}\right)\triangleq C_{\mathsf{lip}},\label{eq:proof-FR-ODE-2}
\end{equation}
and hence $f(y)$ is $C_{\mathsf{Lip}}$-Lipschitz continuous in $\{y:\Vert y-\omega_{0}\Vert_{2}\leq\varepsilon\}$
where $C_{\mathsf{Lip}}\coloneqq\sqrt{m}C_{\mathsf{lip}}$. In addition,
it is easy to show that 
\[
\max_{y:\Vert y-\omega_{0}\Vert_{2}\leq\varepsilon}\left\Vert f_{j}\left(y\right)\right\Vert _{2}\leq-y_{j}+\frac{1}{N}\sum_{i=1}^{N}\frac{y_{j}\phi\left(X_{i}-\mu^{(j)}\right)}{\sum_{l=1}^{m}y_{l}\phi\left(X_{i}-\mu^{(l)}\right)}\leq1+\varepsilon\triangleq M.
\]

By Picard-Lindel\"of theorem, there exists $t_{0}>0$ such that the
ODE has a unique solution on the time interval $[0,t_{0}]$. We first
check that $\omega_{t}^{(j)}>0$ for any $j\in[m]$ and $t\in[0,t_{0}]$.
If this is not true, define
\[
t_{\star}\coloneqq\min\left\{ t\in[0,t_{0}]:\,\exists\,j\in[m]\text{ s.t. }\omega_{t}^{(j)}\leq0\right\} 
\]
and suppose $\omega_{t_{\star}}^{(j_{\star})}\leq0$ for $j_{\star}\in[m]$.
Then we know that $\omega_{t}^{(j)}>0$ for any $j\in[m]$ and $0\leq t\leq t_{\star}$,
and hence $\dot{\omega}_{t}^{(j)}\geq-\omega_{t}^{(j)}$ for all $t\in[0,t_{\star}]$.
Then we can use Gr\"onwall's lemma to achieve $\omega_{t}^{(j)}\geq\omega_{0}^{(j)}e^{-t}$
for all $t\in[0,t_{\star}]$, and as a result $\omega_{t_{\star}}^{(j_{\star})}>0$,
which is a contradiction. In addition, we can also check that

\[
\frac{\mathrm{d}}{\mathrm{d}t}\sum_{j=1}^{m}\omega_{t}^{(j)}=\sum_{j=1}^{m}\dot{\omega}_{t}^{(j)}=-\sum_{j=1}^{m}\omega_{t}^{(j)}\left[1-\frac{1}{N}\sum_{i=1}^{N}\frac{\phi\left(X_{i}-\mu^{(j)}\right)}{\sum_{l=1}^{m}\omega_{t}^{(l)}\phi\left(X_{i}-\mu^{(l)}\right)}\right]=0
\]
for all $t\in[0,t_{0}]$. As as a result $\omega_{t}\in\Delta^{m-1}$
for any $0\leq t\leq t_{0}$. By repeating the same procedure as above
(notice that the above proof only depends on $\omega_{0}\in\Delta^{m-1}$,
and $t_{0}$ only depends on universal constants $C_{\mathsf{Lip}}$
and $M$ and does not depend on $\omega_{0}$), we can show that the
ODE has a unique solution on $[t_{0},2t_{0}]$, $[2t_{0},3t_{0}]$,
and so on. This shows the existence and uniqueness of the solution
to the ODE system  \eqref{eq:FR-ODE}.

Next, we show that $(\rho_{t})_{t\geq0}$ defined as $\rho_{t}\coloneqq\sum_{l=1}^{m}\omega_{t}^{(l)}\delta_{\mu^{(l)}}$
solves  \eqref{eq:Fisher-Rao-GF}. Note that $\rho_{t}$ is a probability
measure since we have shown that $\omega_{t}\in\Delta^{m-1}$ for
any $t\geq0$. For any test function $\varphi(x)\in C_{\mathrm{c}}^{\infty}$,
we have
\begin{align*}
\frac{\mathrm{d}}{\mathrm{d}t}\int_{\mathbb{R}^{d}}\varphi\left(x\right)\rho_{t}\left(\mathrm{d}x\right) & =\frac{\mathrm{d}}{\mathrm{d}t}\left[\sum_{j=1}^{m}\omega_{t}^{(j)}\varphi\left(\mu_{t}^{(j)}\right)\right]=\sum_{j=1}^{m}\dot{\omega}_{t}^{(j)}\varphi\left(\mu_{t}^{(j)}\right)\\
 & =-\sum_{j=1}^{m}\left[1+\delta\ell_{N}\left(\rho_{t}\right)\left(\mu_{t}^{(j)}\right)\right]\omega_{t}^{(j)}\varphi\left(\mu_{t}^{(j)}\right)\\
 & =-\int_{\mathbb{R}^{d}}\left[1+\delta\ell_{N}\left(\rho_{t}\right)\left(x\right)\right]\varphi\left(x\right)\rho_{t}\left(\mathrm{d}x\right).
\end{align*}
This proves that 
\[
\partial_{t}\rho_{t}=-\left[\delta\ell\left(\rho_{t}\right)+1\right]\rho_{t}
\]
holds in the sense of distributions.

\subsection{Wasserstein gradient flow (Proof of Theorem \ref{thm:Wass-ODE})\label{sec:proof-Wass-ODE}}

In this section we will show that the ODE system  \eqref{eq:Wass-ODE}
\[
\dot{\mu}_{t}^{(j)}=\frac{1}{N}\sum_{i=1}^{N}\frac{\phi\left(X_{i}-\mu_{t}^{(j)}\right)}{m^{-1}\sum_{l=1}^{m}\phi\left(X_{i}-\mu_{t}^{(l)}\right)}\left(X_{i}-\mu_{t}^{(j)}\right),\quad\forall\,t\geq0,\,j\in[m]
\]
has unique solution, and $(\rho_{t})_{t\geq0}$ where $\rho_{t}\coloneqq m^{-1}\sum_{l=1}^{m}\delta_{\mu_{t}^{(l)}}$
is a Wasserstein gradient flow in the sense of  \eqref{eq:Wass-GF}. 

We will use Picard-Lindel\"of theorem to prove existence and uniqueness
of the solution. Define a sufficiently large constant
\[
R\coloneqq\max_{1\leq i\leq N}\left\Vert X_{i}\right\Vert _{2}.
\]
For each $j\in[m]$, define a function $f^{(j)}:\mathbb{R}^{md}\to\mathbb{R}^{d}$
as
\[
f^{(j)}\left(z\right)=\frac{1}{N}\sum_{i=1}^{N}\frac{\phi\left(X_{i}-z_{j}\right)}{m^{-1}\sum_{l=1}^{m}\phi\left(X_{i}-z_{l}\right)}\left(X_{i}-z_{j}\right),
\]
where $z=[z_{j}]_{1\leq j\leq m}\in\mathbb{R}^{md}$ and $z_{1},\ldots,z_{m}\in\mathbb{R}^{d}$,
and let $f(z)=[f^{(j)}(z)]_{1\leq j\leq m}$. Then we can write the
ODE system as $\dot{\mu}_{t}=f(\mu_{t})$ where $\mu_{t}=[\mu_{t}^{(j)}]_{1\leq j\leq m}$.
Denote by $f^{(j)}(z)=[f_{k}^{(j)}(z)]_{1\leq k\leq d}$. For any
$z\in\mathbb{R}^{md}$ satisfying $\max_{j\in[m]}\Vert z_{j}\Vert_{2}\leq2R$,
we have
\begin{equation}
\min_{1\leq i\leq N}\frac{1}{m}\sum_{l=1}^{m}\phi\left(X_{i}-z_{l}\right)\geq\frac{1}{\left(2\pi\right)^{d/2}}\exp\left(-\frac{9}{2}R^{2}\right)\triangleq\delta.\label{eq:proof-Wass-ODE-1}
\end{equation}
Then we can compute for $l\neq j$
\begin{align*}
\left\Vert \nabla_{z_{l}}f_{k}^{(j)}\left(z\right)\right\Vert _{2} & =\left\Vert \nabla_{z_{l}}\frac{1}{N}\sum_{i=1}^{N}\frac{\phi\left(X_{i}-z_{j}\right)}{m^{-1}\sum_{l=1}^{m}\phi\left(X_{i}-z_{l}\right)}e_{k}^{\top}\left(X_{i}-z_{j}\right)\right\Vert _{2}\\
 & =\left\Vert -\frac{1}{N}\sum_{i=1}^{N}\frac{\phi\left(X_{i}-z_{j}\right)m^{-1}\phi\left(X_{i}-z_{l}\right)}{\left[m^{-1}\sum_{l=1}^{m}\phi\left(X_{i}-z_{l}\right)\right]^{2}}e_{k}^{\top}\left(X_{i}-z_{j}\right)\left(X_{i}-z_{l}\right)\right\Vert _{2}\\
 & \overset{\text{(i)}}{\leq}\frac{m^{-1}\left\Vert \phi\right\Vert _{\infty}^{2}}{\delta^{2}}\frac{1}{N}\sum_{i=1}^{N}\left|e_{k}^{\top}\left(X_{i}-z_{j}\right)\right|\left\Vert X_{i}-z_{l}\right\Vert _{2}\overset{\text{(ii)}}{\leq}\frac{9m^{-1}}{\delta^{2}\left(2\pi\right)^{d}}R^{2},
\end{align*}
where (i) utilizes  \eqref{eq:proof-Wass-ODE-1} and (ii) follows from
$\max_{i\in[N]}\Vert X_{i}\Vert_{2}\leq R$ and $\max_{j\in[m]}\Vert z_{j}\Vert_{2}\leq2R$.
Similarly we have
\begin{align*}
\left\Vert \nabla_{z_{j}}f_{k}^{(j)}\left(z\right)\right\Vert _{2} & =\left\Vert \nabla_{z_{j}}\frac{1}{N}\sum_{i=1}^{N}\frac{\phi\left(X_{i}-z_{j}\right)}{m^{-1}\sum_{l=1}^{m}\phi\left(X_{i}-z_{l}\right)}e_{k}^{\top}\left(X_{i}-z_{j}\right)\right\Vert _{2}\\
 & \leq\left\Vert -\frac{1}{N}\sum_{i=1}^{N}\frac{\phi^{2}\left(X_{i}-z_{j}\right)m^{-1}}{\left[m^{-1}\sum_{l=1}^{m}\phi\left(X_{i}-z_{l}\right)\right]^{2}}e_{k}^{\top}\left(X_{i}-z_{j}\right)\left(X_{i}-z_{j}\right)\right\Vert _{2}\\
 & \quad+\left\Vert \frac{1}{N}\sum_{i=1}^{N}\frac{\phi\left(X_{i}-z_{j}\right)}{m^{-1}\sum_{l=1}^{m}\phi\left(X_{i}-z_{l}\right)}\left[e_{k}^{\top}\left(X_{i}-z_{j}\right)\left(X_{i}-z_{j}\right)-e_{k}\right]\right\Vert _{2}\\
 & \leq9\left(\frac{m^{-1}}{\delta^{2}\left(2\pi\right)^{d}}+\frac{1}{\delta\left(2\pi\right)^{d/2}}\right)R^{2}+\frac{\left\Vert \phi\right\Vert _{\infty}}{\delta}.
\end{align*}
As a result, we have
\begin{equation}
\left\Vert \nabla_{z}f_{k}^{(j)}\left(z\right)\right\Vert _{2}=\sqrt{\sum_{l=1}^{m}\left\Vert \nabla_{z_{l}}f_{k}^{(j)}\left(z\right)\right\Vert _{2}^{2}}\leq\sqrt{\left(m+1\right)\left(\frac{9m^{-1}}{\delta^{2}\left(2\pi\right)^{d}}R^{2}\right)^{2}+2\frac{1}{\delta^{2}\left(2\pi\right)^{d}}}\triangleq C_{\mathsf{lip}}.\label{eq:proof-Wass-ODE-2}
\end{equation}
Therefore $f^{(j)}(z)$ is $\sqrt{d}C_{\mathsf{Lip}}$-Lipschitz continuous
in $\{z:\max_{j\in[m]}\Vert z_{j}\Vert_{2}\leq2R\}$, and hence $f(z)$
is $C_{\mathsf{Lip}}$-Lipschitz continous where $C_{\mathsf{Lip}}\coloneqq C_{\mathsf{lip}}\sqrt{md}$.
In addition, it is straightforward to show that for any $z\in\mathbb{R}^{md}$
satisfying $\max_{j\in[m]}\Vert z_{j}\Vert_{2}\leq2R$, 
\[
\left\Vert f^{(j)}\left(z\right)\right\Vert _{2}\leq\frac{\left\Vert \phi\right\Vert _{\infty}}{\delta}3R\triangleq M
\]
holds for all $1\leq j\leq m$.

Recall that $\mu_{0}^{(1)},\ldots,\mu_{0}^{(m)}$ are i.i.d.~sampled
from $\mathsf{Uniform}(\{X_{i}\}_{1\leq i\leq N})$, therefore $\max_{j\in[m]}\Vert\mu_{0}^{(j)}\Vert_{2}\leq R$.
Therefore we have
\[
\left\{ z:\left\Vert z-\mu_{0}\right\Vert _{2}\leq R\right\} \subseteq\left\{ z:\max_{j\in[m]}\left\Vert z_{j}\right\Vert _{2}\leq2R\right\} ,
\]
where $\mu_{0}=[\mu_{0}^{(j)}]_{1\leq j\leq m}$. Hence $f(z)$ is
$\sqrt{md}C_{\mathsf{Lip}}$-Lipschitz continous in $\{z:\Vert z-\mu_{0}\Vert_{2}\leq R\}$.
Then we can use Picard-Lindel\"of theorem to show that, there exists
$t_{0}>0$ such that the ODE has a unique solution on the time interval
$[0,t_{0}]$. For any $t\in[0,t_{0}]$ and $j\in[m]$, we can compute
\begin{align*}
\frac{\mathrm{d}}{\mathrm{d}t}\Vert\mu_{t}^{(j)}\Vert_{2}^{2} & =2\langle\mu_{t}^{(j)},\dot{\mu}_{t}^{(j)}\rangle=\frac{2}{N}\sum_{i=1}^{N}\frac{\phi\left(X_{i}-\mu_{t}^{(j)}\right)}{m^{-1}\sum_{l=1}^{m}\phi\left(X_{i}-\mu_{t}^{(l)}\right)}\mu_{t}^{(j)\top}\left(X_{i}-\mu_{t}^{(j)}\right)\\
 & \leq\frac{2}{N}\sum_{i=1}^{N}\frac{\phi\left(X_{i}-\mu_{t}^{(j)}\right)}{m^{-1}\sum_{l=1}^{m}\phi\left(X_{i}-\mu_{t}^{(l)}\right)}\left(\left\Vert X_{i}\right\Vert _{2}\Vert\mu_{t}^{(j)}\Vert_{2}-\Vert\mu_{t}^{(j)}\Vert_{2}^{2}\right)\\
 & \leq\frac{2}{N}\sum_{i=1}^{N}\frac{\phi\left(X_{i}-\mu_{t}^{(j)}\right)}{m^{-1}\sum_{l=1}^{m}\phi\left(X_{i}-\mu_{t}^{(l)}\right)}\left(R-\Vert\mu_{t}^{(j)}\Vert_{2}\right)\Vert\mu_{t}^{(j)}\Vert_{2},
\end{align*}
where we use Cauchy-Schwarz inequality in the penultimate step. This
shows that $\frac{\mathrm{d}}{\mathrm{d}t}\Vert\mu_{t}^{(j)}\Vert_{2}^{2}<0$
as long as $\Vert\mu_{t_{0}}^{(j)}\Vert_{2}>R$, and as a result $\max_{j\in[m]}\Vert\mu_{t_{0}}^{(j)}\Vert_{2}\leq R$.
Then we can repeat the same analysis as above (notice that the above
proof only requires $\max_{j\in[m]}\Vert\mu_{0}^{(j)}\Vert_{2}\leq R$,
and $t_{0}$ only depends on universal constants $C_{\mathsf{Lip}}$
and $M$) to show that the ODE has a unique solution on $[t_{0},2t_{0}]$,
$[2t_{0},3t_{0}]$, and so on. This shows the existence and uniqueness
of the solution to the ODE system  \eqref{eq:FR-ODE}. 

Finally we check that $(\rho_{t})_{t\geq0}$ defined as $\rho_{t}\coloneqq\sum_{l=1}^{m}m^{-1}\delta_{\mu_{t}^{(l)}}$
solves  \eqref{eq:Fisher-Rao-GF}. For any test function $\varphi(x)\in C_{\mathrm{c}}^{\infty}$,
we have
\begin{align*}
\frac{\mathrm{d}}{\mathrm{d}t}\int_{\mathbb{R}^{d}}\varphi\left(x\right)\rho_{t}\left(\mathrm{d}x\right) & =\frac{\mathrm{d}}{\mathrm{d}t}\left[\frac{1}{m}\sum_{j=1}^{m}\varphi\left(\mu_{t}^{(j)}\right)\right]=\frac{1}{m}\sum_{j=1}^{m}\left\langle \nabla\varphi\left(\mu_{t}^{(j)}\right),\dot{\mu}_{t}^{(j)}\right\rangle \\
 & =\frac{1}{m}\sum_{j=1}^{m}\left\langle \nabla\varphi\left(\mu_{t}^{(j)}\right),-\nabla\delta\ell_{N}\left(\rho_{t}\right)\left(\mu_{t}^{(j)}\right)\right\rangle \\
 & =-\int_{\mathbb{R}^{d}}\left\langle \nabla\varphi\left(x\right),\nabla\delta\ell_{N}\left(\rho_{t}\right)\right\rangle \rho_{t}\left(\mathrm{d}x\right).
\end{align*}
This proves that 
\[
\partial_{t}\rho_{t}=\mathsf{div}\left(\rho_{t}\nabla\delta\ell\left(\rho_{t}\right)\right)
\]
holds in the sense of distributions.

\subsection{Wasserstein-Fisher-Rao gradient flow (Proof of Theorem \ref{thm:WFR-ODE})
\label{sec:proof-WFR-ODE}}

In this section we will show that the ODE system  \eqref{eq:WFR-ODE}
\begin{align*}
\dot{\mu}_{t}^{(j)} & =\frac{1}{N}\sum_{i=1}^{N}\frac{\phi\left(X_{i}-\mu_{t}^{(j)}\right)}{m^{-1}\sum_{l=1}^{m}\phi\left(X_{i}-\mu_{t}^{(l)}\right)}\left(X_{i}-\mu_{t}^{(j)}\right),\\
\dot{\omega}_{t}^{(j)} & =\left[\frac{1}{N}\sum_{i=1}^{N}\frac{\phi\left(X_{i}-\mu_{t}^{(j)}\right)}{\sum_{l=1}^{m}\omega_{t}^{(j)}\phi\left(X_{i}-\mu_{t}^{(l)}\right)}-1\right]\omega_{t}^{(j)},
\end{align*}
has unique solution, and $(\rho_{t})_{t\geq0}$ where $\rho_{t}\coloneqq\sum_{l=1}^{m}\omega_{t}^{(l)}\delta_{\mu_{t}^{(l)}}$
is a Wasserstein-Fisher-Rao gradient flow in the sense of  \eqref{eq:WFR-GF}.
We will integrate the proof techniques used in the previous two sections.

We will again use Picard-Lindel\"of theorem to prove existence and
uniqueness of the solution. For each $j\in[m]$, define two functions
$f^{(j)}:\Delta^{m-1}\times\mathbb{R}^{md}\to\mathbb{R}^{d}$ and
$g^{(j)}:\Delta^{m-1}\times\mathbb{R}^{md}\to\mathbb{R}$ as
\begin{align*}
f^{(j)}\left(y,z\right) & =\frac{1}{N}\sum_{i=1}^{N}\frac{\phi\left(X_{i}-z_{j}\right)}{\sum_{l=1}^{m}y_{l}\phi\left(X_{i}-z_{l}\right)}\left(X_{i}-z_{j}\right),\\
g^{(j)}\left(y,z\right) & =-\left[1-\frac{1}{N}\sum_{i=1}^{N}\frac{\phi\left(X_{i}-z_{j}\right)}{\sum_{l=1}^{m}y_{l}\phi\left(X_{i}-z_{l}\right)}\right]y_{j},
\end{align*}
where $y=[y_{j}]_{1\leq j\leq m}\in\Delta^{m-1}$, $z=[z_{j}]_{1\leq j\leq m}\in\mathbb{R}^{md}$
with $z_{1},\ldots,z_{m}\in\mathbb{R}^{d}$. Let
\[
f\left(y,z\right)\coloneqq\left[\begin{array}{c}
f^{(1)}\left(y,z\right)\\
\vdots\\
f^{(m)}\left(y,z\right)
\end{array}\right],\quad g\left(y,z\right)\coloneqq\left[\begin{array}{c}
g^{(1)}\left(y,z\right)\\
\vdots\\
g^{(m)}\left(y,z\right)
\end{array}\right],\quad\text{and}\quad h\left(y,z\right)=\left[\begin{array}{c}
f\left(y,z\right)\\
g\left(y,z\right)
\end{array}\right].
\]
Then we can write the ODE system as
\[
\left[\begin{array}{c}
\dot{\mu}_{t}\\
\dot{\omega}_{t}
\end{array}\right]=h\left(\left[\begin{array}{c}
\mu_{t}\\
\omega_{t}
\end{array}\right]\right),
\]
where $\mu_{t}=[\mu_{t}^{(j)}]_{1\leq j\leq m}$ and $\omega_{t}=[\omega_{t}^{(j)}]_{1\leq j\leq m}$.
Denote by $f^{(j)}(z)=[f_{k}^{(j)}(z)]_{1\leq k\leq d}$. 

For any $z\in\mathbb{R}^{md}$ satisfying $\max_{j\in[m]}\Vert z_{j}\Vert_{2}\leq2R$
and any $y\in\mathbb{R}^{m}$ satisfying $\Vert y-\omega_{0}\Vert_{2}\leq\varepsilon$
where
\[
R\coloneqq\max_{1\leq i\leq N}\left\Vert X_{i}\right\Vert _{2},\qquad\varepsilon\coloneqq\frac{\left(2\pi\right)^{d/2}}{2m}\exp\left(-\frac{9}{2}R^{2}\right),
\]
we have for any $i\in[N]$
\begin{align}
\sum_{l=1}^{m}y_{l}\phi\left(X_{i}-z_{l}\right) & \geq\sum_{l=1}^{m}\omega_{0}^{(l)}\phi\left(X_{i}-z_{l}\right)-\sum_{l=1}^{m}\left(\omega_{0}^{(l)}-y_{l}\right)\phi\left(X_{i}-z_{l}\right)\nonumber \\
 & \overset{\text{(i)}}{\geq}\min_{i\in\left[N\right],l\in\left[m\right]}\phi\left(X_{i}-z_{l}\right)-\left\Vert y-\omega_{0}\right\Vert _{2}\sum_{l=1}^{m}\phi^{2}\left(X_{i}-\mu^{(l)}\right)\nonumber \\
 & \overset{\text{(ii)}}{\geq}\frac{1}{\left(2\pi\right)^{d/2}}\exp\left(-\frac{9}{2}R^{2}\right)-\frac{\varepsilon m}{\left(2\pi\right)^{d}}=\frac{1}{2\left(2\pi\right)^{d/2}}\exp\left(-\frac{9}{2}R^{2}\right)\triangleq\delta.\label{eq:proof-WFR-ODE-1}
\end{align}
Here (i) follows from $\omega_{0}\in\Delta^{m-1}$ and the Cauchy-Schwarz
inequality, while (ii) and (iii) holds since $\Vert X_{i}-z_{l}\Vert_{2}\leq3R$
for any $i\in[N]$ and $l\in[m]$. For any $j\in[m]$, denote by $f^{(j)}(z)=[f_{k}^{(j)}(z)]_{1\leq k\leq d}$.
Similar to the proof of  \eqref{eq:proof-Wass-ODE-2}, we can use  \eqref{eq:proof-WFR-ODE-1}
to show that 
\[
\left\Vert \nabla_{z}f_{k}^{(j)}\left(y,z\right)\right\Vert _{2}\leq\sqrt{\left(m+1\right)\left(\frac{9m^{-1}}{\delta^{2}\left(2\pi\right)^{d}}R^{2}\right)^{2}+2\frac{1}{\delta^{2}\left(2\pi\right)^{d}}}.
\]
We also have
\begin{align*}
\nabla_{y}f_{k}^{(j)}\left(y,z\right) & =\nabla_{y}\frac{1}{N}\sum_{i=1}^{N}\frac{\phi\left(X_{i}-z_{j}\right)}{\sum_{l=1}^{m}y_{l}\phi\left(X_{i}-z_{l}\right)}e_{k}^{\top}\left(X_{i}-z_{j}\right)\\
 & =-\frac{1}{N}\sum_{i=1}^{N}\frac{\phi\left(X_{i}-z_{j}\right)}{\left[\sum_{l=1}^{m}y_{l}\phi\left(X_{i}-z_{l}\right)\right]^{2}}e_{k}^{\top}\left(X_{i}-z_{j}\right)\left[\begin{array}{c}
\phi\left(X_{i}-z_{1}\right)\\
\vdots\\
\phi\left(X_{i}-z_{m}\right)
\end{array}\right],
\end{align*}
and as a result
\[
\left\Vert \nabla_{y}f_{k}^{(j)}\left(y,z\right)\right\Vert _{2}\leq\frac{\sqrt{m}\left\Vert \phi\right\Vert _{\infty}^{2}}{\delta^{2}}\max_{i\in[N],j\in[m]}\left\Vert X_{i}-z_{j}\right\Vert _{2}\leq\frac{3\sqrt{m}R}{\delta^{2}\left(2\pi\right)^{d}}.
\]
In addition, we can compute
\begin{align*}
\nabla_{y}g^{(j)}\left(y,z\right) & =-\nabla_{y}\left[\left(1-\frac{1}{N}\sum_{i=1}^{N}\frac{\phi\left(X_{i}-z_{j}\right)}{\sum_{l=1}^{m}y_{l}\phi\left(X_{i}-z_{l}\right)}\right)y_{j}\right]\\
 & =-\left[1-\frac{1}{N}\sum_{i=1}^{N}\frac{\phi\left(X_{i}-z_{j}\right)}{\sum_{l=1}^{m}y_{l}\phi\left(X_{i}-z_{l}\right)}\right]e_{j}+\frac{1}{N}\sum_{i=1}^{N}\frac{y_{j}\phi\left(X_{i}-z_{j}\right)}{\left[\sum_{l=1}^{m}y_{l}\phi\left(X_{i}-z_{l}\right)\right]^{2}}\left[\begin{array}{c}
\phi\left(X_{i}-z_{1}\right)\\
\vdots\\
\phi\left(X_{i}-z_{m}\right)
\end{array}\right]
\end{align*}
and for each $l\in[m]$
\begin{align*}
\nabla_{z_{l}}g^{(j)}\left(y,z\right) & =-\nabla_{z_{l}}\left[\left(1-\frac{1}{N}\sum_{i=1}^{N}\frac{\phi\left(X_{i}-z_{j}\right)}{\sum_{l=1}^{m}y_{l}\phi\left(X_{i}-z_{l}\right)}\right)y_{j}\right]\\
 & =y_{j}\frac{1}{N}\sum_{i=1}^{N}\left[\frac{\nabla_{z_{l}}\phi\left(X_{i}-z_{j}\right)}{\sum_{l=1}^{m}y_{l}\phi\left(X_{i}-z_{l}\right)}-\frac{\phi\left(X_{i}-z_{j}\right)y_{l}\nabla_{z_{l}}\phi\left(X_{i}-z_{l}\right)}{\left[\sum_{l=1}^{m}y_{l}\phi\left(X_{i}-z_{l}\right)\right]^{2}}\right]\\
 & =\frac{1}{N}\sum_{i=1}^{N}\left[\ind\left\{ l=j\right\} \frac{y_{j}\phi\left(X_{i}-z_{j}\right)}{\sum_{l=1}^{m}y_{l}\phi\left(X_{i}-z_{l}\right)}\left(X_{i}-z_{j}\right)-\frac{y_{j}\phi\left(X_{i}-z_{j}\right)y_{l}\phi\left(X_{i}-z_{l}\right)}{\left[\sum_{l=1}^{m}y_{l}\phi\left(X_{i}-z_{l}\right)\right]^{2}}\left(X_{i}-z_{l}\right)\right].
\end{align*}
As a result, we have
\begin{align*}
\left\Vert \nabla_{y}g^{(j)}\left(y,z\right)\right\Vert _{2} & \leq\left|1-\frac{1}{N}\sum_{i=1}^{N}\frac{\phi\left(X_{i}-z_{j}\right)}{\sum_{l=1}^{m}y_{l}\phi\left(X_{i}-z_{l}\right)}\right|+\frac{1}{N}\sum_{i=1}^{N}\left|\frac{y_{j}\phi\left(X_{i}-z_{j}\right)}{\left[\sum_{l=1}^{m}y_{l}\phi\left(X_{i}-z_{l}\right)\right]^{2}}\right|\sqrt{m}\left\Vert \phi\right\Vert _{\infty}\\
 & \leq1+\frac{\left\Vert \phi\right\Vert _{\infty}}{\delta}+\frac{\left(1+\varepsilon\right)\sqrt{m}}{\delta^{2}}\left\Vert \phi\right\Vert _{\infty}^{2}=1+\frac{1}{\left(2\pi\right)^{d/2}\delta}+\frac{\left(1+\varepsilon\right)\sqrt{m}}{\left(2\pi\right)^{d}\delta^{2}}
\end{align*}
and
\begin{align*}
\left\Vert \nabla_{z_{l}}g^{(j)}\left(y,z\right)\right\Vert _{2} & \leq\frac{1}{N}\sum_{i=1}^{N}\frac{\left(1+\varepsilon\right)\left\Vert \phi\right\Vert _{\infty}}{\delta}\left\Vert X_{i}-z_{j}\right\Vert _{2}+\frac{1}{N}\sum_{i=1}^{N}\frac{\left(1+\varepsilon\right)^{2}\left\Vert \phi\right\Vert _{\infty}^{2}}{\delta^{2}}\left\Vert X_{i}-z_{l}\right\Vert _{2}\\
 & \leq\frac{3\left(1+\varepsilon\right)R}{\left(2\pi\right)^{d/2}\delta}+\frac{3\left(1+\varepsilon\right)^{2}R\left\Vert \phi\right\Vert _{\infty}^{2}}{\left(2\pi\right)^{d}\delta^{2}}.
\end{align*}
 Therefore for any $k\in[d]$, $j\in[m]$, $z\in\mathbb{R}^{md}$
satisfying $\max_{j\in[m]}\Vert z_{j}\Vert_{2}\leq2R$ and $y\in\mathbb{R}^{m}$
such that $\Vert y-\omega_{0}\Vert_{2}\leq\varepsilon$, we have
\begin{align*}
\left\Vert \nabla f_{k}^{(j)}\left(y,z\right)\right\Vert _{2} & =\sqrt{\left\Vert \nabla_{z}f_{k}^{(j)}\left(y,z\right)\right\Vert _{2}^{2}+\left\Vert \nabla_{y}f_{k}^{(j)}\left(y,z\right)\right\Vert _{2}^{2}}\\
 & \leq\sqrt{\left(m+1\right)\left(\frac{9m^{-1}}{\delta^{2}\left(2\pi\right)^{d}}R^{2}\right)^{2}+\frac{2}{\delta^{2}\left(2\pi\right)^{d}}+\left(\frac{3\sqrt{m}R}{\delta^{2}\left(2\pi\right)^{d}}\right)^{2}}\triangleq C_{\mathsf{lip},f},
\end{align*}
which suggests that $f^{(j)}(y,z)$ is $\sqrt{d}C_{\mathsf{lip},f}$-Lipschitz
continuous, and 
\begin{align*}
\left\Vert \nabla g^{(j)}\left(y,z\right)\right\Vert _{2} & =\sqrt{\left\Vert \nabla_{y}g^{(j)}\left(y,z\right)\right\Vert _{2}^{2}+\sum_{l=1}^{m}\left\Vert \nabla_{z_{l}}g^{(j)}\left(y,z\right)\right\Vert _{2}^{2}}\\
 & \leq\sqrt{\left(1+\frac{1}{\left(2\pi\right)^{d/2}\delta}+\frac{\left(1+\varepsilon\right)\sqrt{m}}{\left(2\pi\right)^{d}\delta^{2}}\right)^{2}+m\left(\frac{3\left(1+\varepsilon\right)R}{\left(2\pi\right)^{d/2}\delta}+\frac{3\left(1+\varepsilon\right)^{2}R\left\Vert \phi\right\Vert _{\infty}^{2}}{\left(2\pi\right)^{d}\delta^{2}}\right)^{2}}\triangleq C_{\mathsf{lip},g}.
\end{align*}
which suggests that $g^{(j)}(y,z)$ is $C_{\mathsf{lip},g}$-Lipschitz
continuous. This allows us to conclude that $h(y,z)$ is $C_{\mathsf{Lip}}$-continuous
in $\{(y,z):\Vert y-\omega_{0}\Vert_{2}\leq\varepsilon,\max_{j\in[m]}\Vert z_{j}\Vert_{2}\leq2R\}$,
where $C_{\mathsf{Lip}}=\sqrt{mC_{\mathsf{lip},g}^{2}+mdC_{\mathsf{lip},f}^{2}}$.
In addition, it is easy to check that for any $z\in\mathbb{R}^{md}$
satisfying $\max_{j\in[m]}\Vert z_{j}\Vert_{2}\leq2R$ and any $y\in\mathbb{R}^{m}$
satisfying $\Vert y-\omega_{0}\Vert_{2}\leq\varepsilon$, 
\[
\max_{j\in[m]}\left\Vert f^{(j)}\left(y,z\right)\right\Vert _{2}\leq\frac{3R}{\delta\left(2\pi\right)^{d/2}},\quad\max_{j\in[m]}\left|g^{(j)}\left(y,z\right)\right|\leq\left(1+\frac{1}{\delta\left(2\pi\right)^{d/2}}\right)\left(1+\varepsilon\right)
\]
and therefore
\[
\left\Vert h\left(y,z\right)\right\Vert _{2}\leq\sqrt{m\left(\frac{3R}{\delta\left(2\pi\right)^{d/2}}\right)^{2}+m\left[\left(1+\frac{1}{\delta\left(2\pi\right)^{d/2}}\right)\left(1+\varepsilon\right)\right]^{2}}\triangleq M.
\]

Recall that $\mu_{0}^{(1)},\ldots,\mu_{0}^{(m)}$ are i.i.d.~sampled
from $\mathsf{Uniform}(\{X_{i}\}_{1\leq i\leq N})$, therefore
\[
\left(\omega_{0},\mu_{0}\right)\in(y,z):\left\{ \left(y,z\right):\Vert y-\omega_{0}\Vert_{2}\leq\varepsilon,\max_{j\in[m]}\Vert z_{j}\Vert_{2}\leq2R\right\} 
\]
where $\mu_{0}=[\mu_{0}^{(j)}]_{1\leq j\leq m}$. We are ready to
apply Picard-Lindel\"of theorem to show that there exists $t_{0}>0$,
only depending on $C_{\mathsf{Lip}}$ and $M$, such that the ODE
has a unique solution on the time interval $[0,t_{0}]$. We can use
the same argument in the proof of Theorem \ref{thm:FR-ODE} in Appendix
\ref{sec:proof-FR-ODE} to show that $\omega_{t}\in\Delta^{m-1}$
for all $t\in[0,t_{0}]$, and can use the same argument in the proof
of Theorem \ref{thm:Wass-ODE} in Appendix \ref{sec:proof-Wass-ODE}
to show that $\max_{j\in[m]}\Vert\mu_{t}^{(j)}\Vert_{2}\leq R$ for
all $t\in[0,t_{0}]$. Then we can repeat the same analysis as above
(notice that the above proof only requires $\omega_{0}\in\Delta^{m-1}$
and $\max_{j\in[m]}\Vert\mu_{0}^{(j)}\Vert_{2}\leq R$, and $t_{0}$
only depends on universal constants $C_{\mathsf{Lip}}$ and $M$)
to show that the ODE has a unique solution on $[t_{0},2t_{0}]$, $[2t_{0},3t_{0}]$,
and so on. This shows the existence and uniqueness of the solution
to the ODE system  \eqref{eq:WFR-ODE}. 

Finally we check that $(\rho_{t})_{t\geq0}$ defined as $\rho_{t}\coloneqq\sum_{l=1}^{m}\omega_{t}^{(l)}\delta_{\mu_{t}^{(l)}}$
solves  \eqref{eq:Fisher-Rao-GF}. Note that $\rho_{t}$ is a probability
measure since we have shown that $\omega_{t}\in\Delta^{m-1}$ for
any $t\geq0$. For any test function $\varphi(x)\in C_{\mathrm{c}}^{\infty}$,
we have
\begin{align*}
\frac{\mathrm{d}}{\mathrm{d}t}\int_{\mathbb{R}^{d}}\varphi\left(x\right)\rho_{t}\left(\mathrm{d}x\right) & =\frac{\mathrm{d}}{\mathrm{d}t}\left[\sum_{j=1}^{m}\omega_{t}^{(j)}\varphi\left(\mu_{t}^{(j)}\right)\right]=\sum_{j=1}^{m}\left[\dot{\omega}_{t}^{(j)}\varphi\left(\mu_{t}^{(j)}\right)+\omega_{t}^{(j)}\left\langle \nabla\varphi\left(\mu_{t}^{(j)}\right),\dot{\mu}_{t}^{(j)}\right\rangle \right]\\
 & =-\sum_{j=1}^{m}\left[1+\delta\ell_{N}\left(\rho_{t}\right)\left(\mu_{t}^{(j)}\right)\right]\omega_{t}^{(j)}\varphi\left(\mu_{t}^{(j)}\right)+\sum_{j=1}^{m}\omega_{t}^{(j)}\left\langle \nabla\varphi\left(\mu_{t}^{(j)}\right),-\nabla\delta\ell_{N}\left(\rho_{t}\right)\left(\mu_{t}^{(j)}\right)\right\rangle \\
 & =-\int_{\mathbb{R}^{d}}\left[1+\delta\ell_{N}\left(\rho_{t}\right)\left(x\right)\right]\varphi\left(x\right)\rho_{t}\left(\mathrm{d}x\right)-\int_{\mathbb{R}^{d}}\left\langle \nabla\varphi\left(x\right),\nabla\delta\ell_{N}\left(\rho_{t}\right)\right\rangle \rho_{t}\left(\mathrm{d}x\right).
\end{align*}
This proves that 
\[
\partial_{t}\rho_{t}=\mathsf{div}\left(\rho_{t}\nabla\delta\ell\left(\rho_{t}\right)\right)-\left[\delta\ell\left(\rho_{t}\right)+1\right]\rho_{t}
\]
holds in the sense of distributions.

%% file: appendix_wass_theory.tex
\section{Properties of Wasserstein gradient flow \label{sec:Wass-theory}}

In this section, we present some preliminary results on Wasserstein
gradient flow for learning Gaussian mixtures. We also discuss the
implications of these results, as well as the technical difficulty
of obtaining more general results.

We first establish the connection between the Wasserstein gradient
flow and the classical gradient flow in the Euclidean space. Suppose
we fit the data $\{X_{i}\}_{1\leq i\leq N}$ using a $m$-component
Gaussian mixture model
\[
\frac{1}{m}\sum_{j=1}^{m}\mathcal{N}\left(\mu^{(j)},I_{d}\right),
\]
where $\{\mu^{(j)}\}_{1\leq j\leq m}$ is the location of the $m$
Gaussian components. The negative likelihood function is
\begin{equation}
\ell_{N,m}\left(\mu^{(1)},\ldots,\mu^{(m)}\right)\coloneqq-\frac{1}{N}\sum_{i=1}^{N}\log\left[\frac{1}{m}\sum_{j=1}^{m}\phi\left(X_{i}-\mu^{(j)}\right)\right].\label{eq:finite-MLE}
\end{equation}
The gradient flow for minimizing (\ref{eq:finite-MLE}), denoted by
$(\mu_{t})_{t\geq0}$ where $\mu_{t}=[\mu_{t}^{(j)}]_{1\leq j\leq m}$,
is given by the following ODE system
\begin{equation}
\dot{\mu}_{t}^{(j)}=-\nabla_{\mu^{(j)}}\ell_{N,m}\left(\mu_{t}^{(1)},\ldots,\mu_{t}^{(m)}\right)\label{eq:Euclidean-GF}
\end{equation}
with initialization $\mu_{0}^{(1)},\ldots,\mu_{0}^{(m)}\overset{i.i.d.}{\sim}\mathsf{Uniform}(\{X_{i}\}_{1\leq i\leq N})$.
The following theorem shows that the gradient flow (\ref{eq:Euclidean-GF})
captures the evolution of the location of particles in the Wasserstein
gradient flow (\ref{eq:Wass-GF}) initialized from a discrete distribution
$\frac{1}{m}\sum_{l=1}^{m}\delta_{\mu_{0}^{(l)}}$. The proof is deferred
to Appendix \ref{subsec:proof-thm-Wass-Euclidean-connection}.

\begin{theorem} \label{thm:Wass-Euclidean-connection} Consider the Euclidean gradient flow $(\mu_{t})_{t\geq0}$ in (\ref{eq:Euclidean-GF}).
Then the flow $(\rho_{t})_{t\geq0}$ defined as
\begin{equation}
\rho_{t}\coloneqq\frac{1}{m}\sum_{l=1}^{m}\delta_{\mu_{t}^{(l)}}\label{eq:distributional-Euclidean-GF}
\end{equation}
is the Wasserstein gradient flow, i.e.~(\ref{eq:distributional-Euclidean-GF})
is a distributional solution to the PDE (\ref{eq:Wass-GF}).

\end{theorem}

Similar connection can also be established for the gradient descent
algorithm for minimizing (\ref{eq:finite-MLE}) and the particle Wasserstein
gradient descent (cf.~Algorithm \ref{alg:Wasserstein-GD}), which
is omitted for brevity.

Then we focus on the infinite sample limit of Wasserstein gradent
flow and analyze its convergence property. The population level loss
function is 
\begin{equation}
\ell_{\infty}\left(\rho\right)=-\mathbb{E}_{X\sim\rho^{\star}*\phi}\left\{ \log\left[\rho*\phi\left(X\right)\right]\right\} =\mathsf{KL}\left(\rho^{\star}*\phi\,\Vert\,\rho*\phi\right)+\mathsf{const}.\label{eq:loss-infinite}
\end{equation}
In Appendix \ref{subsec:First-variation} we have computed that 
\begin{equation}
\delta\ell_{\infty}\left(\rho\right)=-\int_{\mathbb{R}^{d}}\frac{\rho^{\star}\ast\phi\left(y\right)}{\rho*\phi\left(y\right)}\phi\left(x-y\right)\mathrm{d}y.\label{eq:FR-l-infty}
\end{equation}
We know that the Wasserstein gradient flow $(\rho_{t})_{t\geq0}$
with respect to $\ell_{\infty}(\rho)$ is described by the following
PDE: 
\begin{equation}
\partial_{t}\rho_{t}=\mathsf{div}\left(\rho_{t}\nabla\delta\ell_{\infty}\left(\rho_{t}\right)\right)\label{eq:population-WGF}
\end{equation}
with $\rho_{0}=\rho^{\star}*\mathcal{N}(0,I_{d})$, which is the data
distribution when we have infinite samples. This Wasserstein gradient
flow has the following particle interpretation: suppose at time $t=0$
we initialize a particle $x_{0}\sim\rho_{0}$ in the vector field
$(v_{t})_{t\geq0}$ where $v_{t}=-\nabla\delta\ell_{\infty}(\rho_{t})$,
namely 
\[
\dot{x}_{t}=v_{t}\left(x_{t}\right),
\]
then $x_{t}\sim\rho_{t}$, namely the marginal distribution of $(x_{t})_{t\geq0}$
evolves according to the Wasserstein gradient flow.

The following theorem shows that, when the true mixing distribution
$\rho^{\star}$ is a singleton (we assume without loss of generality
that $\rho^{\star}=\delta_{0}$), Wasserstein gradient flow converges
to $\rho^{\star}$. The proof can be found in Appendix \ref{subsec:proof-thm-Wass-convergence}.

\begin{theorem}\label{thm:Wass-convergence}Consider the Wasserstein
gradient flow in \eqref{eq:population-WGF} with $\rho^{\star}=\delta_{0}$.
For any $\varepsilon<1$, we have $\int_{\mathbb{R}^{d}}\Vert x\Vert_{2}^{2}\rho_{t}(\mathrm{d}x)=O(\varepsilon)$
as long as 
\[
t\geq\exp\left(2d\right)\varepsilon^{-1-\max\left\{ 8,\sqrt{8d}\right\} }.
\]
\end{theorem}

Although Theorem \ref{thm:Wass-convergence} only focuses on the case
when $\rho^{\star}$ is a singleton, the convergence result already
provides some intuition about the behavior of Wasserstein gradient
flow in more general setting. Consider a well-separated Gaussian mixture
model with $K$ components. Assume that the mixing distribution is
$\rho^{\star}=\sum_{j=1}^{K}\omega_{j}^{\star}\delta_{\mu_{j}^{\star}}$,
and the location of each Gaussian components, $\{\mu_{j}^{\star}\}_{1\leq j\leq K}$,
are well-separated. Since the push-forward mapping $v_{t}=-\nabla\delta\ell_{\infty}(\rho_{t})$
is localized (see \ref{eq:FR-l-infty}), there exists some $T>0$
such that the Wasserstein gradient flow \eqref{eq:population-WGF}
initialized from $\rho^{\star}*\mathcal{N}(0,I_{d})$ can be approximated,
up to time $T$, by 
\[
\rho_{t}\approx\sum_{j=1}^{K}\omega_{j}^{\star}\rho_{t}^{(j)}\qquad\forall\,t\in\left[0,T\right],
\]
where for each $j\in[K]$, $\rho_{t}^{(j)}$ is the Wasserstein gradient
flow $\partial_{t}\rho_{t}^{(j)}=\mathsf{div}(\rho_{t}^{(j)}\nabla\delta\ell_{\infty}(\rho_{t}^{(j)}))$
with initialization $\rho_{0}^{(j)}=\mathcal{N}(\mu_{j}^{\star},I_{d})$.
This suggests that Wasserstein gradient flow approximately converges
to $\rho^{\star}$ since, by Theorem \ref{thm:Wass-convergence},
each $\rho_{t}^{(j)}$ converges to $\delta_{\mu_{j}^{\star}}$. However
this observation also suggests that Wasserstein gradient flow is not
robust to weight mismatch. Consider initializing the Wasserstein gradient
flow \eqref{eq:population-WGF} with $\rho_{0}=\widetilde{\rho}*\mathcal{N}(0,1)$,
where $\widetilde{\rho}=\sum_{j=1}^{K}\widetilde{\omega}_{j}\delta_{\mu_{j}^{\star}}$
is a mixing distribution with correct support $\{\mu_{j}^{\star}\}_{1\leq j\leq K}$
but wrong weights $\{\widetilde{\omega}_{j}\}_{1\leq j\leq K}\neq\{\omega_{j}^{\star}\}_{1\leq j\leq K}$.
Then we also have 
\[
\rho_{t}\approx\sum_{j=1}^{K}\widetilde{\omega}_{j}\rho_{t}^{(j)}\qquad\forall\,t\in\left[0,T\right],
\]
which shows that $\rho_{t}$ approximately converges to $\widetilde{\rho}$
instead of $\rho^{\star}$ when $0\leq t\leq T$. Note that the time
length $T$ that such approximations are valid can be arbitrarily
large as long as the separation $\min_{i\neq j}\Vert\mu_{i}^{\star}-\mu_{j}^{\star}\Vert_{2}\to\infty$.
The above discussion suggests that using the correct initial weights
are important for Wasserstein gradient flow to converge to the true
mixing distribution.

We also would like to compare the convergence rate in Theorem \ref{thm:Wass-convergence}
to a benchmark provided by the Bures-Wasserstein gradient flow. The
Bures-Wasserstein gradient flow is defined on the space of non-degenerate
Gaussian distributions on $\mathbb{R}^{d}$, denoted by $\mathsf{BW}(\mathbb{R}^{d})=\mathbb{R}^{d}\times\mathbb{S}_{++}^{d}$
(where we identify a non-degenerate Gaussian distribution $\nu=\mathcal{N}(\mu,\Sigma)$
with $(\mu,\Sigma)\in\mathbb{R}^{d}\times\mathbb{S}_{++}^{d}$) equipped
with the Wasserstein distance \eqref{eq:Wass-metric}, which has the
following closed form expression 
\[
d_{\mathsf{W}}^{2}\left(\nu_{1},\nu_{2}\right)=\left\Vert \mu_{1}-\mu_{2}\right\Vert _{2}^{2}+\mathsf{tr}\left[\Sigma_{1}+\Sigma_{2}-2\left(\Sigma_{1}^{1/2}\Sigma_{2}\Sigma_{1}^{1/2}\right)^{1/2}\right]
\]
when $\nu_{1}=\mathcal{N}(\mu_{1},\Sigma_{1})$ and $\nu_{2}=\mathcal{N}(\mu_{2},\Sigma_{2})$
are both non-degenerate Gaussians. The Bures-Wasserstein gradient
flow $(\nu_{t})_{t\geq0}$ can be viewed as the Wasserstein gradient
flow $(\rho_{t})_{t\geq0}$ constrained to lie on $\mathsf{BW}(\mathbb{R}^{d})$.
We refer interested readers to \citet{lambert2022variational,altschuler2021averaging}
for more detailed discussion. We can see from the proof of Theorem
\ref{thm:Wass-convergence} that the push-forward mapping $v_{t}(x)$
of Wasserstein gradient flow decays exponentially fast as $\Vert x\Vert_{2}\to\infty$,
this will make the Wasserstein gradient flow $(\rho_{t})_{t\geq0}$
becomes more and more heavy-tailed . However the push-forward mapping
of Bures-Wasserstein gradient flow is always linear, and the Bures-Wasserstein
gradient flow $(\nu_{t})_{t\geq0}$ is always Gaussian. For example,
in Appendix \ref{subsec:BW-1-Gaussian} we can compute the push forward
mapping explicitly for the two gradient flows at $t=0$: 
\begin{align*}
v_{0}\left(x\right) & =-\frac{1}{3}\left(\frac{4}{3}\right)^{d/2}\exp\left(-\frac{\left\Vert x\right\Vert _{2}^{2}}{6}\right)x\qquad\qquad\qquad\;\text{(Wasserstein)},\\
v_{0}\left(x\right) & =\frac{x}{4}\qquad\qquad\qquad\qquad\qquad\qquad\qquad\quad\text{(Bures-Wasserstein).}
\end{align*}
Therefore it is natural to expect that Bures-Wasserstein gradient
flow $(\nu_{t})_{t\geq0}$ initialized from $\nu_{0}=\mathcal{N}(0,I_{d})$
converges faster than Wasserstein gradient flow $(\rho_{t})_{t\geq0}$
initialized from $\rho_{0}=\delta_{0}$. In Appendix \ref{subsec:BW-1-Gaussian}
we show that the Bures-Wasserstein gradient flow $(\nu_{t}=\mathcal{N}(\mu_{t},\Sigma_{t}))_{t\geq0}$
is characterized by the following ODE: 
\begin{align*}
\mu_{t} & =0\\
\dot{\Sigma}_{t} & =-2\left(\Sigma_{t}+I_{d}\right)^{-1}\Sigma_{t}^{2}\left(\Sigma_{t}+I_{d}\right)^{-1}.
\end{align*}
We also show that $\Sigma_{t}$ is sandwiched between 
\[
\frac{1}{1+2t}I\preceq\Sigma_{t}\preceq\frac{2}{2+t}I,
\]
and as a result $\int_{\mathbb{R}^{d}}\Vert x\Vert_{2}^{2}\nu_{t}(\mathrm{d}x)=O(d/t)$.
Since Bures-Wasserstein gradient flow is not converging exponentially
fast (we can see that the convergence rate is polynomial in $t$),
we conjecture that Wasserstein gradient flow does not enjoy exponential
convergence as well.

Lastly, we numerically show in Figure \ref{fig:geodesic} that the
loss function $\ell_{\infty}(\rho)$ (cf.~ \eqref{eq:loss-infinite})
is not geodesically convex \citep{ambrosio2008gradient} even when
$\rho^{\star}=\delta_{0}$. We can also check that Polyak-\L ojasiewicz
(PL) inequality 
\[
\forall\,\rho:\quad\left\Vert \nabla_{\mathsf{W}}\ell_{\infty}\left(\rho\right)\right\Vert _{\rho}^{2}\geq C_{\mathsf{PL}}\left[\ell_{\infty}\left(\rho\right)-\ell_{\infty}\left(\rho^{\star}\right)\right]\quad\text{for some }C_{\mathsf{PL}}>0
\]
does not hold in general: consider $\rho^{\star}=\frac{1}{2}\delta_{-1}+\frac{1}{2}\delta_{1}$
and $\rho=\delta_{0}$, then it is straightforward to check that $\nabla_{\mathsf{W}}\ell_{\infty}(\rho)=0$
but $\ell_{\infty}(\rho)>\ell_{\infty}(\rho^{\star})$. Therefore
we cannot use standard proof technique (e.g.~\citet{ambrosio2008gradient}
when the loss function is geodesically convex, or \citet{chewi2020gradient}
when there is a PL inequality) to show exponential convergence for
the Wasserstein gradient flow \eqref{eq:population-WGF}.

\begin{figure}[t]
\centering

\begin{tabular}{cc}
\multicolumn{2}{c}{\qquad{}(a) constant speed geodesic $(\rho_{t})_{0\leq t\leq1}$
joining $\rho_{0}=\mathcal{N}(0,3)$ to $\rho_{1}=\mathcal{N}(0,1)$}\tabularnewline
\includegraphics[scale=0.4]{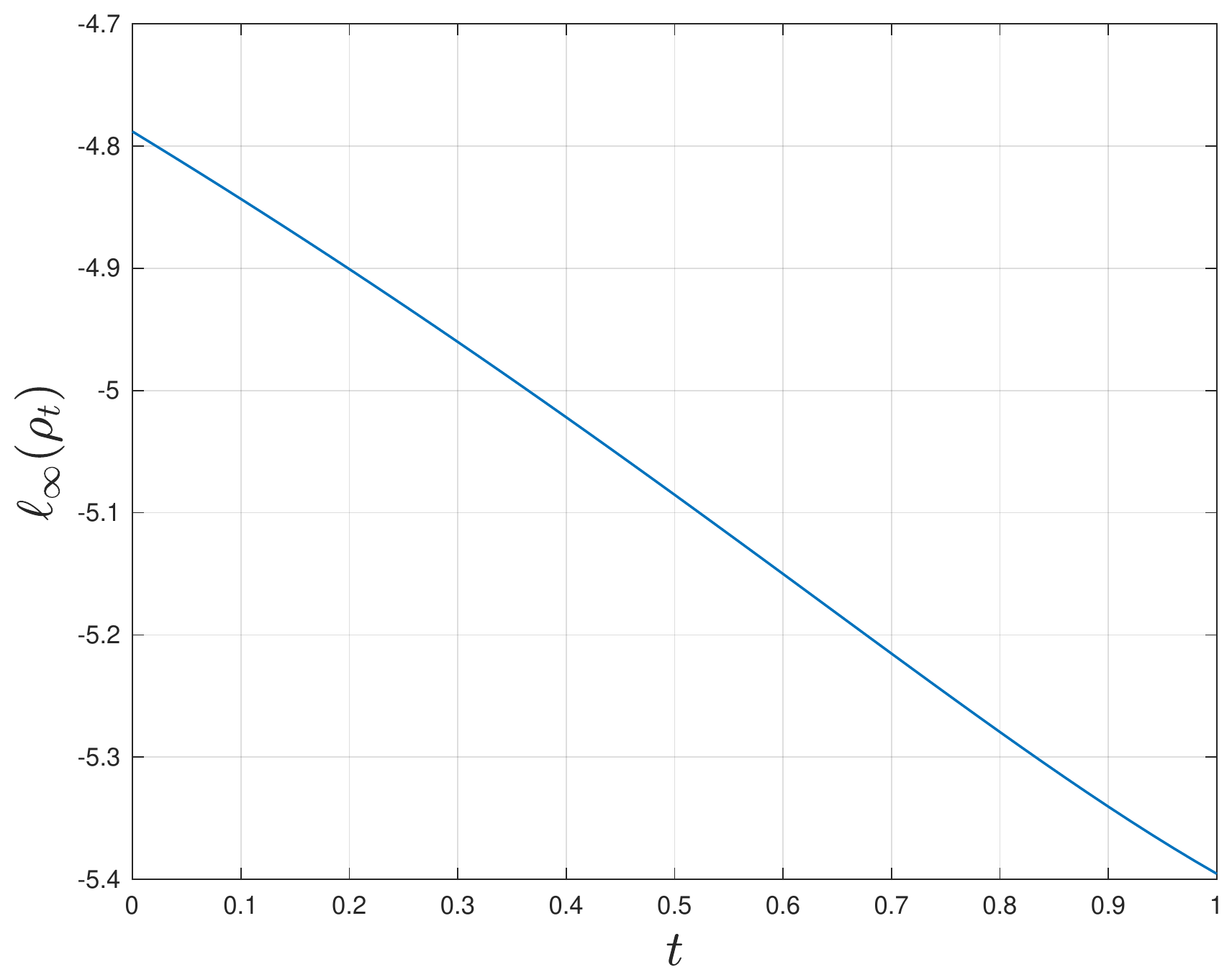}  & \includegraphics[scale=0.4]{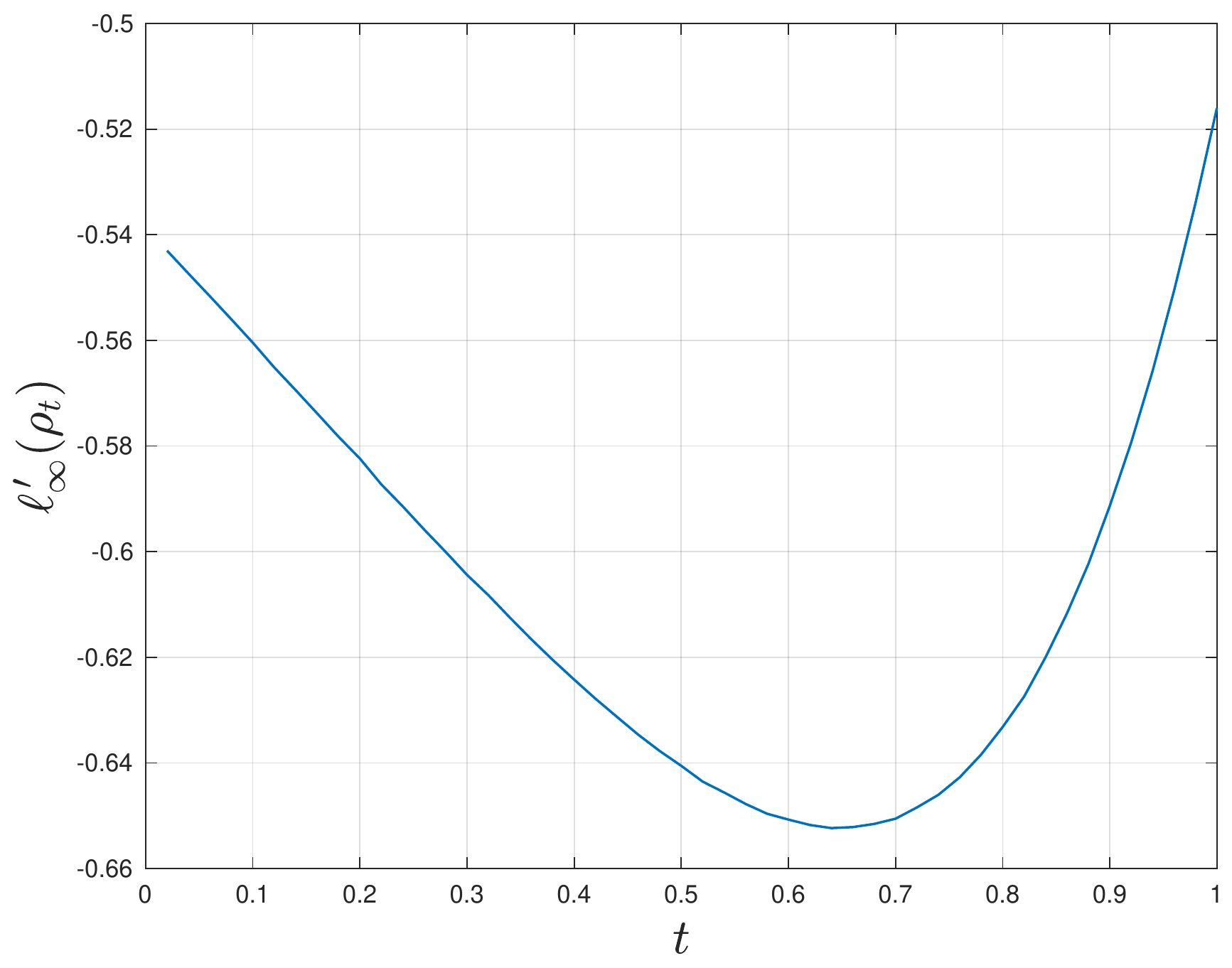}\tabularnewline
\multicolumn{2}{c}{\qquad{}(b) constant speed geodesic $(\rho_{t})_{0\leq t\leq1}$
joining $\rho_{0}=\mathcal{N}(0,1)$ to $\rho_{1}=\delta_{0}$}\tabularnewline
\includegraphics[scale=0.4]{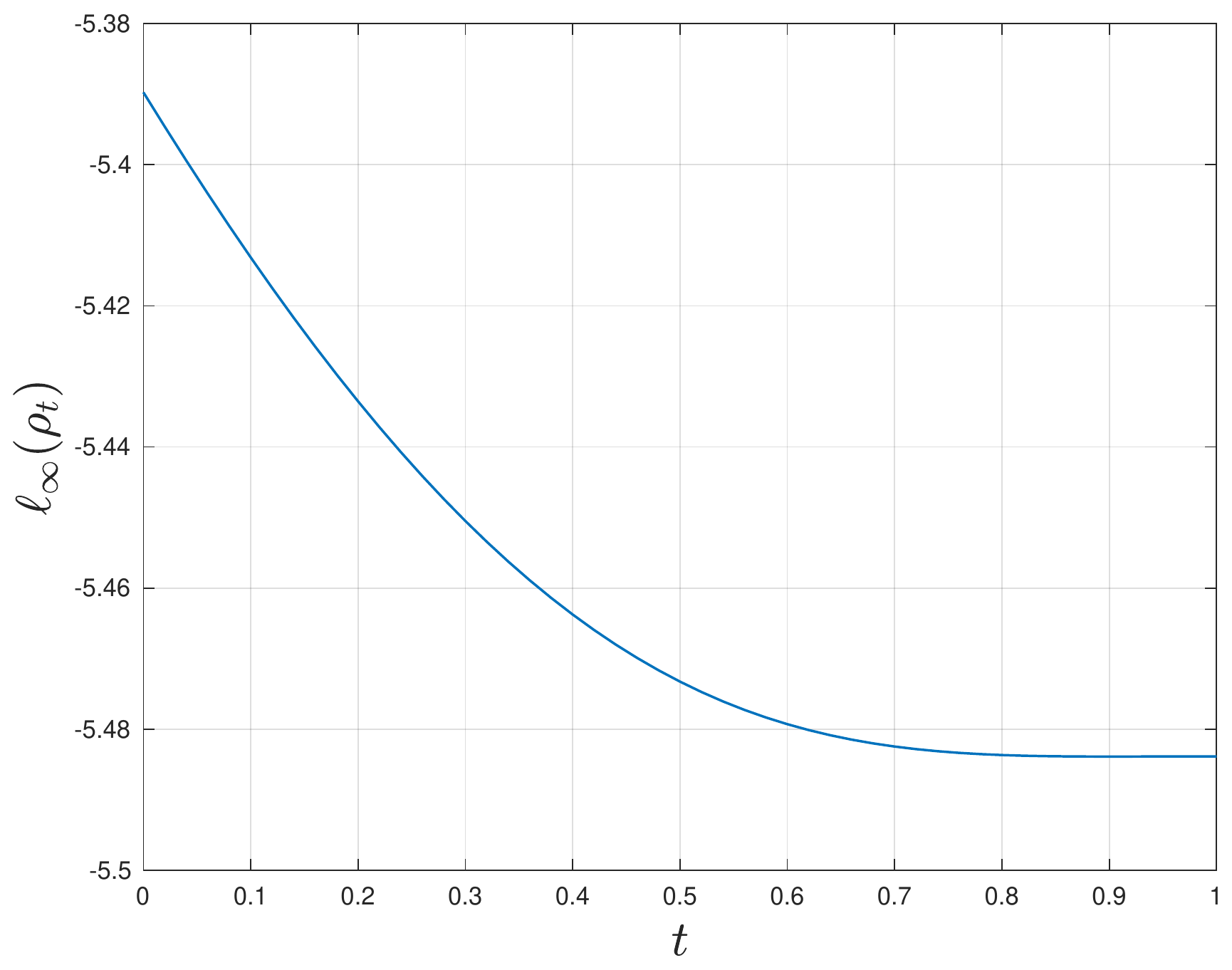}  & \includegraphics[scale=0.4]{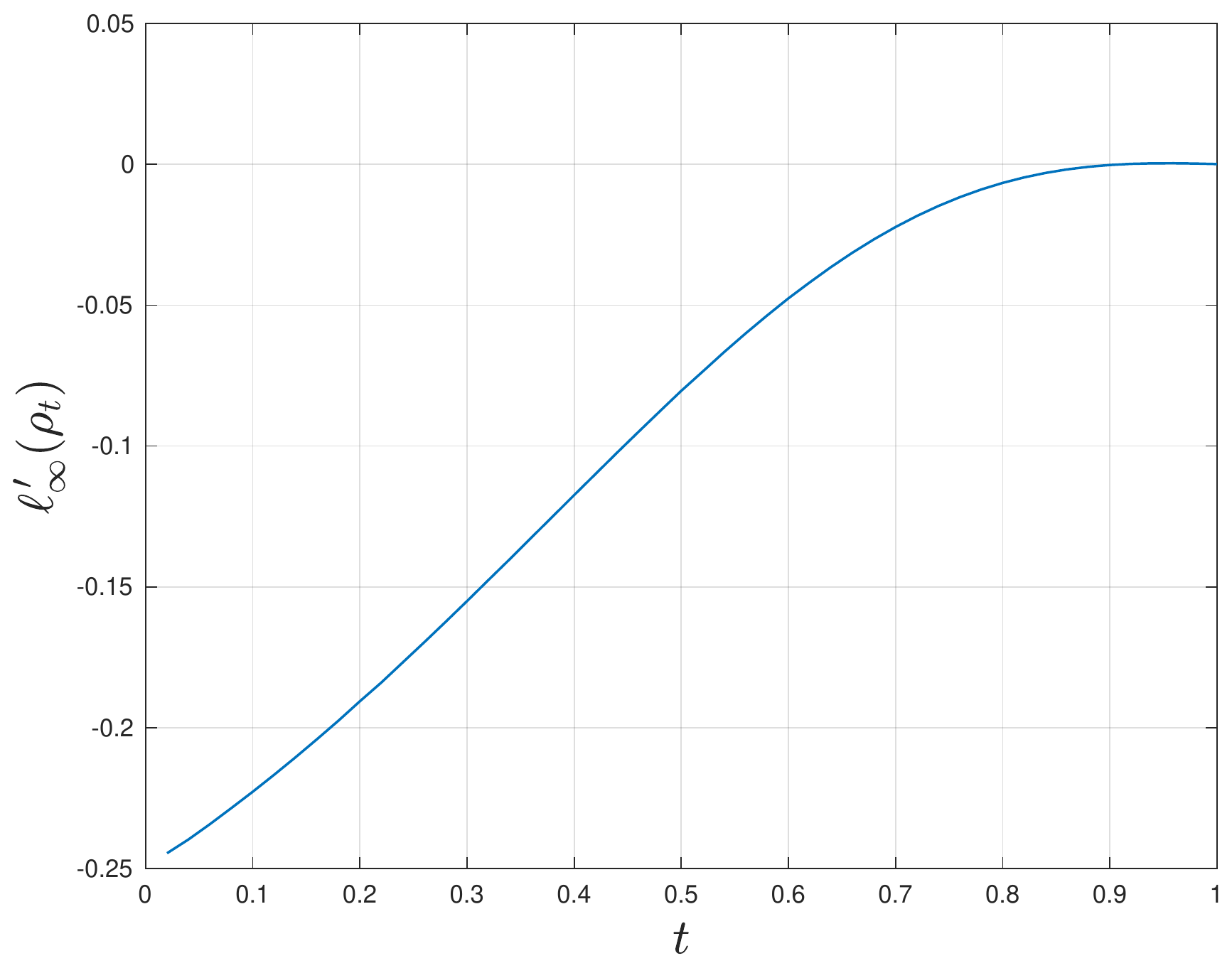}\tabularnewline
\end{tabular}

\caption{The loss function $\ell_{\infty}(\rho_{t})$ or its derivative $\ell_{\infty}'(\rho_{t})$
vs.~$t$. In Figures (a), $(\rho_{t})_{0\protect\leq t\protect\leq1}$
is the constant speed geodesic joining $\rho_{0}=\mathcal{N}(0,3)$
to $\rho_{1}=\mathcal{N}(0,1)$. In Figures (b), $(\rho_{t})_{0\protect\leq t\protect\leq1}$
is the constant speed geodesic joining $\rho_{0}=\mathcal{N}(0,1)$
to $\rho_{1}=\delta_{0}$. This shows that $\ell_{\infty}(\rho)$
is not globally geodesically convex, but might be locally geodesically
conex around $\rho^{\star}=\delta_{0}$.\label{fig:geodesic}}
\end{figure}

\subsection{Proof of Theorem \ref{thm:Wass-Euclidean-connection} \label{subsec:proof-thm-Wass-Euclidean-connection}}

It is straightforward to compute the gradient of $\ell_{N,m}$. For
any $j\in[m]$, we have
\begin{align*}
\nabla_{\mu^{(j)}}\ell_{N,m}\left(\mu^{(1)},\ldots,\mu^{(m)}\right) & =-\frac{1}{N}\sum_{i=1}^{N}\frac{1}{\sum_{l=1}^{m}\phi\left(X_{i}-\mu^{(l)}\right)}\nabla_{\mu^{(j)}}\phi\left(X_{i}-\mu^{(j)}\right)\\
 & =-\frac{1}{N}\sum_{i=1}^{N}\frac{\phi\left(X_{i}-\mu^{(l)}\right)}{\sum_{l=1}^{m}\phi\left(X_{i}-\mu^{(l)}\right)}\left(X_{i}-\mu^{(l)}\right).
\end{align*}
Therefore the Euclidean gradient flow (\ref{eq:Euclidean-GF}) is
given by
\[
\dot{\mu}_{t}^{(j)}=-\nabla_{\mu^{(j)}}\ell_{N,m}\left(\mu_{t}^{(1)},\ldots,\mu_{t}^{(m)}\right)=\frac{1}{N}\sum_{i=1}^{N}\frac{\phi\left(X_{i}-\mu_{t}^{(l)}\right)}{\sum_{l=1}^{m}\phi\left(X_{i}-\mu_{t}^{(l)}\right)}\left(X_{i}-\mu_{t}^{(l)}\right).
\]
Then we can invoke Theorem \ref{thm:Wass-ODE} to finish the proof.

\subsection{Proof of Theorem \ref{thm:Wass-convergence} \label{subsec:proof-thm-Wass-convergence}}

\paragraph{Step 1: characterizing the push-forward mapping.}

First of all, it is straightforward to check that $(\rho_{t})_{t\geq0}$
is spherically symmetric for all $t\geq0$, namely $\rho_{t}(\mathrm{d}x)$
only depends on $\Vert x\Vert_{2}$. The push-forward mapping $v_{t}(x):\mathbb{R}^{d}\to\mathbb{R}^{d}$
at time $t$ is 
\begin{align}
v_{t}\left(x\right) & =-\nabla\delta\ell_{\infty}\left(\rho_{t}\right)\left(x\right)=\nabla_{x}\int_{\mathbb{R}^{d}}\frac{\rho^{\star}\ast\phi\left(y\right)}{\rho_{t}*\phi\left(y\right)}\phi\left(x-y\right)\mathrm{d}y\nonumber \\
 & =-\int_{\mathbb{R}^{d}}\frac{\rho^{\star}\ast\phi\left(y\right)}{\rho_{t}*\phi\left(y\right)}\left(x-y\right)\phi\left(x-y\right)\mathrm{d}y\nonumber \\
 & =\int_{\mathbb{R}^{d}}\nabla_{y}\frac{\rho^{\star}\ast\phi\left(y\right)}{\rho_{t}*\phi\left(y\right)}\phi\left(y-x\right)\mathrm{d}y\nonumber \\
 & =\int_{\mathbb{R}^{d}}h_{t}\left(y\right)\phi\left(y-x\right)\mathrm{d}y,\label{eq:Wass-convergence-inter-1}
\end{align}
where the penultimate line follows from Stein's lemma or Gaussian
integration by parts, and $h_{t}(y)$ in the last line is defined
as 
\begin{align}
h_{t}\left(y\right) & \coloneqq\nabla\frac{\phi\left(y\right)}{\rho_{t}*\phi\left(y\right)}=\nabla_{y}\frac{1}{\int\exp\left(-\frac{1}{2}\left\Vert z\right\Vert _{2}^{2}+y^{\top}z\right)\rho_{t}\left(\mathrm{d}z\right)}\nonumber \\
 & =-\frac{\int\exp\left(-\frac{1}{2}\left\Vert z\right\Vert _{2}^{2}+y^{\top}z\right)z\rho_{t}\left(\mathrm{d}z\right)}{\left[\int\exp\left(-\frac{1}{2}\left\Vert z\right\Vert _{2}^{2}+y^{\top}z\right)\rho_{t}\left(\mathrm{d}z\right)\right]^{2}}=-\phi\left(y\right)\frac{\int\phi\left(y-z\right)z\rho_{t}\left(\mathrm{d}z\right)}{\left[\int\phi\left(y-z\right)\rho_{t}\left(\mathrm{d}z\right)\right]^{2}}\nonumber \\
 & =-\phi\left(y\right)\frac{\int\phi\left(y-z\right)z\rho_{t}\left(\mathrm{d}z\right)}{\left[\rho_{t}*\phi\left(y\right)\right]^{2}}=-\frac{\phi\left(y\right)}{\rho_{t}*\phi\left(y\right)}\cdot\frac{\int\phi\left(y-z\right)z\rho_{t}\left(\mathrm{d}z\right)}{\rho_{t}*\phi\left(y\right)}.\label{eq:Wass-convergence-inter-2}
\end{align}
For any $y\in\mathbb{R}^{d}$, we can compute 
\begin{align}
\int\phi\left(y-z\right)z\rho_{t}\left(\mathrm{d}z\right) & =\int_{y^{\top}z>0}\phi\left(y-z\right)z\rho_{t}\left(\mathrm{d}z\right)+\int_{y^{\top}z<0}\phi\left(y-z\right)z\rho_{t}\left(\mathrm{d}z\right)+\int_{y^{\top}z=0}\phi\left(y-z\right)z\rho_{t}\left(\mathrm{d}z\right)\nonumber \\
 & \overset{\text{(i)}}{=}\int_{y^{\top}z>0}\left[\phi\left(y-z\right)-\phi\left(y+z\right)\right]z\rho_{t}\left(\mathrm{d}z\right)\nonumber \\
 & =\phi\left(y\right)\int_{y^{\top}z>0}\left[\exp\left(y^{\top}z\right)-\exp\left(-y^{\top}z\right)\right]z\exp\left(-\frac{1}{2}\left\Vert z\right\Vert _{2}^{2}\right)\rho_{t}\left(\mathrm{d}z\right)\nonumber \\
 & \overset{\text{(ii)}}{=}\phi\left(y\right)\int_{y^{\top}z>0}\left[\exp\left(y^{\top}z\right)-\exp\left(-y^{\top}z\right)\right]\frac{y^{\top}z}{\left\Vert y\right\Vert _{2}^{2}}y\exp\left(-\frac{1}{2}\left\Vert z\right\Vert _{2}^{2}\right)\rho_{t}\left(\mathrm{d}z\right)\nonumber \\
 & =\underbrace{\int_{y^{\top}z>0}\left[\exp\left(y^{\top}z\right)-\exp\left(-y^{\top}z\right)\right]\frac{y^{\top}z}{\left\Vert y\right\Vert _{2}^{2}}\exp\left(-\frac{1}{2}\left\Vert z\right\Vert _{2}^{2}\right)\rho_{t}\left(\mathrm{d}z\right)}_{\eqqcolon a_{t}}\,\cdot\,\,\phi\left(y\right)y.\label{eq:Wass-convergence-inter-3}
\end{align}
Here (i) and (ii) both follow from the spherical symmetry of $\rho_{t}$,
and it is straightforward to check that the integral in the last line
does not depend on $y$ due to the spherical symmetry of $\rho_{t}$,
therefore $a_{t}$ is a universal constant that is independent of
$y$. Note that when $y^{\top}z>0$, we have $\exp(y^{\top}z)-\exp(-y^{\top}z)\geq2y^{\top}z$,
therefore 
\begin{align*}
a_{t} & \geq2\int_{y^{\top}z>0}\frac{\left(y^{\top}z\right)^{2}}{\left\Vert y\right\Vert _{2}^{2}}\exp\left(-\frac{1}{2}\left\Vert z\right\Vert _{2}^{2}\right)\rho_{t}\left(\mathrm{d}z\right)=\int_{\mathbb{R}^{d}}\frac{\left(y^{\top}z\right)^{2}}{\left\Vert y\right\Vert _{2}^{2}}\exp\left(-\frac{1}{2}\left\Vert z\right\Vert _{2}^{2}\right)\rho_{t}\left(\mathrm{d}z\right).
\end{align*}
Since $a_{t}$ does not depend on $y$, we take $y=e_{i}$ for $i\in[d]$
to achieve 
\[
a_{t}\geq\int_{\mathbb{R}^{d}}z_{i}^{2}\exp\left(-\frac{1}{2}\left\Vert z\right\Vert _{2}^{2}\right)\rho_{t}\left(\mathrm{d}z\right),\qquad\forall\,i\in[d].
\]
By taking average over $d$, we have 
\begin{equation}
a_{t}\geq\frac{1}{d}\int_{\mathbb{R}^{d}}\left\Vert z\right\Vert _{2}^{2}\exp\left(-\frac{1}{2}\left\Vert z\right\Vert _{2}^{2}\right)\rho_{t}\left(\mathrm{d}z\right)=\frac{m_{t}}{d},\label{eq:Wass-convergence-inter-4}
\end{equation}
where we define 
\[
m_{t}\coloneqq\int\left\Vert z\right\Vert _{2}^{2}\exp\left(-\frac{1}{2}\left\Vert z\right\Vert _{2}^{2}\right)\rho_{t}\left(\mathrm{d}z\right).
\]
Taking \eqref{eq:Wass-convergence-inter-2}, \eqref{eq:Wass-convergence-inter-3}
and \eqref{eq:Wass-convergence-inter-4} collectively gives 
\[
h_{t}\left(y\right)=-a_{t}y\left[\frac{\phi\left(y\right)}{\rho_{t}*\phi\left(y\right)}\right]^{2},\qquad\text{where}\qquad a_{t}\geq\frac{m_{t}}{d}.
\]
Then we use \eqref{eq:Wass-convergence-inter-1} to characterize the
push-forward mapping: 
\begin{align*}
v_{t}\left(x\right) & =\int_{y^{\top}x>0}h\left(y\right)\phi\left(y-x\right)\mathrm{d}y+\int_{y^{\top}x<0}h\left(y\right)\phi\left(y-x\right)\mathrm{d}y+\int_{y^{\top}x=0}h\left(y\right)\phi\left(y-x\right)\mathrm{d}y\\
 & \overset{\text{(i)}}{=}\int_{y^{\top}x>0}h\left(y\right)\left[\phi\left(y-x\right)-\phi\left(-y-x\right)\right]\mathrm{d}y\\
 & =-\phi\left(x\right)a_{t}\int_{y^{\top}x>0}y\left[\frac{\phi\left(y\right)}{\rho_{t}*\phi\left(y\right)}\right]^{2}\left[\exp\left(y^{\top}x\right)-\exp\left(-y^{\top}x\right)\right]\exp\left(-\frac{1}{2}\left\Vert y\right\Vert _{2}^{2}\right)\mathrm{d}y\\
 & \overset{\text{(ii)}}{=}-\phi\left(x\right)a_{t}\int_{y^{\top}x>0}\frac{x^{\top}y}{\left\Vert x\right\Vert _{2}^{2}}x\left[\frac{\phi\left(y\right)}{\rho_{t}*\phi\left(y\right)}\right]^{2}\left[\exp\left(y^{\top}x\right)-\exp\left(-y^{\top}x\right)\right]\exp\left(-\frac{1}{2}\left\Vert y\right\Vert _{2}^{2}\right)\mathrm{d}y\\
 & =-a_{t}\underbrace{\int_{y^{\top}x>0}\frac{x^{\top}y}{\left\Vert x\right\Vert _{2}^{2}}\left[\frac{\phi\left(y\right)}{\rho_{t}*\phi\left(y\right)}\right]^{2}\left[\exp\left(y^{\top}x\right)-\exp\left(-y^{\top}x\right)\right]\exp\left(-\frac{1}{2}\left\Vert y\right\Vert _{2}^{2}\right)\mathrm{d}y}_{\eqqcolon b_{t}}\,\cdot\,\,\phi\left(x\right)x.
\end{align*}
Similar to \eqref{eq:Wass-convergence-inter-3}, here (i) and (ii)
both follow from the spherical symmetry of $\rho_{t}$, and the integral
in the last line does not depend on $x$ due to the spherical symmetry
of $\rho_{t}$, as a result $b_{t}$ is a universal constant that
is independent of $x$. Note that when $y^{\top}x>0$, we have $\exp(y^{\top}x)-\exp(-y^{\top}x)\geq2y^{\top}x$,
therefore 
\begin{align*}
b_{t} & \geq2\int_{y^{\top}x>0}\frac{x^{\top}y}{\left\Vert x\right\Vert _{2}^{2}}\left[\frac{\phi\left(y\right)}{\rho_{t}*\phi\left(y\right)}\right]^{2}y^{\top}x\exp\left(-\frac{1}{2}\left\Vert y\right\Vert _{2}^{2}\right)\mathrm{d}y\\
 & =\int_{\mathbb{R}^{d}}\frac{\left(x^{\top}y\right)^{2}}{\left\Vert x\right\Vert _{2}^{2}}\left[\frac{\phi\left(y\right)}{\rho_{t}*\phi\left(y\right)}\right]^{2}\exp\left(-\frac{1}{2}\left\Vert y\right\Vert _{2}^{2}\right)\mathrm{d}y.
\end{align*}
Note that $\Vert\rho_{t}*\phi\Vert_{\infty}\leq\Vert\phi\Vert_{\infty}\leq(2\pi)^{-d/2}$,
and as a result 
\begin{align*}
b_{t} & \geq\int\frac{\left(x^{\top}y\right)^{2}}{\left\Vert x\right\Vert _{2}^{2}}\exp\left(-\frac{3}{2}\left\Vert y\right\Vert _{2}^{2}\right)\mathrm{d}y=\frac{1}{3}\left(\frac{2\pi}{3}\right)^{d/2}
\end{align*}
Therefore we have 
\begin{equation}
v_{t}\left(x\right)=-a_{t}b_{t}\phi\left(x\right)x=-c_{t}\phi\left(x\right)x\label{eq:Wass-convergence-inter-5}
\end{equation}
where 
\begin{equation}
c_{t}\coloneqq a_{t}b_{t}\geq\frac{1}{3d}\left(\frac{2\pi}{3}\right)^{d/2}m_{t}.\label{eq:Wass-convergence-inter-6}
\end{equation}

\paragraph{Step 2: showing the convergence of Wasserstein gradient flow.}

Recall the particle interpretation of Wasserstein gradient flow as
follows: let $x_{0}\sim\rho_{0}=\mathcal{N}(0,I_{d})$ and $\dot{x}_{t}=v_{t}(x_{t})$,
then for any $t\geq0$ we have $x_{t}\sim\rho_{t}$. This allows us
to compute 
\begin{align}
\partial_{t}\mathbb{E}\left[\left\Vert x_{t}\right\Vert _{2}^{2}\right] & =2\mathbb{E}\left[\left\langle x_{t},\dot{x}_{t}\right\rangle \right]=2\mathbb{E}\left[\left\langle x_{t},v_{t}\left(x_{t}\right)\right\rangle \right]\overset{\text{(i)}}{=}-2c_{t}\mathbb{E}\left[\left\Vert x_{t}\right\Vert _{2}^{2}\phi\left(x_{t}\right)\right]\nonumber \\
 & \overset{\text{(ii)}}{\leq}-\frac{2}{3d}\left(\frac{2\pi}{3}\right)^{d/2}m_{t}\mathbb{E}\left[\left\Vert x_{t}\right\Vert _{2}^{2}\phi\left(x_{t}\right)\right]\nonumber \\
 & \overset{\text{(iii)}}{=}-\frac{2}{3d}\left(\frac{4\pi^{2}}{3}\right)^{d/2}\mathbb{E}^{2}\left[\left\Vert x_{t}\right\Vert _{2}^{2}\phi\left(x_{t}\right)\right],\label{eq:Wass-convergence-inter-7}
\end{align}
where (i) follows from \eqref{eq:Wass-convergence-inter-5}, (ii)
utilizes \eqref{eq:Wass-convergence-inter-6}, and (iii) holds since
\[
m_{t}=\int\left\Vert z\right\Vert _{2}^{2}\exp\left(-\frac{1}{2}\left\Vert z\right\Vert _{2}^{2}\right)\rho_{t}\left(\mathrm{d}z\right)=\left(2\pi\right)^{d/2}\mathbb{E}\left[\left\Vert x_{t}\right\Vert _{2}^{2}\phi\left(x_{t}\right)\right].
\]
For any $\tau>0$, by Cauchy-Schwarz inequality we have 
\[
\mathbb{E}\left[\left\Vert x_{0}\right\Vert _{2}^{2}\ind\left\{ \left\Vert x_{0}\right\Vert _{2}^{2}>d+\tau\right\} \right]\overset{\text{(i)}}{\leq}\left[\mathbb{E}\left\Vert x_{0}\right\Vert _{2}^{4}\right]^{1/2}\left[\mathbb{P}\left(\left\Vert x_{0}\right\Vert _{2}^{2}>d+\tau\right)\right]^{1/2}.
\]
Note that $\Vert x_{0}\Vert_{2}^{2}\sim\chi^{2}(d)$, therefore $\mathbb{E}\Vert x_{0}\Vert_{2}^{4}=\mathsf{var}(\Vert x_{0}\Vert_{2}^{2})+(\mathbb{E}\Vert x_{0}\Vert_{2}^{2})^{2}=2d+d^{2}$.
In addition, by the tail probability bound for $\chi^{2}$ random
variables (e.g.~\citet[equation (2.18)]{wainwright2019high}), we
have 
\[
\mathbb{P}\left(\left\Vert x_{0}\right\Vert _{2}^{2}>d+\tau\right)\leq\exp\left(-\min\left\{ \frac{\tau^{2}}{8d},\frac{\tau}{8}\right\} \right).
\]
Therefore we have 
\begin{align*}
\mathbb{E}\left[\left\Vert x_{0}\right\Vert _{2}^{2}\ind\left\{ \left\Vert x_{0}\right\Vert _{2}^{2}>d+\tau\right\} \right] & \leq\sqrt{2d+d^{2}}\exp\left(-\min\left\{ \frac{\tau^{2}}{8d},\frac{\tau}{8}\right\} \right)\\
 & \leq\left(d+1\right)\exp\left(-\min\left\{ \frac{\tau^{2}}{8d},\frac{\tau}{8}\right\} \right)\leq\varepsilon
\end{align*}
as long as we choose 
\[
\tau\triangleq\max\left\{ 8\log\frac{d+1}{\varepsilon},\sqrt{8d\log\frac{d+1}{\varepsilon}}\right\} .
\]
Since the push forward mapping $v_{t}(x)$ is always pointing towards
zero (cf.~ \eqref{eq:Wass-convergence-inter-5} and \eqref{eq:Wass-convergence-inter-6}),
we know that $\Vert x_{t}\Vert_{2}$ is non-increasing in $t$. Therefore
we have 
\begin{align}
\mathbb{E}\left[\left\Vert x_{t}\right\Vert _{2}^{2}\phi\left(x_{t}\right)\right] & \geq\mathbb{E}\left[\left\Vert x_{t}\right\Vert _{2}^{2}\phi\left(x_{t}\right)\ind\left\{ \left\Vert x_{0}\right\Vert _{2}^{2}\leq d+\tau\right\} \right]\nonumber \\
 & \overset{\text{(i)}}{\geq}\left(2\pi\right)^{-d/2}\exp\left(-\frac{d+\tau}{2}\right)\mathbb{E}\left[\left\Vert x_{t}\right\Vert _{2}^{2}\ind\left\{ \left\Vert x_{0}\right\Vert _{2}^{2}\leq d+\tau\right\} \right]\nonumber \\
 & \geq\left(2\pi\right)^{-d/2}\exp\left(-\frac{d+\tau}{2}\right)\left(\mathbb{E}\left[\left\Vert x_{t}\right\Vert _{2}^{2}\right]-\mathbb{E}\left[\left\Vert x_{t}\right\Vert _{2}^{2}\ind\left\{ \left\Vert x_{0}\right\Vert _{2}^{2}>d+\tau\right\} \right]\right)\nonumber \\
 & \overset{\text{(ii)}}{\geq}\left(2\pi\right)^{-d/2}\exp\left(-\frac{d+\tau}{2}\right)\left(\mathbb{E}\left[\left\Vert x_{t}\right\Vert _{2}^{2}\right]-\mathbb{E}\left[\left\Vert x_{0}\right\Vert _{2}^{2}\ind\left\{ \left\Vert x_{0}\right\Vert _{2}^{2}>d+\tau\right\} \right]\right)\nonumber \\
 & \geq\left(2\pi\right)^{-d/2}\exp\left(-\frac{d+\tau}{2}\right)\left(\mathbb{E}\left[\left\Vert x_{t}\right\Vert _{2}^{2}\right]-\varepsilon\right),\label{eq:Wass-convergence-inter-8}
\end{align}
where both (i) and (ii) follows from the fact that $\Vert x_{t}\Vert_{2}$
is non-increasing. Taking \eqref{eq:Wass-convergence-inter-7} and
\eqref{eq:Wass-convergence-inter-8} collectively gives 
\begin{align*}
\partial_{t}\mathbb{E}\left[\left\Vert x_{t}\right\Vert _{2}^{2}\right] & \leq-\frac{2}{3d}\left(\frac{4\pi^{2}}{3}\right)^{d/2}\left(2\pi\right)^{-d}\exp\left[-\left(d+\tau\right)\right]\left(\mathbb{E}\left[\left\Vert x_{t}\right\Vert _{2}^{2}\right]-\varepsilon\right)^{2}\\
 & =-\frac{2}{3d}\left(\frac{1}{3}\right)^{d/2}\exp\left[-\left(d+\tau\right)\right]\left(\mathbb{E}\left[\left\Vert x_{t}\right\Vert _{2}^{2}\right]-\varepsilon\right)^{2}.
\end{align*}
Let $f(t)=\mathbb{E}[\Vert x_{t}\Vert_{2}^{2}]$, we know that $f(0)=d$
and 
\begin{align*}
\frac{\mathrm{d}f}{\mathrm{d}t} & \leq-\frac{2}{3d}\left(\frac{1}{3}\right)^{d/2}\exp\left[-\left(d+\tau\right)\right]\left(f-\varepsilon\right)^{2}.
\end{align*}
Solving this ordinary differential inequality gives 
\[
\frac{1}{f\left(t\right)-\varepsilon}-\frac{1}{f\left(0\right)-\varepsilon}\geq\frac{2}{3d}\left(\frac{1}{3}\right)^{d/2}\exp\left[-\left(d+\tau\right)\right]t,
\]
which is equivalent to 
\[
\mathbb{E}\left[\left\Vert x_{t}\right\Vert _{2}^{2}\right]\leq\varepsilon+\left\{ \frac{2}{3d}\left(\frac{1}{3}\right)^{d/2}\exp\left(-d-\tau\right)t+\frac{1}{d-\varepsilon}\right\} ^{-1}.
\]
Then we immediately know that $\mathbb{E}[\Vert x_{t}\Vert_{2}^{2}]\leq O(\varepsilon)$
as long as 
\[
t\geq\exp\left(2d\right)\varepsilon^{-1-\max\left\{ 8,\sqrt{8d}\right\} }.
\]

\subsection{Calculation for Bures-Wasserstein gradient flow \label{subsec:BW-1-Gaussian}}

Define $\ell(\mu,\Sigma)=\ell_{\infty}(\rho)$ where we parameterize
$\rho=\mathcal{N}(\mu,\Sigma)$. Then we can compute 
\begin{align*}
\ell\left(\mu,\Sigma\right) & =-\int\log\left[\left(2\pi\right)^{-d/2}\left[\det\left(\Sigma+I_{d}\right)\right]^{-1/2}\exp\left(-\frac{1}{2}\left(x-\mu\right)^{\top}\left(\Sigma+I_{d}\right)^{-1}\left(x-\mu\right)\right)\right]\phi\left(x\right)\mathrm{d}x+\mathsf{constant}\\
 & =\frac{1}{2}\log\det\left(\Sigma+I_{d}\right)+\int\frac{1}{2}\left(x-\mu\right)^{\top}\left(\Sigma+I_{d}\right)^{-1}\left(x-\mu\right)\phi\left(x\right)\mathrm{d}x+\mathsf{constant}\\
 & =\frac{1}{2}\log\det\left(\Sigma+I_{d}\right)+\frac{1}{2}\mathbb{E}_{x\sim\mathcal{N}(0,I)}\left[\left(x-\mu\right)^{\top}\left(\Sigma+I_{d}\right)^{-1}\left(x-\mu\right)\right]+\mathsf{constant}\\
 & =\frac{1}{2}\log\det\left(\Sigma+I_{d}\right)+\frac{1}{2}\mathsf{tr}\left[\left(\Sigma+I_{d}\right)^{-1}\right]+\frac{1}{2}\mu^{\top}\left(\Sigma+I_{d}\right)^{-1}\mu+\mathsf{constant}.
\end{align*}
Then we can compute the Euclidean gradient of $\ell(\mu,\Sigma)$
as follows: 
\begin{align*}
\nabla_{\mu}\ell\left(\mu,\Sigma\right) & =\left(\Sigma+I_{d}\right)^{-1}\mu,\\
\nabla_{\Sigma}\ell\left(\mu,\Sigma\right) & =\frac{1}{2}\left(\Sigma+I_{d}\right)^{-1}-\frac{1}{2}\left(\Sigma+I_{d}\right)^{-2}-\frac{1}{2}\left(\Sigma+I_{d}\right)^{-1}\mu\mu^{\top}\left(\Sigma+I_{d}\right)^{-1}\\
 & =\frac{1}{2}\left(\Sigma+I_{d}\right)^{-1}\left(\Sigma+I_{d}-I_{d}-\mu\mu^{\top}\right)\left(\Sigma+I_{d}\right)^{-1}\\
 & =\frac{1}{2}\left(\Sigma+I_{d}\right)^{-1}\left(\Sigma-\mu\mu^{\top}\right)\left(\Sigma+I_{d}\right)^{-1}.
\end{align*}
According to \citet[Appendix B.3]{lambert2022variational}, when initialized
from $(\mu_{0},\Sigma_{0})=(0,I_{d})$, the Bures-Wasserstein gradient
flow can be described using the following ODE: 
\begin{align*}
\dot{\mu}_{t} & =-\left(\Sigma_{t}+I_{d}\right)^{-1}\mu_{t}\\
\dot{\Sigma}_{t} & =-\Sigma_{t}\left(\Sigma_{t}+I_{d}\right)^{-1}\left[\Sigma_{t}-\mu\mu^{\top}\right]\left(\Sigma_{t}+I_{d}\right)^{-1}-\left(\Sigma_{t}+I_{d}\right)^{-1}\left[\Sigma_{t}-\mu\mu^{\top}\right]\left(\Sigma_{t}+I_{d}\right)^{-1}\Sigma_{t}
\end{align*}
with initial condition $\mu_{0}=0$ and $\Sigma_{0}=I_{d}$. It is
straightforward to check that $\mu_{t}=0$ for all $t\geq0$, and
the dynamic of $\Sigma_{t}$ is governed by 
\begin{align*}
\dot{\Sigma}_{t} & =-\Sigma_{t}\left(\Sigma_{t}+I_{d}\right)^{-1}\Sigma_{t}\left(\Sigma_{t}+I_{d}\right)^{-1}-\left(\Sigma_{t}+I_{d}\right)^{-1}\Sigma_{t}\left(\Sigma_{t}+I_{d}\right)^{-1}\Sigma_{t}\\
 & =-\Sigma_{t}\left(\Sigma_{t}+I_{d}\right)^{-1}+2\left(\Sigma_{t}+I_{d}\right)^{-1}\Sigma_{t}\left(\Sigma_{t}+I_{d}\right)^{-1}-\left(\Sigma_{t}+I_{d}\right)^{-1}\Sigma_{t}\\
 & =-2I_{d}+2\left(\Sigma_{t}+I_{d}\right)^{-1}+2\left(\Sigma_{t}+I_{d}\right)^{-1}\Sigma_{t}\left(\Sigma_{t}+I_{d}\right)^{-1}\\
 & =-2\left(\Sigma_{t}+I_{d}\right)^{-1}\Sigma_{t}^{2}\left(\Sigma_{t}+I_{d}\right)^{-1}
\end{align*}
with initial condition $\Sigma_{0}=I_{d}$. We can check that the
off-diagonal entries of $\Sigma_{t}$ are always zero, and its diagonal
entries are identical and evloves according to the following ODE 
\[
\dot{\sigma}_{t}=-2\frac{\sigma_{t}^{2}}{\left(\sigma_{t}+1\right)^{2}}
\]
with initial condition $\sigma_{0}=1$. It is straightforward to check
that $\sigma_{t}$ is monotonically decreasing and is always non-negative,
namely $0\leq\sigma_{t}\leq1$ always holds. Therefore we have 
\[
-2\sigma_{t}^{2}\leq\dot{\sigma}_{t}\leq-\frac{1}{2}\sigma_{t}^{2}.
\]
This gives 
\[
\frac{1}{1+2t}\leq\sigma_{t}\leq\frac{2}{2+t},
\]
and therefore 
\[
\frac{1}{1+2t}I\preceq\Sigma_{t}\preceq\frac{2}{2+t}I,
\]
which suggests that $\rho_{t}$ converges to $\rho^{\star}$ at the
speed of $O(d/t)$.

When $t=0$, we can compute the push forward mapping of Wasserstein
gradient flow explicitly, which intuitively explains why Wasserstein
gradient flow does not converge exponentially fast. We first compute
\[
\nabla\frac{\rho^{\star}\ast\phi\left(y\right)}{\rho_{0}*\phi\left(y\right)}==\nabla\frac{\left(\det I\right)^{-d/2}\exp\left(-\frac{1}{2}\left\Vert y\right\Vert _{2}^{2}\right)}{\left(\det2I\right)^{-d/2}\exp\left(-\frac{1}{4}\left\Vert y\right\Vert _{2}^{2}\right)}=2^{d/2}\nabla\exp\left(-\frac{1}{4}\left\Vert y\right\Vert _{2}^{2}\right)=-2^{d/2-1}y\exp\left(-\frac{1}{4}\left\Vert y\right\Vert _{2}^{2}\right),
\]
then the push forward mapping at $t=0$ is given by $x\mapsto v_{0}(x)$
where 
\begin{align*}
v_{0}\left(x\right) & =\int\left(\nabla_{y}\frac{\rho^{\star}\ast\phi\left(y\right)}{\rho_{0}*\phi\left(y\right)}\right)\phi\left(y-x\right)\mathrm{d}y=-2^{d/2-1}\int y\exp\left(-\frac{1}{4}\left\Vert y\right\Vert _{2}^{2}\right)\cdot\frac{1}{\left(2\pi\right)^{d/2}}\exp\left(-\frac{1}{2}\left\Vert x-y\right\Vert _{2}^{2}\right)\mathrm{d}y\\
 & =-\frac{2^{d/2-1}}{\left(2\pi\right)^{d/2}}\int y\exp\left(-\frac{1}{2}\left\Vert x\right\Vert _{2}^{2}+x^{\top}y-\frac{3}{4}\left\Vert y\right\Vert _{2}^{2}\right)\mathrm{d}y=-\frac{2^{d/2-1}}{\left(2\pi\right)^{d/2}}\int y\exp\left(-\frac{1}{6}\left\Vert x\right\Vert _{2}^{2}-\frac{3}{4}\left\Vert y-\frac{2}{3}x\right\Vert _{2}^{2}\right)\mathrm{d}y\\
 & =-2^{d/2-1}\left(\frac{2}{3}\right)^{d/2}\exp\left(-\frac{1}{6}\left\Vert x\right\Vert _{2}^{2}\right)\int\frac{1}{\left(2\pi\right)^{d/2}\left(2/3\right)^{d/2}}y\exp\left(-\frac{3}{4}\left\Vert y-\frac{2}{3}x\right\Vert _{2}^{2}\right)\mathrm{d}y\\
 & =-\frac{1}{3}\left(\frac{4}{3}\right)^{d/2}\exp\left(-\frac{1}{6}\left\Vert x\right\Vert _{2}^{2}\right)x.
\end{align*}
On the other hand, in view of \citet[Appendix B.3]{lambert2022variational},
the Bures-Wasserstein gradient at time $t=0$ is given by 
\[
\nabla_{\mathsf{BW}}\ell_{\infty}\left(\rho_{0}\right)=\left[\begin{array}{c}
\nabla_{\mu}\ell\left(\mu_{0},\Sigma_{0}\right)\\
2\nabla_{\Sigma}\ell\left(\mu_{0},\Sigma_{0}\right)
\end{array}\right]=\left[\begin{array}{c}
0\\
\frac{1}{4}I_{d}
\end{array}\right],
\]
and therefore the push forward mapping at $t=0$ is given by $x\mapsto x/4$. 